\documentclass{amsart}

\usepackage{amsmath,amscd,amsfonts,amsthm,amssymb,amscd,mathrsfs}
\usepackage{appendix}
\usepackage{bm}
\usepackage{cite}
\usepackage{geometry}
\usepackage{graphicx}
\usepackage[colorlinks,linkcolor=blue,citecolor=red]{hyperref}
\hypersetup{CJKbookmarks}
\usepackage{setspace}
\usepackage{stmaryrd}
\usepackage{url}

\allowdisplaybreaks[4]

\newtheorem{definition}{Definition}[section]

\newtheorem{theorem}{Theorem}[section]
\newtheorem{lemma}[theorem]{Lemma}
\newtheorem{proposition}{Proposition}[section]
\newtheorem{corollary}{Corollary}[theorem]

\theoremstyle{definition}
\newtheorem{remark}{Remark}[section]
\newtheorem{notation}{Notation}[section]
\newtheorem{task}{Task}[section]

\numberwithin{equation}{section}

\DeclareMathOperator{\Tr}{Tr} 
\newcommand{\HS[1]}{\left\llbracket {#1} \right\rrbracket}
\newcommand{\Ft}{\widehat}
\newcommand{\Gp}{\mathbb{G}}
\newcommand{\DuGp}{\Ft{\mathbb{G}}}
\newcommand{\SU}{\mathrm{SU}}
\newcommand{\Op}{\mathbf{Op}}
\newcommand{\size[1]}{\langle{#1}\rangle}
\newcommand{\dist}{\mathrm{dist}}
\newcommand{\Hh[1]}{\mathcal{H}_{#1}}
\newcommand{\I[1]}{\mathrm{Id}_{\mathcal{H}_{#1}}}
\newcommand{\Df}{\mathbf{D}}
\newcommand{\Pa[1]}{\partial_{#1}}
\newcommand{\hN}{\frac{1}{2}\mathbb{N}_0}
\newcommand{\Ind}{\{+,-,0\}}

\author{Chengyang Shao}
\title{On the Cauchy Problem of Spherical Capillary Water Waves}

\begin{document}
\maketitle
\begin{abstract}
The spherical capillary water waves equation describes the motion of an almost spherical water droplet under zero gravity governed by water-air interface tension. Using para-differential calculus on compact Lie groups and homogeneous spaces developed by the author, the system is symmetrized into a quasi-linear dispersive para-differential equation of order 1.5 defined on the 2-sphere. An immediate consequence of this symmetrization is a new proof of local well-posedness of the system under much weaker regularity assumption compared to previous results.
\end{abstract}
\begin{spacing}{1.2}

\section{Introduction}
This is the sequel of the author's previous paper \cite{Shao2023Toolbox}, being the second split from \cite{Shao2023}. In this paper, we reduce the spherical capillary water waves equation into a para-differential form, thus giving a new proof of the local well-posedness of the Cauchy problem. This reduction sets stage for the study of long-time behaviour of the system.

\subsection{Equation of Motion for a Water Drop}
We are interested in the initial value problem for the motion of a water droplet under zero gravity, which is the starting point of a program proposed by the author in \cite{Shao2022}. Let us first describe the physical scenario. We pose the following assumptions on the fluid motion we aim to describe:
\begin{itemize}
    \item (A1) The perfect, incompressible, irrotational fluid of constant density $\rho_0$ occupies a smooth, compact region in $\mathbb{R}^3$.
    \item (A2) There is no gravity or any other external force in presence.
    \item (A3) The air-fluid interface is governed by the Young-Laplace law, and the effect of air flow is neglected.
\end{itemize}
Hydrodynamics of water droplet governed by (A1)-(A3) is a long-standing interest for hydrophysicists and astronautical engineers. To mention a few, hydrophysicists Tsamopoulos-Brown \cite{TB1983}, Natarajan-Brown \cite{NB1986} and Lundgren-Mansour \cite{LM1988} all carried out initial ``weakly nonlinear analysis" towards the fluid motion satisfying assumptions (A1)-(A3). There are also numbers of visual materials on such experiments conducted in spacecrafts by astronauts\footnote{See for example \url{https://www.youtube.com/watch?v=H_qPWZbxFl8&t} or \url{https://www.youtube.com/watch?v=e6Faq1AmISI&t}.}. In this paper, we aim to set stage for mathematical study of this problem.

We assume that the boundary of the fluid region has the topological type of a smooth compact orientable surface $M$, and is described by a time-dependent embedding $\iota(t,\cdot):M\to\mathbb{R}^3$. We will denote a point on $M$ by $x$, the image of $M$ under $\iota(t,\cdot)$ by $M_t$, and the region enclosed by $M_t$ by $\Omega_t$. The unit outer normal will be denoted by $N(\iota)$. We also write $\bar\nabla$ for the flat connection on $\mathbb{R}^3$.

Adopting assumption (A3), we have the Young-Laplace equation:
$$
\sigma_0 H_\iota=p_i-p_e,
$$
where $H_\iota$ is the (scalar) mean curvature of the embedding, $\sigma_0$ is the surface tension coefficient (which is assumed to be a constant), and $p_i,p_e$ are respectively the inner and exterior air pressure at the boundary; they are scalar functions on the boundary and we assume that $p_e$ is a constant. Under assumptions (A1) and (A2), we obtain Bernoulli's equation, sometimes referred as the pressure balance condition, on the evolving surface:
\begin{equation}\label{BernEq}
\left.\frac{\partial\Phi}{\partial t}\right|_{M_t}+\frac{1}{2}|\bar\nabla\Phi|_{M_t}|^2-p_e=\frac{\sigma_0}{\rho_0}H_\iota,
\end{equation}
where $\Phi$ is the velocity potential of the velocity field of the air. Note that $\Phi$ is determined up to a function in $t$, so we shall leave the external (constant) pressure $p_e$ around for convenience reason. According to assumption (A1), the function $\Phi$ is a harmonic function within the region $\Omega_t$, so it is uniquely determined by its boundary value, and the velocity field within $\Omega_t$ is $\bar\nabla\Phi$. The kinetic equation on the free boundary $M_t$ is naturally obtained as
\begin{equation}\label{VelEq}
\frac{\partial\iota}{\partial t}\cdot N(\iota)
=\bar\nabla\Phi|_{M_t}\cdot N(\iota).
\end{equation}

We would like to discuss the conservation laws for (\ref{BernEq})-(\ref{VelEq}). The conservation of volume $\text{Vol}(\Omega_t)=\text{Vol}(\Omega_0)$ is a consequence of incompressibility. Since the flow is Eulerian without any external force, the center of mass must move at a uniform speed along a fixed direction, i.e.
\begin{equation}\label{CenterofMass}
\frac{1}{\mathrm{Vol}(\Omega_0)}\int_{\Omega_t}Pd\mathrm{Vol}(P)=V_0t+C_0,
\end{equation}
with Vol being the Lebesgue measure, $P$ marking points in $\mathbb{R}^3$, $V_0$ and $C_0$ being the velocity and starting position of center of mass respectively. Furthermore, the total momentum is conserved, and since the flow is a potential incompressible one, the conservation of total momentum is expressed as
\begin{equation}\label{ConsMomentum}
\int_{M_t}\rho_0\Phi N(\iota)d\mathrm{Area}(M_t)\equiv\rho_0\mathrm{Vol}(\Omega_0) V_0.
\end{equation}
Most importantly, as Zakharov pointed out in \cite{Zakharov1968}, (\ref{BernEq})-(\ref{VelEq}) is a Hamilton system, with Hamiltonian
\begin{equation}\label{Hamiltonian}
\sigma_0\mathrm{Area}(\iota)+\frac{1}{2}\int_{\Omega_t}\rho_0|\bar\nabla\Phi|^2d\mathrm{Vol}
=\sigma_0\mathrm{Area}(M_t)+\frac{1}{2}\int_{M_t}\rho_0\Phi|_{M_t}\left(\bar\nabla\Phi|_{M_t}\cdot N(\iota)\right) d\mathrm{Area},
\end{equation}
i.e. potential energy proportional to surface area plus kinetic energy of the fluid.

We explain why the scenario of oscillating almost spherical water drop is of particular interest. If the system is static, then the kinetic equation (\ref{VelEq}) implies that the outer normal derivative of the velocity potential $\Phi$ is zero, so $\Phi\equiv\text{const}$. The Bernoulli equation on boundary (\ref{BernEq}) then implies that the embedded surface $\iota(M)$ has constant mean curvature. Although the topology of $M$ is not prescribed, the famous rigidity theorem of Alexandrov asserts that $\iota(M)$ must be an Euclidean sphere (see for example \cite{MP2019}). Thus, Euclidean ball is the \emph{only} static configuration for the physical system, and oscillating almost spherical water drop is the \emph{only} possible perturbative oscillation near a static configuration. We emphasize that surface tension plays an important role: since there is no gravity in presence, it is the only constraining force preventing the fluid from spreading in the space. This is in contrast to the familiar gravity or gravity-capillary waves oscillating near the horizontal level.

\subsection{Spherical Capillary Water Waves Equation}
It is not hard to verify that the system (\ref{BernEq})-(\ref{VelEq}) is invariant if $\iota$ is composed with a diffeomorphism of $M$. We may thus regard it as a \emph{geometric flow}. If we are only interested in perturbation near a given configuration, we may reduce system (\ref{BernEq})-(\ref{VelEq}) to a non-degenerate dispersive differential system concerning two scalar functions defined on $M$, just as Beyer and G\"{u}nther did in \cite{BeGu1998}. In fact, during a short time of evolution, the interface can be represented as the graph of a function defined on the initial surface: if $\iota_0:M\to\mathbb{R}^3$ is a fixed embedding close to the initial embedding $\iota(0,x)$, we may assume that $\iota(t,x)=\iota_0(x)+\zeta(t,x)N_0(x)$, where $\zeta$ is a scalar ``height" function defined on $M_0$ and $N_0$ is the outer normal vector field of $M_0$. See Figure \ref{Height}.

With this observation, we shall transform the system (\ref{BernEq})-(\ref{VelEq}) into a non-local system of two real scalar functions $(\zeta,\phi)$ defined on $M$, where $\zeta$ is the ``height" function described as above, and $\phi(t,x)=\Phi(t,\iota(t,x))$ is the boundary value of the velocity potential, pulled back to the underlying manifold $M$.

\begin{figure}[h]\label{Height}
\centering
\includegraphics[width=0.4\textwidth,angle=0]{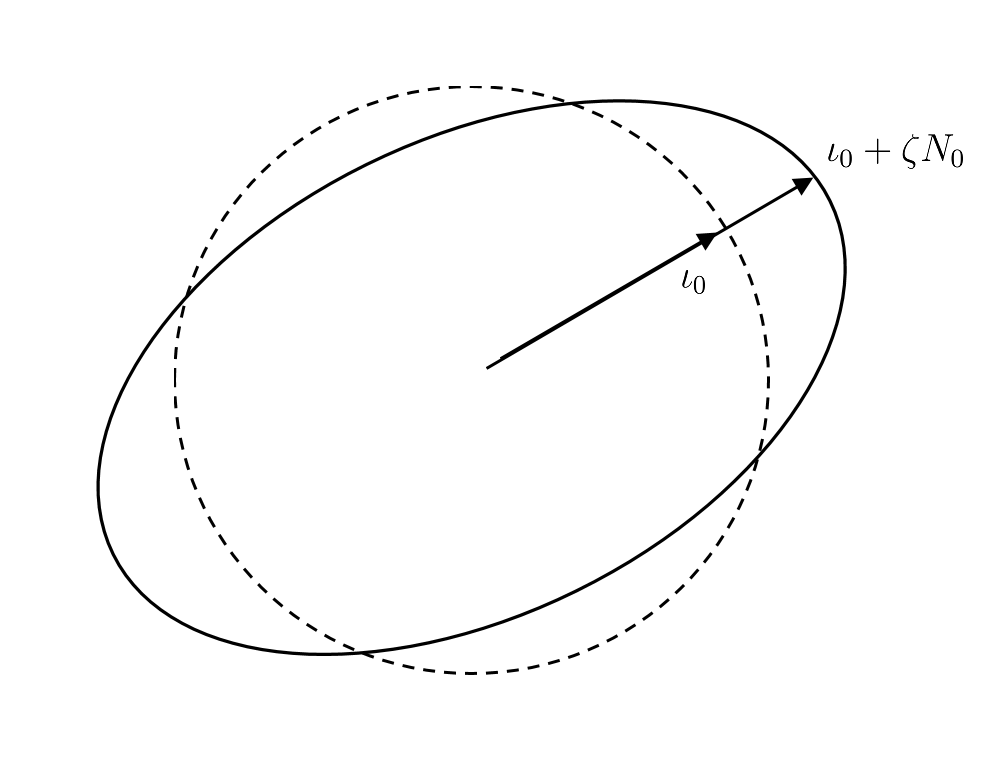}
\caption{The surface defined by a height function.}
\end{figure}

Set
$$
B_\zeta:\phi\to\bar\nabla\Phi|_{M_t}
$$
to be the operator mapping the pulled-back Dirichlet boundary value $\phi$ to the boundary value of the gradient of $\Phi$. Define the \emph{Dirichlet-Neumann operator} $D[\zeta]\phi$ corresponding to the region enclosed by the image of $\iota=\iota_0+\zeta N_0$ as the weighted outer normal derivative:
$$
D[\zeta]\phi:=\frac{\bar{\nabla}\Phi\cdot N(\iota)}{N_0\cdot N(\iota)}.
$$
Thus the kinetic equation (\ref{VelEq}) becomes
$$
\frac{\partial\zeta}{\partial t}=D[\zeta]\phi.
$$

We also need to calculate the restriction of $\partial_t\Phi$ on $M_t$ in terms of $\phi$ and $\iota$. By the chain rule,
$$
\begin{aligned}
\left.\frac{\partial\Phi}{\partial t}\right|_{M_t}
&=\frac{\partial\phi}{\partial t}-\bar\nabla\Phi|_{M_t}\cdot\frac{\partial\iota}{\partial t}\\
&=\frac{\partial\phi}{\partial t}-\left(\bar\nabla\Phi|_{M_t}\cdot N_0\right)\frac{\partial\zeta}{\partial t}\\
&=\frac{\partial\phi}{\partial t}-\left(\bar\nabla\Phi|_{M_t}\cdot N_0\right)\cdot D[\zeta]\phi.
\end{aligned}
$$
We thus arrive at the following nonlinear system:
\begin{equation}\label{EQ(M)}\tag{EQ(M)}
\left\{
\begin{aligned}
\frac{\partial\zeta}{\partial t}
&=D[\zeta]\phi,\\
\frac{\partial\phi}{\partial t}
&=\left(B_\zeta\phi\cdot N_0\right)\cdot D[\zeta]\phi-\frac{1}{2}|B_\zeta\phi|^2+\frac{\sigma_0}{\rho_0}H(\zeta)+p_e,
\end{aligned}
\right.
\end{equation}
where $H(\zeta)$ is the (scalar) mean curvature of the surface given by the height function $\zeta$.

For $M=\mathbb{S}^2$, the case that we shall discuss in detail, we name the system as \emph{spherical capillary water waves equation}. To simplify our discussion, we will be working under the center of mass frame, and require the mean of $\phi$ vanishes for all $t$. This could be easily accomplished by absorbing the mean into $\phi$ since the equation is invariant by a shift of $\phi$. The quadratic terms in the right-hand-side of (\ref{EQ(M)}) are also computed explicitly using the Riemann connection $\nabla_0$ on $\mathbb{S}^2$ corresponding to the standard spherical metric $g_0$. In a word, from now on, we will be focusing on the non-dimensional capillary spherical water waves equation
\begin{equation}\label{EQ}\tag{EQ}
\left\{
\begin{aligned}
\frac{\partial\zeta}{\partial t}
&=D[\zeta]\phi,\\
\frac{\partial\phi}{\partial t}
&=-\frac{|\nabla_0\phi|_{g_0}^2}{2(1+\zeta)^2}
+\frac{\left((1+\zeta)^2D[\zeta]\phi+\nabla_0\zeta\cdot\nabla_0\phi\right)^2}{2\big((1+\zeta)^2+|\nabla_0\zeta|_{g_0}^2\big)(1+\zeta)^2}
+\frac{\sigma_0}{\rho_0}H(\zeta)+p_e.
\end{aligned}
\right.
\end{equation}
By suitable spatial scaling, we may assume that $\sigma_0/\rho_0=1$, the total volume of the fluid is $4\pi/3$, and $p_e=-2$, so that the conservation of volume is expressed as
\begin{equation}\label{Vol}
\frac{1}{3}\int_{\mathbb{S}^2}(1+\zeta)^3d\mu_0\equiv\frac{4\pi}{3},
\end{equation}
where $\mu_0$ is the standard measure on $\mathbb{S}^2$. The inertial movement of center of mass (\ref{CenterofMass}) and conservation of total momentum (\ref{ConsMomentum}) under our center of mass frame are expressed respectively as
\begin{equation}\label{SphereCons.}
\int_{\mathbb{S}^2}(1+\zeta)^4N_0d\mu_0=0,
\quad
\int_{\mathbb{S}^2}\phi N(\iota)d\mu(\iota)=0,
\end{equation}
where $\mu(\iota)$ is the induced surface measure. Further, the Hamiltonian of the system is
\begin{equation}\label{SHamilton}
\mathbf{H}[\zeta,\phi]=\mathrm{Area}(\iota)+\frac{1}{2}\int_{\mathbb{S}^2}\frac{\phi\cdot D[\zeta]\phi}{N_0\cdot N(\iota)} d\mu(\iota),
\end{equation}
and for a solution $(\zeta,\phi)$ there holds $\mathbf{H}[\zeta,\phi]\equiv\text{Const.}$.

The system (\ref{EQ}) resembles the well-known Zakharov-Craig-Schanz-Sulem formulation of gravity-capillary water waves equation \cite{CSS1992}. Formal linearization of (\ref{EQ}) around the (unique) static solution $(\zeta,\phi)=(0,0)$ indicates that the system is a dispersive one of order 3/2. In fact, if we denote by $\mathcal{E}_n$ the space of degree $n$ spherical harmonics, then the Dirichlet-Neumann operator $D[0]$ acts as the multiplier $n$ on $\mathcal{E}_n$, and the linearization $H'(0)$ acts as the multiplier $-(n-2)(n+1)$ on $\mathcal{E}_n$. Consequently, with $\Pi^{(n)}$ being the projection to $\mathcal{E}_n$, we can introduce the diagonal unknown 
\begin{equation}\label{Formal1}
u=\Pi^{(0)}\zeta+\Pi^{(1)}\zeta
+\sum_{n\geq2}\sqrt{(n-1)(n+2)}\cdot\Pi^{(n)}\zeta+i\sum_{n\geq1}\sqrt{n}\cdot\Pi^{(n)}\phi,
\end{equation}
so that the linearization (\ref{EQ}) around the static solution is a linear dispersive equation
\begin{equation}
\frac{\partial u}{\partial t}+i\Lambda u=0,
\end{equation}
where the $3/2$-order elliptic operator $\Lambda$ is given by a multiplier
\begin{equation}\label{DispersiveRelation}
\Lambda=\sum_{n\geq2}\sqrt{n(n-1)(n+2)}\cdot \Pi^{(n)}=:\sum_{n\geq0}\Lambda(n)\Pi^{(n)}.
\end{equation}
The system (\ref{EQ}) for small amplitude oscillation can be formally re-written as
\begin{equation}\label{Formal4}
\frac{\partial u}{\partial t}+i\Lambda u=\mathfrak{N}(u),
\end{equation}
with $\mathfrak{N}(u)=O(u^{\otimes2})$ vanishing quadratically as $u\simeq0$. In \cite{Shao2022}, the author provided formal analysis regarding small amplitude solutions and pointed out obstructions in understanding the long time behaviour.

From a hydrophysical point of view, the linearized equation of (\ref{EQ}) and the dispersive relation (\ref{DispersiveRelation}) dates back to Lord Rayleigh \cite{Rayleigh1879} (see also the well-known textbook \cite{Lamb1932}, Section 275). It has been studied by several widely cited papers on hydrodynamics \cite{TB1983} \cite{NB1986} \cite{LM1988}, regarding the following topics: Poincar\'{e}-Lindstedt series of periodic solutions (assuming its existence, which is not guaranteed), possible ``chaotic" behaviour, and numerical simulation.

From a mathematical point of view, it is already known that the general free-boundary problem for Euler equation is locally well-posed, due to the work of Coutand-Shkoller \cite{CoSh2007} and Shatah-Zeng \cite{ShZe2008}. The curl equation ensures that if the flow is curl free at the beginning, then it remains so during the evolution. This justifies the local well-posedness of Cauchy problem of (\ref{EQ}), at least for sufficiently regular initial data. On the other hand, Beyer-G\"{u}nther \cite{BeGu1998} also showed that the Cauchy problem of (\ref{EQ}) is locally well-posed, without referring to general free-boundary problem. They transformed the problem into an ODE problem defined on graded Hilbert spaces. See also \cite{Lannes2005} and \cite{MingZhang2009} for proof of local well-posedness for water waves near horizontal level using Nash-Moser type theorems. The potential-theoretic approach of Wu \cite{Wu1997}-\cite{Wu1999} may also be transplanted to this case.

In summary, the Cauchy problem for (\ref{EQ}) is locally well-posed, at least for sufficiently regular initial data. But this is all we can assert for the motion of a water droplet under zero gravity. None of the mathematical results mentioned above goes beyond the time regime guaranteed by energy estimate. 

By a formal analysis conducted in \cite{Shao2022}, the author observed that the geometry of the sphere plays a fundamental role in the long-time behaviour of (\ref{EQ}), and leads to very different oscillatory behaviour compared to gravity or gravity-capillary water waves near horizontal level. To analyze such oscillations in detail, we need to take into account the spectral properties of the sphere. This calls for the development of new analytic tools that explicitly reflects the geometry of the underlying space.

\subsection{Main Result and Discussion}
The main result of this paper is a sequence of para-differential formulas based on the toolbox developed by the author in \cite{Shao2023Toolbox}. The toolbox is briefed in Section \ref{2}, collecting precise definition of symbols, para-differential operators and a symbolic calculus. The first main result of this paper is the para-linearization formula for the Dirichlet-Neumann operator $D[\zeta]\phi$:
\begin{theorem}\label{Thm1}
Fix $s>3$. Suppose $\zeta\in H^{s+0.5}$, $\phi\in H^s$. Let $\nabla_0$ be the standard connection on $\mathbb{S}^2$, and let $\nabla_{G_0}$ be the standard connection on $\SU(2)$. Introduce the quantities
$$
\mathfrak{b}=\left(1+\frac{|\nabla_{0}\zeta|^2}{(1+\zeta)^{2}}\right)^{-1}\left(
D[\zeta]\phi+\frac{\nabla_{0}\zeta\cdot\nabla_{0}\phi}{(1+\zeta)^{2}}\right),
\quad
\mathfrak{v}=\frac{\nabla_0\phi-\mathfrak{b}\nabla_0\zeta}{(1+\zeta)^2},
$$
so that $\mathfrak{b}$ is the velocity of the boundary along the radial direction, $\mathfrak{v}$ is the velocity of the boundary along the tangential direction of the unperturbed sphere. The lift of $D[\zeta]\phi$ to $\SU(2)$ is para-linearized as follows:
$$
\big(D[\zeta]\phi\big)^\sharp
=T_{\lambda}\big(\phi^\sharp-T_{\mathfrak{b}^\sharp}\zeta^\sharp\big)
-T_{\mathfrak{v}^\sharp}\cdot\nabla_{G_0}\zeta^\sharp
+f_1(\zeta,\phi).
$$
The definition of lifting to $\SU(2)$ is given by (\ref{Fsharp}). The symbol $\lambda$ is defined in  (\ref{DNlambda}), satisfying
$$
T_\lambda\simeq |\nabla_{G_0}|.
$$
The error term $f_1(\zeta,\phi)$ projects to a function on $\mathbb{S}^2$ with $H^{s+0.5}$ norm controlled by $\big(\|\zeta\|_{H^{s+0.5}}+\|\phi\|_{H^s}\big)^2$ when $(\zeta,\phi)\simeq0$.
\end{theorem}

Section \ref{4} is devoted to the proof of Theorem \ref{Thm1}. The idea of the proof is quite straightforward: one factorizes the Laplacian as the product of two first order elliptic operators near the boundary, and extract the boundary value. Using this para-linearization formula, the system (\ref{EQ}) is converted to a quasi-linear para-differential system of order 1.5, resembling the formal linearization (\ref{Formal1})-(\ref{Formal4}):

\begin{theorem}\label{Thm2}
Suppose $s>3$, and $\zeta\in C^0_tH_x^{s+0.5}$, $\phi\in C^0_tH^s_x$, so that $\|\zeta\|_{H^{s+0.5}_x}$ remains small. Keep the notation of Theorem \ref{Thm1}, and introduce the ``good unknown" $w^\sharp=\phi^\sharp-T_{\mathfrak{b}^\sharp}\zeta^\sharp$. The spherical capillary water waves system (\ref{EQ}) is then equivalent to a quasi-linear para-differential system on $\SU(2)$:
\begin{equation}\label{EQSymm}
\partial_t\left(\begin{matrix}
T_p\zeta^\sharp \\ T_q w^\sharp \end{matrix}\right)
=\left(\begin{matrix}
 & T_\gamma \\
-T_\gamma & 
\end{matrix}\right)\left(\begin{matrix}
T_p\zeta^\sharp \\ T_q w^\sharp \end{matrix}\right)
-T_{\mathfrak{v}^\sharp}\cdot\nabla_{G_0}\left(\begin{matrix}
T_p\zeta^\sharp \\ T_q w^\sharp \end{matrix}\right)
+f_3(\zeta,\phi).
\end{equation}
Here 
\begin{itemize}
\item The para-differential operator $T_p\simeq|\nabla_{G_0}|^{0.5}$, and $q$ is a scalar function;
\item The ``dispersive relation" $T_\gamma\simeq|\nabla_{G_0}|^{1.5}$ is almost self-adjoint:
$$
T_\gamma-T_\gamma^*\text{ has order }0;
$$
\item The ``transporting operator" $T_{\mathfrak{v}^\sharp}\cdot\nabla_{G_0}$ has purely imaginary principal symbol, so taking into account the almost self-adjointness of $T_\gamma$, (\ref{EQSymm}) is of dispersive type;
\item The error $f_3$ is quadratic: when $(\zeta,\phi)\simeq0$ in $H^{s+0.5}_x\times H^s_x$, $\|f_3(\zeta,\phi)\|_{H^s_x}\lesssim\left(\|\zeta\|_{H_x^{s+0.5}}+\|\phi\|_{H_x^s}\right)^2$.
\end{itemize}
\end{theorem}

Proof of Theorem \ref{Thm2} occupies Section \ref{5}. An initial application of Theorem \ref{Thm2} is the following local well-posedness theorem for (\ref{EQ}):
\begin{theorem}\label{Thm3}
Fix $s>3$. Suppose the initial data for the system (\ref{EQ}) satisfies $(\zeta(0),\phi(0))\in H^{s+0.5}\times H^s$, sufficiently small in norm. Then there is a time interval $[0,T]$, on which the Cauchy problem of (\ref{EQ}) admits a unique solution
$$
(\zeta,\phi)\in \left(C_TH^{s+0.5}_x\times C_TH^{s}_x\right)
\cap \left(C^1_TH^{s-1}_x\times C_TH^{s-1.5}_x\right).
$$
\end{theorem}
Proof of Theorem \ref{Thm3} is only sketched in Section \ref{6}, as the key computations for this are quite parallel to the counterparts in \cite{ABZ2011}; we confine ourselves elaborating what is different. Theorem \ref{Thm3} shares the same regime of regularity as the main result in \cite{ABZ2011} does for nearly horizontal water waves. The requirement on regularity is obviously milder than that in \cite{BeGu1998}, which is $(\zeta,\phi)\in H^{5.5}\times H^{5}$; or that in \cite{CoSh2007}, which is $(\zeta,\phi)\in H^{6.5}\times H^{6}$; or that in \cite{ShZe2008}, which is $(\zeta,\phi)\in H^7\times H^{6.5}$. However, this should not be surprising, as the results in \cite{ShZe2008} and \cite{CoSh2007} apply to \emph{general} motion of perfect fluid, without assuming irrotationality, while the argument in \cite{BeGu1998} deals with ODE on graded spaces so that some information will certainly be lost.

Let us briefly describe previous attempts towards para-differential calculus on curved manifolds. In \cite{KR2006}, Klainerman and Rodnianski extended the Littlewood-Paley theory to compact surfaces via smooth cut-off of the Laplacian spectrum. In \cite{Delort2015}, Delort introduced a para-differential calculus on general compact Riemann manifold characterized by commutators. In \cite{BGdP2021}, Bonthonneau, Guillarmou and de Poyferr{\'e} developed a para-differential calculus for any smooth manifold via local coordinate chart. 

However, we proceed within a formalism different from all above. In this paper, we extensively utilize the symmetry of the 2-sphere, on which (\ref{EQ}) is defined. The 2-sphere is considered as a $\SU(2)$-homogeneous space via Hopf fiberation, and the Cauchy problem of (\ref{EQ}) is lifted to $\SU(2)$, a compact Lie group. We then employ the para-differential calculus for rough symbols on compact Lie groups proposed in \cite{Shao2023Toolbox}. The advantage of this approach is that the lower order symbols, being non-neglible for quasi-linear problems of order $>1$, are manipulated in an explicit, global, coordinate-independent way. We present this toolbox in Section \ref{2} and Section \ref{3}. The theory turns out to be parallel to para-differential calculus on $\mathbb{R}^n$ or $\mathbb{T}^n$. 

With para-differential calculus constructed, the para-linearization and symmetrization of (\ref{EQ}) becomes parallel to the gravity-capillary water waves problem for the flat water level. This procedure is accomplished in Section \ref{4}-\ref{5}, with method very similar to \cite{ABZ2011}. The system is then reduced to a quasi-linear, non-local, dispersive para-differential equation of order 3/2, with precise form given in formula (\ref{EQSymm}). The local existence result, sketched in Section \ref{6}, then follows from a standard fixed point argument. 

The complexity of this para-differential reduction raises the natural question of whether it is necessary for the study of the water drop problem at all. Given the previous local well-posedness results already mentioned, such reduction does not seem unavoidable, if one is concerned with the local Cauchy theory. For example, if one employs Delort's coordinate-free definition \cite{Delort2015} of para-differential operators on manifolds, it is expected that the same regime of regularity remains valid for the local theory for initial fluid-air interface with higher genus. Anyway, it would be hard to imagine that the regime of short-time regularity depends on the topology of the interface.

However, previous approaches of para-differential calculus may fail to incorporate spectral properties of the sphere, and thus may cause extra difficulty in understanding long-time solutions. For example, no symbolic calculus is developed in \cite{KR2006} at all, since no additional structure beyond a compact surface is assumed. Both constructions in \cite{Delort2015} and \cite{BGdP2021} result in \emph{calculus involving principal symbol alone}, similar to H\"{o}rmander calculus on curved manifolds. These approaches thus may not capture the lower order symbols in quasi-linear problems that are crucial for long-time evolution. Furthermore, much of the spectral information related to dispersive properties become implicit, although not lost, under the frameworks of them. 

It then seems feasible to utilize the representation-theory-based approach to para-differential calculus, since it clearly reflects the spectral properties of the dispersive system (\ref{EQ}), while giving all the lower-order terms in symbolic calculus very explicit, manipulable form. Evidence of this convenience is obvious, if one takes into account the extensive volume of literature on gravity-capillary water waves with periodic boundary condition. People usually reduce the gravity-capillary water waves equation to a para-differential form, and use it to obtain global or almost global results. See for example \cite{DIPP2017}, \cite{BD2018}, \cite{IoPu2019}, or the review \cite{IP2018}. Given that (\ref{EQSymm}) is quite similar to the para-differential form of gravity-capillary water waves equation, we regard it as the starting point for the study of long-time behaviour of small amplitude solutions of (\ref{EQ}).

\section{Toolbox of Para-differential Calculus}\label{2}
In this section, we list the results of para-differential calculus on a compact Lie group that are necessary for our study. We shall treat the collection of these results as a toolbox. Detailed explanation and proof can be found in \cite{Shao2023Toolbox}.

\subsection{Representation and Spectral Properties}
Let $\Gp$ be a simply connected, compact, semisimple Lie group with dimension $n$ and rank $\varrho$, i.e. the dimension of any maximal torus in $\Gp$. We write $e$ for the identity element of $\Gp$. Let $\mathfrak{g}$ be the Lie algebra of $\Gp$, identified with the space of left-invariant vector fields on $\Gp$. If a basis $X_1,\cdots,X_n$ for $\mathfrak{g}$ is given, then for any multi-index $\alpha\in\mathbb{N}_0^n$, write 
\begin{equation}\label{NormalOrder}
X^\alpha=X_1^{\alpha_1}\cdots X_n^{\alpha_n}
\end{equation}
for the left-invariant differential operator composed by these basis vectors, in the ``normal" ordering.

Since $\Gp$ is semi-simple and compact, the Killing form
$$
\Tr\big(\mathrm{Ad}(X)\circ\mathrm{Ad}(Y)\big),\quad X,Y\in\mathfrak{g}
$$ 
on $\Gp$ is non-degenerate and negative definite, where $\mathrm{Ad}$ is the adjoint representation. A Riemann metric on $\Gp$ can thus be introduced as the opposite of the Killing form. The Laplace operator corresponding to this metric, usually called \emph{Casimir element} in representation theory, is given by
$$
\Delta=\sum_{i=1}^n X_i^2,
$$
where $\{X_i\}$ is any orthonormal basis of $\mathfrak{g}$; the operator is independent from the choice of basis.

Let $\DuGp$ be the dual of $\Gp$, i.e. the set of equivalence classes of irreducible unitary representations of $\Gp$. For simplicity, we do not distinguish between an equivalence class $\xi\in\DuGp$ and a certain unitary representation that realizes it. The ambient space of $\xi\in\DuGp$ is an Hermite space $\Hh[\xi]$, with (complex) dimension $d_\xi$. If we equip the space $\Hh[\xi]$ with an orthonormal basis containing $d_\xi$ elements, then $\xi:\Gp\to\mathrm{U}(\Hh[\xi])$ can be equivalently viewed as a unitary-matrix-valued function $(\xi_{jk})_{j,k=1}^{d_\xi}$.

We shall equip the Lie group $\Gp$ with the normalized Haar measure. For simplicity, we will denote the integration with respect to this measure by $dx$. The \emph{Fourier transform} of $f\in\mathcal{D}'(\Gp)$ is defined by 
$$
\Ft f(\xi):=\int_\Gp f(x)\xi^*(x)dx\in\mathrm{End}(\Hh[\xi]).
$$
Conversely, for every field $a(\xi)$ on $\DuGp$ such that $a(\xi)\in\mathrm{End}(\Hh[\xi])$ for all $\xi\in\DuGp$, the \emph{Fourier inversion} of $a$ is defined by
$$
\check{a}(x)=\sum_{\xi\in\DuGp} d_\xi\Tr\big(a(\xi)\cdot\xi(x)\big),
$$
as long as the right-hand-side converges at least in the sense of distribution. 

Two different norms for $a(\xi)\in\mathrm{End}(\Hh[\xi])$ will be used in this paper, one being the operator norm
$$
\|a(\xi)\|:=\sup_{v\in\Hh[\xi],v\neq0}\frac{|a(\xi)v|}{|v|},
$$
one being the Hilbert-Schmidt norm
$$
\HS[a(\xi)]:=\sqrt{\Tr(a(\xi)\cdot a^*(\xi))}.
$$
To simplify notation, when there is no risk of confusion, we shall omit the dependence of these norms on the representation $\xi$.

We fix the convolution on $\Gp$ to be \emph{right} convolution:
$$
(f*g)(x)=\int_\Gp f(y)g(y^{-1}x)dy.
$$
The only property not inherited from the usual convolution on Euclidean spaces is commutativity. Every other property, including the convolution-product duality, Young's inequality, is still valid. The only modification is that 
$$
\Ft{f*g}(\xi)=(\Ft g\cdot\Ft f)(\xi),
$$
which in general does not coincide with $(\Ft f\cdot\Ft g)(\xi)$ since the Fourier transform is now a matrix.

The Peter-Weyl theorem is fundamental for harmonic analysis on Lie groups:
\begin{theorem}[Peter-Weyl]\label{PeterWeyl}
Let $\mathcal{M}_\xi$ be the subspace in $L^2(\Gp)$ spanned by matrix entries of the representation $\xi\in\DuGp$. Then $\mathcal{M}_\xi$ is a bi-invariant subspace of $L^2(\Gp)$ of dimension $d_\xi^2$, and there is a Hilbert space decomposition
$$
L^2(\Gp)=\bigoplus_{\xi\in\DuGp}\mathcal{M}_\xi.
$$
If $f\in L^2(\Gp)$, the Fourier inversion
$$
(\Ft f)^\vee(x)=\sum_{\xi\in\DuGp} d_\xi\Tr\left(\Ft f(\xi)\cdot\xi(x)\right)
$$
converges to $f$ in the $L^2$ norm, and there holds the Plancherel identity
$$
\|f\|_{L^2(\Gp)}^2=\sum_{\xi\in\DuGp} d_\xi\HS[\Ft f(\xi)]^2.
$$
\end{theorem}

In the following, standard constructions and results in highest weight theory will be directly cited. They are already collected in Berti-Procesi \cite{BP2011}.

Let $\mathfrak{t}\subset\mathfrak{g}$ be the Lie algebra of a maximal torus of $\Gp$, and write $\mathfrak{t}^*$ for its dual. Let $\{\bm{\alpha}_i\}_{i=1}^\varrho\subset\mathfrak{t}^*$ be the set of \emph{positive simple roots} of $\mathfrak{g}$ with respect to $\mathfrak{t}$. They form a basis for $\mathfrak{t}^*$. The \emph{fundamental weights} $\bm{\varpi}_1,\cdots,\bm{\varpi}_\varrho\in\mathfrak{t}^*$is the unique set of vectors $\{\bm{\varpi}_i\}_{i=1}^\varrho$ satisfying 
$$
(\bm{\varpi}_i,\bm{\alpha}_j)=\frac{1}{2}\delta_{ij}|\bm{\alpha}_j|^2,
\quad
i,j=1,\cdots,\varrho.
$$
Here the inner product on the dual is inherited from the Killing form on $\mathfrak{g}$.

\begin{theorem}[Highest Weight Theorem]
Irreducible unitary representations of $\Gp$ are in 1-1 correspondence with the discrete cone
$$
\mathscr{L}^+:=\left\{\sum_{i=1}^\varrho n_i\bm{\varpi}_i,\, n_i\in\mathbb{N}_0\right\}.
$$
The correspondence assigns an irreducible unitary representation to its highest weight vector.
\end{theorem}

The weight lattice compeletely describes the eigenvalues of $\Delta$.

\begin{theorem}\label{DeltaSpec}
Each space $\mathcal{M}_\xi$ is an eigenspace of the Laplace operator $\Delta$, with eigenvalue
$$
-|\bm{J}(\xi)+\bm{\varpi}|^2+|\bm{\varpi}|^2,
$$
where $\bm{J}(\xi)$ is the highest weight vector corresponding to $\xi$, and $\bm{\varpi}=\sum_{i=1}^\varrho \bm{\varpi}_i$. 
\end{theorem}
We write $\lambda_\xi=|\bm{J}(\xi)+\bm{\varpi}|^2-|\bm{\varpi}|^2$, so that $\{\lambda_\xi\}$ is the spectrum of the positive self-adjoint operator $-\Delta$. Imitating commutative harmonic analysis, we set the size functions $|\xi|:=\sqrt{|\lambda_\xi|}$, $\size[\xi]:=(1+\lambda_\xi)^{1/2}$. The operator $|\nabla|$, which is the square-root of $-\Delta$ defined via spectral analysis, will be frequently used. The values of $\size[\xi]$ are exactly the eigenvalues of $(1-\Delta)^{1/2}$. For $s\in\mathbb{R}$, the Sobolev space $H^s(\Gp)$ is defined to be the subspace of $f\in\mathcal{D}'(\Gp)$ such that $\|(1-\Delta)^{s/2}f\|_{L^2(\Gp)}<\infty$. By the Peter-Weyl theorem and the above characterization of Laplacian, the Sobolev norm is computed by 
$$
\|f\|_{H^s(\Gp)}^2=\sum_{\xi\in\DuGp} d_\xi\size[\xi]^{2s}\HS[\Ft f(\xi)]^2.
$$

To proceed further, we need to introduce a notion of partial ordering on $\mathscr{L}^+$. Define the cone 
$$
\mathscr{R}^+:=\left\{\sum_{i=1}^\varrho n_i\bm{\alpha}_i,\, n_i\in\mathbb{N}_0\right\}.
$$
We then introduce the \emph{dominance order} on $\mathscr{L}^+$ as follows: for $\bm{j},\bm{k}\in\mathscr{L}^+$, we set $\bm{j}\prec \bm{k}$ if $\bm{j}\neq \bm{k}$ and $\bm{k}-\bm{j}\in\mathscr{R}^+$, and set $\bm{j}\preceq \bm{k}$ to include the possibility $\bm{j}=\bm{k}$. The relation $\prec$ defines a partial ordering on $\mathscr{L}^+$. With the aid of this root system, we obtain the following characterization of product of Laplace eigenfunctions (see, for example, \cite{Procesi2007}, Proposition 3 of page 345):

\begin{proposition}\label{SpecPrd0}
Let $u_{\bm{j}}\in\mathcal{M}_{\Xi(\bm{j})}$, $u_{\bm{k}}\in\mathcal{M}_{\Xi(\bm{k})}$ be Laplace eigenfunctions corresponding to highest weights $\bm{j},\bm{k}\in\mathscr{L}^+$ respectively. Then the product $u_{\bm{j}}u_{\bm{k}}$ is in the space
$$
\bigoplus_{\bm{l}\in\mathscr{L}^+:\,\bm{l}\preceq \bm{j}+\bm{k}}\mathcal{M}_{\Xi(\bm{l})}.
$$
\end{proposition}

An important spectral localization property for products can be deduced from this characterization, which will play a fundamental role in dyadic analysis:
\begin{corollary}[Corollary 2.3.1. \cite{Shao2023Toolbox}]\label{SpecPrd}
There is a constant $0<c\leq1$, depending only on the algebraic structure of $\Gp$, with the following property. If $u_1\in\mathcal{M}_{\xi_1}$, $u_2\in\mathcal{M}_{\xi_2}$, then the product $u_1u_2\in\bigoplus_{\xi}\mathcal{M}_\xi$, where the range of sum is for
$$
c\big||\xi_1|-|\xi_2|\big|
\leq|\xi|
\leq|\xi_1|+|\xi_2|.
$$
Roughly speaking, in the frequency space, product of two eigenfunctions ``localizes" in between the sum and difference of their frequencies.
\end{corollary}

\subsection{Symbol and Symbolic Calculus}
A global notion of symbol can be defined on Lie groups due to their high symmetry. An explicit symbolic calculus was formally constructed by a series of works of Ruzhansky, Turunen, Wirth. Fischer's work \cite{Fis2015} provided a complete study of the object. We will closely follow \cite{Fis2015}.

We fix our compact Lie group $\Gp$ as in the previous subsection. The starting point will be a Taylor's formula on $\Gp$.
\begin{proposition}[Taylor expansion]\label{TaylorGp}
Let $q=(q_1,\cdots,q_M)$ be an $M$-tuple of smooth functions on $\Gp$, all vanishing at $e\in\Gp$, such that $(\nabla q_i)_{i=1}^M$ has rank $n$. For a multi-index $\alpha\in\mathbb{N}_0^m$, set $q^\alpha:=q_1^{\alpha_1}\cdots q_M^{\alpha_M}$. Corresponding to each multi-index $\alpha$, there is a left-invariant differential operator $X_q^{(\alpha)}$ of order $|\alpha|$ on $\Gp$, such that the following Taylor's formula holds for every smooth function $f$ on $\Gp$ and every $N\in\mathbb{N}_0$:
$$
f(xy)=\sum_{|\alpha|<N}q^\alpha(y^{-1})X_q^{(\alpha)}f(x)
+R_{N}(f;x,y).
$$
Here the remainder $R_{N}(f;x,y)$ depends linearly on derivatives of $f$ of order $\leq N$, is smooth in $x,y$, and satisfies $R_{N}(f;x,y)=O\big(|f|_{C^N}\dist(y,e)^N\big)$.
\end{proposition}

Note that the differential operators $X_q^{(\alpha)}$ are defined independently for every multi-index, so the equality $X_q^{(\alpha+\beta)}=X_q^{(\alpha)}X_q^{(\beta)}$ is, in general, not valid.

A \emph{symbol} $a$ on $\Gp$ is simply a field $a$ defined on $\Gp\times\DuGp$, such that $a(x,\xi)$ is a distribution of value in $\mathrm{End}(\Hh[\xi])$ for each $\xi\in\DuGp$. If a basis for $\Hh[\xi]$ is chosen, the value $a(x,\xi)$ can be simply understood as a matrix function of size $d_\xi\times d_\xi$. We define the \emph{quantization} of a symbol $a$ formally by
\begin{equation}\label{Op(a)}
\Op(a)f(x):=\sum_{\xi\in\DuGp} d_\xi\Tr\left(a(x,\xi)\cdot\Ft f(\xi)\cdot\xi(x)\right).
\end{equation}
Conversely, if $A:C^\infty(\Gp)\to\mathcal{D}'(\Gp)$ is a continuous operator, then it is the quantization of the symbol
\begin{equation}\label{Op-Symbol}
\sigma[A](x,\xi):=\xi^*(x)\cdot(A\xi)(x).
\end{equation}
Here $A\xi$ is understood as entry-wise action. In this case the series (\ref{Op(a)}) then converges in $\mathcal{D}'(\Gp)$. The \emph{associated right convolution kernel} for $a(x,\xi)$ is defined by
\begin{equation}\label{Symbol-Ker}
\mathcal{K}(x,y):=\left(a(x,\cdot)\right)^\vee(y)
=\sum_{\xi\in\DuGp}d_\xi\Tr\left(a(x,\xi)\cdot\xi(y)\right),
\end{equation}
where the Fourier inversion is taken with respect to $\xi$. Formally, the action $\Op(a)f$ can be written as a convolution:
\begin{equation}\label{Ker-Conv}
f(\cdot)*\mathcal{K}(x,\cdot)
=\int_{\Gp}f(y)\mathcal{K}(x,y^{-1}x)dy.
\end{equation}

An intrinsic notion of difference operators acting on symbols was introduced by Fischer \cite{Fis2015}, generalizing the differential operator with respect to $\xi$ in harmonic analysis on $\mathbb{R}^n$. For any (continuous) unitary representation $(\tau,\mathcal{H}_\tau)$, Maschke's theorem ensures that $\tau=\oplus_j\xi_j$ for finitely many $\xi_j\in\DuGp$. A symbol $a(x,\xi)$ can be naturally extended to any $\tau$ by $a(x,\tau)=\oplus_ja(x,\xi_j)$, up to equivalence of representations. The definition of difference operators is then given by
\begin{definition}
Given any extended symbol $a(x,\cdot)$ and representation $\tau$, the difference operator $\Df_\tau$ gives rise to a new extended symbol in the following manner:
$$
\Df_\tau a(x,\pi):=a(x,\tau\otimes\pi)-a(x,\I[\tau]\otimes\pi).
$$
\end{definition}
For a tuple of representations $\boldsymbol{\tau}=(\tau_1,\cdots,\tau_p)$, write $\Df^{\boldsymbol{\tau}}=\Df_{\tau_1}\cdots\Df_{\tau_p}$, and $\Df^{\boldsymbol{\tau}}a(x,\xi)$ is then an endomorphism of $\mathcal{H}_{\tau_1}\otimes\cdots\otimes\mathcal{H}_{\tau_p}\otimes\Hh[\xi]$.

Corresponding to the functions $q=\{q_i\}_{i=1}^M$ as in Proposition \ref{TaylorGp}, Ruzhansky et al. \cite{RT2009} -- \cite{RTW2014} introduced the so-called \emph{admissible difference operators}:

\begin{definition}\label{RTAdm}
An $M$-tuple of smooth functions $q=(q_i)_{i=1}^M$ on $\Gp$ is said to be RT-admissible if they all vanishing at $e\in\Gp$ and $(\nabla q_i)_{i=1}^M$ has rank $n$. If in addition the only common zero of $(q_i)_{i=1}^M$ is the identity element, then the $M$-tuple is said to be strongly RT-admissible.
\end{definition}

\begin{definition}\label{Difference}
Given a function $q\in C^\infty(\Gp)$, the corresponding RT-difference operator $\Df_q$ acts on the Fourier transform of a $f\in\mathcal{D}'(\Gp)$ by
$$
\Big(\Df_q\Ft f\Big)(\xi):=\Ft{qf}(\xi).
$$
The corresponding collection of RT-difference operators corresponding to an $M$-tuple $q=(q_i)_{i=1}^M$ is the set of difference operators
$$
\Df_q^\alpha:=\Df_{q^\alpha}=\Df_{q_1}^{\alpha_1}\cdots\Df_{q_M}^{\alpha_M}.
$$
If the tuple $q$ is RT-admissible (strongly RT-admissible), the corresponding collection of RT difference operators is said to be RT-admissible (strongly RT-admissible).
\end{definition}

We write $\Df_{q,\xi}$ for the action of a difference operator on the $\xi$ variable. Formally, we have
\begin{equation}\label{Diff_qa}
(\Df_{q,\xi}a)(x,\xi)=\int_\Gp q(y)\left(\sum_{\eta\in\DuGp} d_\eta\Tr\left(a(x,\eta)\cdot\eta(y)\right)\right)\xi^*(y)dy,
\end{equation}
so $\Df_{q,\xi}$ commutes with any differential operator acting on $x$. To compare with $\mathbb{R}^n$, we simply notice that given a symbol $a(x,\xi)$ on $\mathbb{R}^n$, the Fourier inversion of $\partial_\xi a(x,\xi)$ with respect to $\xi$ is $iy a^{\vee}(x,y)$, i.e. multiplication by a polynomial. The functions $q$ on $\Gp$ then play the role of monomials on $\mathbb{R}^n$.

With the aid of difference operators, Fischer \cite{Fis2015} introduced symbol classes of interest.
\begin{definition}\label{S^mrd}
Let $m\in\mathbb{R}$, $0\leq\delta\leq\rho\leq1$. Fix a basis $X_1,\cdots,X_n$ of $\mathfrak{g}$, and define $X^\alpha$ as in Proposition \ref{NormalOrder}. The symbol class $\mathscr{S}^m_{\rho,\delta}(\Gp)$ is the set of all symbols $a(x,\xi)$, such that $a(x,\xi)$ is smooth in $x\in\Gp$, and for any tuple of representation $\boldsymbol{\tau}=(\tau_1,\cdots,\tau_p)$ and any left-invariant differential operator $X^\alpha$, there is a constant $C_{\alpha\boldsymbol{\tau}}$ such that
$$
\big\|X^\alpha\Df^{\boldsymbol{\tau}}a(x,\xi)\big\|
\leq C_{\alpha\boldsymbol{\tau}}\size[\xi]^{m-\rho p+\delta|\alpha|}.
$$
Here the norm is taken to be the operator norm of $\mathcal{H}_{\tau_1}\otimes\cdots\otimes\mathcal{H}_{\tau_p}\otimes\Hh[\xi]$.
\end{definition}
\begin{definition}\label{S^mrdAdmi}
Let $m\in\mathbb{R}$, $0\leq\delta\leq\rho\leq1$. Fix a basis $X_1,\cdots,X_n$ of $\mathfrak{g}$, and define $X^\alpha$ as in Proposition \ref{NormalOrder}. Let $q=(q_1,\cdots,q_M)$ be an RT-admissible $M$-tuple. The symbol class $\mathscr{S}^m_{\rho,\delta}(\Gp;\Df_q)$ is the set of all symbols $a(x,\xi)$, such that $a(x,\xi)$ is smooth in $x\in\Gp$, and 
$$
\big\|X^{\alpha}_x\Df_{q,\xi}^{\beta}a(x,\xi)\big\|
\leq C_{\alpha}\size[\xi]^{m+\delta|\alpha|-\rho|\beta|}.
$$
We can also introduce the norms
\begin{equation}\label{RhoDeltaNorm}
\mathbf{M}_{k,\rho;l,\delta;q}^m(a):=\sup_{|\alpha|\leq k}\sup_{|\beta|\leq l}
\sup_{x,\xi}\size[\xi]^{\rho|\beta|-m-\delta|\alpha|}
\left\|X^\alpha_x\Df_{q;\xi}^\beta a(x,\xi)\right\|.
\end{equation}
In particular, for $\rho=\delta=1$, we write $\mathbf{M}_{k,l;q}^m(a)$ for simplicity.
\end{definition}

In \cite{Fis2015}, Fischer proved that if $q=(q_i)_{i=1}^M$ is a strongly RT-admissible tuple, then the symbol class in Definition \ref{S^mrdAdmi} does not depend on the choice of $q$, and in fact gives rise to the usual H\"{o}rmander class of pseudo-differential operators.

\begin{theorem}[Fischer \cite{Fis2015}, Theorem 5.9. and Corollary 8.13.]\label{Fischer}
\hfill\par
(1) Suppose $0\leq\delta\leq\rho\leq1$. If $q=(q_i)_{i=1}^M$ is a strongly RT-admissible tuple, then a symbol $a\in \mathscr{S}^m_{\rho,\delta}$ if and only if all the norms $\mathbf{M}_{r,\rho;l,\delta;q}^m(a)$ are finite.

(2) Moreover, if $\rho>\delta$ and $\rho\geq1-\delta$, then the operator class $\Op \mathscr{S}^m_{\rho,\delta}$ coincides with the H\"{o}rmander class $\Psi^m_{\rho,\delta}$ of $(\rho,\delta)$-pseudo-differential operators defined via local charts.
\end{theorem}

A convenient choice of strongly RT-admissible tuple is necessary for calculations. From now on, we will just defined the \emph{fundamental tuple of $\Gp$} as
\begin{equation}\label{QFund}
Q:=\bigcup_{\tau:\text{ fundamental representation}}\left\{\tau_{ij}-\delta_{ij}:i,j=1,\cdots,d_\tau\right\}.
\end{equation}
It is not hard to verify that the tuple $Q$ is indeed a strongly RT-admissible tuple. 

A particular property of the fundamental tuple $Q$ deserves a specific mention: if $\tau$ is a fundamental representation of $\Gp$, then for $q_{ij}(x)=\tau_{ij}(x)-\delta_{ij}$, there holds
\begin{equation}\label{LeibnizFund}
q_{ij}(xy)=q_{ij}(x)+q_{ij}(y)+\sum_{k=1}^{d_\tau}q_{ik}(x)q_{kj}(y).
\end{equation}
Thus we have the Leibniz type property:
\begin{equation}\label{Leibniz}
\Df_{q_{ij}}(ab)
=\Df_{q_{ij}}a\cdot b+a\cdot\Df_{q_{ij}}b
+\sum_{k=1}^{d_\tau}\Df_{q_{ik}}a\cdot \Df_{q_{kj}}b.
\end{equation}

\begin{theorem}[\cite{Fis2015}, Corollary 7.9.]\label{RegCompo}
Suppose $0\leq\delta<\rho\leq1$. If $a\in\mathscr{S}^m_{\rho,\delta}$, $b\in\mathscr{S}^{m'}_{\rho,\delta}$, then the composition $\Op(a)\circ\Op(b)\in\Op\mathscr{S}^{m+m'}_{\rho,\delta}$, and in fact if $q$ is any strongly RT-admissible tuple and $X^{(\alpha)}_q$ is as in Proposition \ref{TaylorGp}, then the symbol $\sigma$ of $\Op(a)\circ\Op(b)$ satisfies
$$
\sigma(x,\xi)-\sum_{|\alpha|< N}\left(\Df_{q,\xi}^\alpha a\cdot X_{q,x}^{(\alpha)}b\right)(x,\xi)
\in\mathscr{S}^{m+m'-(\rho-\delta)N}_{\rho,\delta}.
$$
One can thus write
$$
\sigma(x,\xi)\sim\sum_{\alpha}\left(\Df_{q,\xi}^\alpha a\cdot X_{q,x}^{(\alpha)}b\right)(x,\xi).
$$
\end{theorem}
\begin{theorem}[\cite{Fis2015}, Corollary 7.6.]\label{RegAdj}
Suppose $0\leq\delta<\rho\leq1$. If $a\in\mathscr{S}^m_{\rho,\delta}$, then the adjoint $\Op(a)^*\in\Op\mathscr{S}^{m}_{\rho,\delta}$, and in fact if $q$ is any strongly RT-admissible tuple and $X^{(\alpha)}_q$ is as in Proposition \ref{TaylorGp}, then the symbol $a^{\bullet;q}$ of $\Op(a)^*$ satisfies
$$
a^{\bullet;q}(x,\xi)-\sum_{|\alpha|< N}\left(\Df_{q,\xi}^\alpha X_{q,x}^{(\alpha)}a^*\right)(x,\xi)
\in\mathscr{S}^{m+m'-(\rho-\delta)N}_{\rho,\delta}.
$$
One can thus write
$$
a^{\bullet;q}(x,\xi)\sim\sum_{\alpha}\left(\Df_{q,\xi}^\alpha X_{q,x}^{(\alpha)}a^*\right)(x,\xi).
$$
\end{theorem}

\subsection{Order of a Symbol}
Unlike the case of $\mathbb{R}^n$, symbolic calculus on the non-commutative Lie group $\Gp$ involves endomorphisms of the representation spaces, hence suffers from non-commutativity. Thus, for example, properties of the commutator\footnote{Note that this is not the commutator of pseudo-differential operators.} $(ab-ba)(x,\xi)=:[a,b](x,\xi)$ of two symbols $a(x,\xi)$ and $b(x,\xi)$ is not as clear as on $\mathbb{R}^n$ (it simply vanishes for symbols on $\mathbb{R}^n$). With the aid of Fischer's theorem, however, we are able to show that the commutator of symbols of order $m$ and $m'$ respectively ``is reduced by order 1".

We introduce a formal definition of order as follows.
\begin{definition}\label{2Order}
Let $m\in\mathbb{R}$. We say that a symbol $a(x,\xi)$ on $\Gp$, regardless of regularity in $x$, is of order $m$, if for some strongly RT-admissible tuple $q$, there always holds
$$
\sup_{x\in\Gp}\big\|\Df_q^\beta a(x,\xi)\big\|\lesssim \size[\xi]^{m-|\beta|}.
$$
\end{definition}

Thus the class of symbols of order $m$, in our convention, is the collection of symbols that ``possess best decays upon differentiation in $\xi$". It necessarily includes all the $\mathscr{S}^m_{1,\delta}(\Gp)$ with $0\leq\delta\leq1$. We also find that the commutator of two symbols, for which the order can be defined, is reduced in order by 1.
\begin{proposition}\label{2OrderComm}
Let $a,b$ be symbols on $\Gp$ of order $m$ and $m'$ and type 1 respectively. Then the commutator $[a,b]:=ab-ba$ is of order $m+m'-1$. In particular, if $\delta\in[0,1]$, $a\in\mathscr{S}^m_{1,\delta}(\Gp)$ and $b\in\mathscr{S}^{m'}_{1,\delta}(\Gp)$, then $[a,b]\in\mathscr{S}^{m+m'-1}_{1,\delta}(\Gp)$.
\end{proposition}

\subsection{Littlewood-Paley Decomposition}
The Littlewood-Paley decomposition on $\Gp$ is defined via spectral calculus. Fix an even function $\phi\in C_0^\infty(\mathbb{R})$, such that $\phi(\lambda)=1$ for $|\lambda|\leq1/2$, and $\phi(\lambda)=0$ for $|\lambda|\geq1$. Setting $\psi(\lambda)=-\lambda\phi'(\lambda)$, we obtain a continuous partition of unity
$$
1=\phi(\lambda)+\int_1^\infty\psi\Big(\frac{\lambda}{t}\Big)\frac{dt}{t}.
$$
The \emph{continuous Littlewood-Paley decomposition} of a distribution $f\in\mathcal{D}'(\Gp)$ will then be defined by 
\begin{equation}\label{LPCont}
f=\phi\big(|\nabla|\big)f+\int_1^\infty\psi_t\big(|\nabla|\big)f\frac{dt}{t},
\quad 
\psi_t(\cdot)=\psi\left(\frac{\cdot}{t}\right)
\end{equation}
We also write the \emph{partial sum operator} as
$$
\phi_T(|\nabla|)f:=\phi\left(\frac{|\nabla|}{T}\right)f
=\phi\big(|\nabla|\big)f+\int_1^T\psi_t\big(|\nabla|\big)f\frac{dt}{t}.
$$
This is the convention employed by H\"{o}rmander \cite{Hormander1997}, Chapter 9.

The Littlewood-Paley characterization of Sobolev space is obtained immediately: 
\begin{proposition}\label{LPHs}
Given $s\in\mathbb{R}$, a distribution $f$ belongs to $H^s(\Gp)$ if and only if for some non-vanishing $h\in C_0^\infty(0,\infty)$, there holds
$$
\left\|\phi\big(|\nabla|\big)f\right\|_{L^2}^2
+\int_0^\infty t^{2s-1}\left\|h_t\big(|\nabla|\big)f\right\|_{L^2_x}^2dt
<\infty,
$$
where $h_t(\lambda)=h(\lambda/t)$. The lower bound of integral does not cause singularity since $h=0$ near 0. Similarly, distribution $f$ belongs to $H^s(\Gp)$ if and only if for some non-zero $h\in C_0^\infty(0,\infty)$, there holds
$$
\left\|\phi\big(|\nabla|\big)f\right\|_{L^2}^2
+\sum_{j\geq0}2^{2sj}\left\|h_{2^j}\big(|\nabla|\big)f\right\|_{L^2}^2.
$$
The square root of either of the above quadratic forms is equivalent to $\|f\|_{H^s}$.
\end{proposition}
Similarly, the Littlewood-Paley characterization of Zygmund space $C^r_*(\Gp)$ is available:

\begin{proposition}\label{LPZyg}
For $r\in\mathbb{R}$, a distribution $f\in\mathcal{D}'(\Gp)$ is in the Zygmund class $C^r_*(\Gp)$ if and only if
$$
\sup|\phi(|\nabla|)f|+\sup_{t\geq1}t^r\left|\psi_t\big(|\nabla|\big)f\right|_{L^\infty}<\infty.
$$
This quantity is equivalent to the Zygmund space norm defined via local coordinate charts.
\end{proposition}

\subsection{(1,1) Pseudo-differential Operator}
Even in the Euclidean case, the symbol class $\mathscr{S}^m_{1,1}(\mathbb{R}^n)$ exhibits exotic properties compared to smaller classes $\mathscr{S}^m_{1,\delta}(\mathbb{R}^n)$ with $\delta<1$, and ``must remain forbidden fruit" as commented by Stein \cite{SteMur1993} (Subsection 1.2., Chapter 7). We thus cannot expect a satisfactory calculus for general $\mathscr{S}^m_{1,1}(\Gp)$ symbols. But in analogy to the Euclidean case, a series of theorems and constructions still remain valid for $\mathscr{S}^m_{1,1}(\Gp)$.

\begin{theorem}[Stein]\label{SteinTheorem}
Suppose $a\in \mathscr{S}^m_{1,1}(\Gp)$. Then for $s>0$, $\Op(a)$ is a bounded linear operator from $H^{s+m}$ to $H^s$. The effective version reads
$$
\|\Op(a)f\|_{H^s}\leq C_{s;q}\mathbf{M}^m_{[s]+1,n+2;q}(a)\|f\|_{H^{s+m}},
$$
where $q$ is a strongly RT-admissible tuple, and the norm $\mathbf{M}^m_{k,l;q}$ is defined in (\ref{RhoDeltaNorm}).
\end{theorem}

A special subclass of $\mathscr{S}^m_{1,1}$ can be introduced, namely the set of symbols satisfying the spectral condition:
\begin{definition}
Fix $\delta>0$. The subclass $\Sigma_\delta^m(\Gp)\subset \mathscr{S}^m_{1,1}(\Gp)$ consists of all symbols $a\in \mathscr{S}^m_{1,1}(\Gp)$ such that the partial Fourier transform of matrix entries of $a(x,\xi)$ satisfies
$$
\Ft{a_{ij}}(\eta,\xi)
:=\int_{\Gp}a_{ij}(x,\xi)\eta^*(x)dx=0
\quad\mathrm{if}\quad|\eta|\geq \delta\size[\xi],
\quad i,j=1,\cdots,d_\xi,
$$
when $\size[\xi]$ is large enough. This is called the spectral condition with parameter $\delta$.
\end{definition}

The subclass $\Sigma^m_\delta$ is important because the quantization enjoys boundedness property on all Sobolev spaces.

\begin{theorem}\label{Stein'}
Let $\delta\in(0,1/2)$, $m\in\mathbb{R}$. Suppose $a\in \Sigma_\delta^m(\Gp)$. Then the operator $\Op(a)$ maps $H^{s+m}(\Gp)$ to $H^{s}(\Gp)$ continuously for all $s\in\mathbb{R}$. More precisely, with $s_+=0$ for $s\leq0$ and $s_+=s$ for $s>0$, for any strongly RT-admissible tuple $q$,
$$
\|\Op(a)f\|_{H^s}\leq C_{s,\delta;q}\mathbf{M}^m_{[s_+]+1,n+2;q}(a)\|f\|_{H^{s+m}}.
$$
\end{theorem}

As an initial application, we are able to define para-products on $\Gp$ in a coordinate-free manner. Given two distributions $a$ and $u$, we define the \emph{para-product} $T_au$ by
\begin{equation}\label{T_au}
T_au
:=\int_1^\infty\left[\phi_{2^{-10}t}\big(|\nabla|\big)a\right]\cdot\psi_t\big(|\nabla|\big)u\frac{dt}{t},
\end{equation}
where $\phi$ and $\psi$ are as in the continuous Littlewood-Paley decomposition (\ref{LPCont}). Similarly as in the Euclidean case, the gap 10 is inessential. The symbol of $T_a$ is 
$$
\sigma[T_a](x,\xi)
=\int_1^\infty\left[\phi_{2^{-10}t}\big(|\nabla|\big)a(x)\right]\cdot\psi_t\big(|\xi|\big)\cdot\I[\xi]\frac{dt}{t}.
$$
\begin{proposition}
If $a\in L^\infty$, the symbol $\sigma[T_a]$ of the para-product operator $T_a$ is of class $\Sigma^0_{2^{-10}}$.
\end{proposition}

By Theorem \ref{Stein'}, the para-product operator $T_a$ is thus continuous from $H^s$ to itself for all $s\in\mathbb{R}$. Just as in the Euclidean case, we have the para-product decomposition on $\Gp$:
\begin{theorem}\label{ParaPrd}
If $s>0$, $a,u\in (L^\infty\cap H^s)(\Gp)$, then
$$
au=T_au+T_ua+R(a,u)+\phi\big(|\nabla|\big)a\cdot\phi\big(|\nabla|\big)u
$$
where the smoothing remainder 
$$
\|R(a,u)\|_{H^s}\lesssim |a|_{L^\infty}\|u\|_{H^s}.
$$
\end{theorem}
\begin{remark}\label{R(a)Symbol}
If in addition $a\in C^r_*(\Gp)$, then from the spectral characterization of Zygmund spaces, we find that in fact $R[a]\in \mathscr{S}^{-r}_{1,1}$, so 
$$
\|R(a,u)\|_{H^s}\lesssim|a|_{C^r_*}\|u\|_{H^{s-r}}.
$$
\end{remark}

We also have a direct generalization of Bony's para-linearization theorem.
\begin{theorem}[Bony]\label{Bony}
Suppose $F\in C^\infty(\mathbb{C};\mathbb{C})$ (understood as smooth mapping on the plane instead of holomorphic function), $F(0)=0$. Let $r>0$, and suppose $u\in C^r_*$. Write $u_t=\phi(|\nabla|/t)u$, where $\phi$ is the function in the Littlewood-Paley decomposition (\ref{LPCont}). Then with the symbol
$$
l_u(x,\xi):=\int_1^\infty F'(u_t(x))\cdot\psi_t\big(|\xi|\big)\cdot\I[\xi]\frac{dt}{t},
$$
there holds $F(u)=F(u_1)+\Op(l_u)u$, and we have the para-linearization formula
$$
F(u)=F(u_1)+T_{F'(u)}u+\Op(R[u])u
$$
with the symbol $R[u](x,\xi)\in \mathscr{S}^{-r}_{1,1}$. Consequently, if $u\in C^r_*\cap H^s$, then $F(u)-T_{F'(u)}u\in H^{s+r}$. In particular, for $s>n/2$, $F(u)-T_{F'(u)}u\in H^{2s-n/2}$.
\end{theorem}

\subsection{Para-differential Operators on Compact Lie Group}
Just like $\mathbb{R}^n$, para-differential calculus starts with rough symbols.

\begin{definition}\label{Rough}
For $r\geq0$ and $m\in\mathbb{R}$, define the symbol class $\mathcal{A}_r^m(\Gp)$ to be the collection of all symbols $a(x,\xi)$ on $\Gp\times\DuGp$, such that if $q=(q_i)_{i=1}^M$ is a strongly RT-admissible tuple, then
$$
\big\|\Df_{q,\xi}^\beta a(x,\xi)\big\|_{r;x}\leq C_\beta\size[\xi]^{m-|\beta|},\quad\forall\xi\in\DuGp.
$$
Here the norm $\|a(x,\xi)\|_{r;x}$ for $a(\cdot,\xi):\Gp\to\mathrm{End}(\Hh[\xi])$ is defined by
$$
\|a(x,\xi)\|_{r;x}:=
\sup_{B^*\in\mathrm{End}(\Hh[\xi])^*}\frac{\big|\langle B^*, a(x,\xi)\rangle\big|_{C^r_{*x}}}{\|B^*\|}.
$$
Here we consider $\mathrm{End}(\Hh[\xi])$ merely as a normed space instead of a normed algebra. Introduce the following norm on $\mathcal{A}_r^m(\Gp)$:
$$
\mathbf{W}^{m;r}_{l;q}(a):=\sup_{\xi\in\DuGp}\sum_{|\beta|\leq l}\size[\xi]^{|\beta|-m}\big\|\Df_{q,\xi}^\beta a(x,\xi)\big\|_{r;x}
$$
\end{definition}
The definition of $\mathcal{A}^m_r(\Gp)$ does not depend on the choice of $q$, so we can choose any set of functions $q$ that facilitates our computation. In particular, the fundamental tuple $Q$ defined in (\ref{QFund}) is a convenient choice.

\begin{definition}[Admissible cut-off]\label{AdmCutoff}
An admissible cut-off function $\chi$ with parameter $\delta\in(0,1/2)$ is a smooth function on $\mathbb{R}\times\mathbb{R}$ such that
$$
\chi(\mu,\lambda)=\left\{
\begin{aligned}
1,&\quad |\mu|\leq\frac{\delta}{2}\size[\lambda],\\
0,&\quad |\mu|\geq\delta\size[\lambda],
\end{aligned}
\right.
$$
and there also holds
$$
|\partial_\mu^k\partial_\lambda^l\chi(\mu,\lambda)|
\lesssim_{k,l}\langle\lambda\rangle^{-k-l}.
$$
\end{definition}

\begin{definition}[Para-differential Operator]
Let $\chi$ be an admissible cut-off function with some parameter $\delta$. For $a\in\mathcal{A}_r^m(\Gp)$, set $a^\chi (x,\xi)=\chi(|\nabla_x|,|\xi|)a(x,\xi)$, the regularized symbol corresponding to $a$. Define the para-differential operator $T_a^\chi$ corresponding to $a\in\mathcal{A}_r^m(\Gp)$ as
$$
T_a^\chi u(x):= \Op(a^\chi)u(x).
$$
\end{definition}

If $r>0$ and $a\in\mathcal{A}^m_r$, then $a^\chi$ is a (1,1) symbol satisfying the spectral condition with parameter $\delta$. For different choices of $\chi$, the para-differential operator $T_a^\chi$ are equal modulo $\Op\Sigma^{m-r}_{<1/2}$. Thus it is safe to omit the dependence on $\chi$ and talk about \emph{the} para-differential operator.

Para-differential operators still enjoy a symbolic calculus.

\begin{definition}
Let $r>0$ be a real number. Let $q$ be a RT-admissible tuple, and $X_q^{(\alpha)}$ be the left-invariant differential operators as in Proposition \ref{TaylorGp}. If $a,b$ are any symbol of at least $C^r_*$ regularity in $x$, define
$$
(a\#_{r;q} b)(x,\xi):=\sum_{\alpha:|\alpha|\leq r}\Df_{q,\xi}^\alpha a(x,\xi)\cdot X_{q,x}^{(\alpha)} b(x,\xi).
$$
\end{definition}

\begin{theorem}[Composition of Para-differential Operators]\label{Compo2}
Suppose $r>0$, $m,m'$ are real numbers. Let $q$ be a RT-admissible tuple, whose components are are linear combinations of the fundamental tuple $Q$ of $\Gp$, and let $X_q^{(\alpha)}$ be the left-invariant differential operators as in Proposition \ref{TaylorGp}. Given $a\in \mathcal{A}_r^{m}$, $b\in\mathcal{A}_r^{m'}$, it follows that
$$
T_a\circ T_b-T_{a\#_{r,q}b}
\in\Op\Sigma_{<1/2}^{m+m'-r}.
$$
More precisely, the operator norm of $T_a\circ T_b-T_{a\#_{r,q}b}\in\mathcal{L}(H^{s+m+m'-r},H^s)$ is bounded by $C_s\mathbf{W}^{m;r}_{l;q}(a)\mathbf{W}^{m';r}_{l;q}(b)$, where the integer $l$ does not depend on $a,b$.
\end{theorem}

\begin{corollary}\label{ParaComm}
Suppose $r>0$, $m,m'$ are real numbers.  Given $a\in \mathcal{A}_r^{m}$, $b\in\mathcal{A}_r^{m'}$, the commutator $[T_a,T_b]$ is in the class $\Op\Sigma_{<1/2}^{m+m'-1}$, with operator norm in $\mathcal{L}(H^{s+m+m'-1},H^s)$ bounded by $C_s\mathbf{W}^{m;r}_{l;q}(a)\mathbf{W}^{m';r}_{l;q}(b)$, where the integer $l$ does not depend on $a,b$.
\end{corollary}

\begin{theorem}[Adjoint of Para-differential Operator]\label{ParaAdj}
Suppose $r>0$, $m$ is a real number. Let $q$ be a RT-admissible tuple, whose components are from the fundamental tuple $Q$ of $\Gp$, and let $X_q^{(\alpha)}$ be the left-invariant differential operators as in Proposition \ref{TaylorGp}. Given $a\in \mathcal{A}_r^{m}$, the adjoint operator $T_a^*$ satisfies
$$
T_a^*-T_{a^{\bullet;r,q}}\in\Op\Sigma^{m-r}_{<1/2}
$$
where
$$
a^{\bullet;r,q}(x,\xi)=\sum_{|\alpha|\leq r}\left(\Df_{q,\xi}^\alpha X_{q,x}^{(\alpha)}a^*\right)(x,\xi).
$$
More precisely, the operator norm of $T_a^*-T_{a^{\bullet;r,q}}\in\mathcal{L}(H^{s+m-r},H^s)$ is bounded by $C_s\mathbf{W}^{m;r}_{l;q}(a)$, where the integer $l$ does not depend on $a$.
\end{theorem}

Sometimes it is also necessary to discuss symbols $a(x,\xi)$ which are merely $C^{-r}_*$ in $x$ with $r>0$. We have the following proposition:
\begin{proposition}\label{T_aNegIndex}
Suppose $r>0$. Fix a basis $X_1,\cdots,X_n$ of $\mathfrak{g}$, and define $X^\alpha$ as in Proposition \ref{NormalOrder}. Suppose $a\in\mathcal{A}_{-r}^m(\Gp)$ \footnote{The definition of symbol class $\mathcal{A}_{r}^m(\Gp)$ of course can be directly extended beyond $r>0$.} and $\chi$ is an admissible cut-off function with parameter $\delta$. Then $a^\chi(x,\xi)\in\Sigma^{m+r}_{<1/2}$. In fact, for any strongly RT-admissible tuple $q$, we have
$$
\big\|X_x^\alpha\Df_{q,\xi}^\beta a^\chi (x,\xi)\big\|
\lesssim_{\alpha,\beta;q}
\mathbf{W}^{m;r}_{|\beta|;q}(a)\cdot\size[\xi]^{m+r+|\alpha|-|\beta|}.
$$
\end{proposition}
Since such symbols are quite irrregular, we cannot expect that choosing a different admissible cut-off function would produce a negeligible error. But we may still refer $\Op(a^\chi)$ as a \emph{para-differential operator} $T_a$, which maps $H^{s+r}$ to $H^{s}$ continuously for any $s\in\mathbb{R}$. Fortunately, there is a simple trick to bypass its low regularity and incorporate it into the symbolic calculus that we just constructed. For example, if $0<r<1$ and $a\in\mathcal{A}^m_{-r}$, then we consider $\Delta^{-1}a$ (subtracting the mean value to make this well-defined), and simply notice that
$$
T_af
=\big[\Delta,T_{\Delta^{-1}a}\big] f-2T_{\nabla \Delta^{-1}a\cdot\nabla}f.
$$
The right-hand-side involves only regular symbols. By the formula of compositions, noting that $\Delta^{-1}a\in\mathcal{A}^{m+2}_{2-r}$, we find
$$
\big[\Delta,T_{\Delta^{-1}a}\big]
-2T_{\nabla \Delta^{-1}a\cdot\nabla}
$$
is in fact a para-differential operator of order $m+2-(2-r)=m+r$. 

\subsection{Quasi-homogeneous Symbols}
We define a special class of symbols consisting of symbols of ``homogenized" classical differential operators. The prototype of such symbols is
$$
\sqrt{|\xi|^2+b(x,\xi)},
\quad 
b \text{ is the symbol of a vector field},
$$
i.e. the square-root of perturbed Laplacian. In the commutative group case, symbols like this may be manipulated as usual scalar-valued functions. We find that such manipulation remains partially valid on our group $\Gp$.

\begin{definition}\label{QuasiHomoSym}
A quasi-homogeneous symbol of order $m$ on $\Gp$ takes the form
$$
\kappa(\xi)^mf\left(\frac{b(x,\xi)}{\kappa(\xi)}\right).
$$
Here the scalar Fourier multiplier $\kappa:\DuGp\to(0,+\infty)$ is the symbol of $|\nabla|=\sqrt{-\Delta}$, and $b(x,\xi)$ is the symbol of a vector field on $\Gp$, and the Borel function $f:\mathbb{C}\to\mathbb{C}$ is bounded. The action $f$ on each $\mathrm{End}(\Hh[\xi])$ is defined by spectral calculus of matrices.
\end{definition}

We list some properties concerning the difference operator acting on the scalar symbol $\kappa(\xi)$. For simplicity, we fix a basis $\{X_i\}_{i=1}^n$ of $\mathfrak{g}$, so that $b(x,\xi)=\sum_{i=1}^nb_i(x)\sigma[X_i](\xi)$; and write $\{\Df_\mu\}$ for the RT-difference operators corresponding to the fundamental representations of $\Gp$. Throughout the rest of this subsection, summation with respect to Greek indices is interpreted as summation over these difference operators.

It turns out that the symbol of $|\nabla|$ enjoys a symbolic calculus which is quite similar to the Euclidean case.

\begin{proposition}\label{Dkappa}
Let $\kappa$ be the symbol of $|\nabla|$ on $\Gp$.

(1) There necessarily holds $\kappa\in\mathscr{S}^1_{1,0}$, and $\Df_\mu \kappa$ commutes with every symbol of order $m$ up to an error of order $m-1$. Here the notion of order is as in Definition \ref{2Order}.

(2) For every real number $m\in\mathbb{R}$, the symbol $\Df_\mu\kappa^m-m\kappa^{m-1}\Df_\mu\kappa\in\mathscr{S}^{m-2}_{1,0}$.
\end{proposition}

The advantage of introducing quasi-homogeneous symbols is that their differentiation (in the $x$ variable) or difference (in the $\xi$ variable) enjoy approximate Leibniz rules:
\begin{proposition}\label{PSSymbol}
Let $\kappa$ be the symbol of $|\nabla|$ on $\Gp$, $b$ be the symbol of a vector field on $\Gp$. Let $f$ be a bounded holomorphic function defined near $z=0$ on the complex plane, covering the closed disk of radius $R_0:=\sup_{x,\xi}|\xi|^{-1}\|b(x,\xi)\|$.

(1) The quasi-homogeneous symbol $f(b/\kappa)$ is of order 0 in the sense of Definition \ref{2Order}, and the difference $\Df_\mu f(b/\kappa)$ is such that
$$
\Df_\mu f\left(\frac{b(x,\xi)}{\kappa(\xi)}\right)
-\frac{1}{\kappa(\xi)}f'\left(\frac{b(x,\xi)}{\kappa(\xi)}\right)\Df_\mu b
-\frac{\Df_\mu\kappa(\xi)}{\kappa(\xi)^2}f'\left(\frac{b(x,\xi)}{\kappa(\xi)}\right)b(x,\xi)
$$
is a symbol of order $-2$.

(2) If $a$ is the symbol of another vector field, then the commutator $[f(b/\kappa),a]$ is still a symbol of order $0$.

(3) If $b$ has at least $C^1$ coefficients, then for any left-invariant vector field $X$, the symbol $Xf(b/\kappa)$ is of order $0$, and 
$$
Xf\left(\frac{b}{\kappa}\right)-\frac{1}{\kappa}f'\left(\frac{b}{\kappa}\right)Xb
$$
is a symbol of order $-1$.
\end{proposition}

\begin{corollary}\label{CoroPSSymbol}
Under the same assumptions of Proposition \ref{PSSymbol}, if in addition the vector field $\Op(b)$ is of class $C^r_*$ for $r>0$, then the symbol $\kappa^m f(b/\kappa)$ is of class $\mathcal{A}^m_r$.
\end{corollary}

Finally, we state a proposition regarding symbols that involve very rough coefficients.
\begin{proposition}\label{SymbolVeryRough}
Suppose $\sum_{n=0}^\infty c_nz^n$ is a convergent power series near $z=0$. Fix $r>0$. Let $s>n/2$, $s>r$. Let $a,b$ be symbols of vector fields on $\Gp$, with $a$ of merely $C^{-r}_*$ regularity in $x$ and $b$ of $H^s$ regularity in $x$, with
$$
\sup_{\xi\in\DuGp}\frac{\|b(x,\xi)\|_{H^s_x}}{|\xi|}
$$
suitably small. Define 
$$
[b]_na:=ab^n+bab^{n-1}+\cdots+b^na.
$$
Then 
$$
\sum_{n=0}^\infty c_n\kappa^{-n}[b]_na
$$
is a symbol of class $\mathcal{A}^{0}_{-r}(\Gp)$.
\end{proposition}

\section{Global Symbolic Calculus on \texorpdfstring{$\SU(2)$}{a} and \texorpdfstring{$\mathbb{S}^2$}{b}}\label{3}
We apply the general theory developed so far to construct para-differetial calculus on $\mathbb{S}^2$, the standard 2-sphere. Although $\mathbb{S}^2$ cannot be endowed with a Lie group structure, it still inherits a symbolic calculus as $\SU(2)$-homogeneous space in a simple manner. The para-differential calculus on $\mathbb{S}^2$ will be constructed in this way. We start by formulating a global symbolic calculus on $\SU(2)$, and then pass to $\mathbb{S}^2$.

\subsection{Representation Theory of \texorpdfstring{$\SU(2)$}{c}}
Let us first sketch the representation theory of $\SU(2)$. An element of $\SU(2)$ will be denoted as
$$
x=\left(\begin{matrix}
x_1 & x_2 \\
-\bar{x}_2 & \bar{x}_1
\end{matrix}\right),
\quad
|x_1|^2+|x_2|^2=1.
$$
The Lie algebra $\mathfrak{su}(2)$, being the set of all skew-Hermitian matrices with trace 0, can be identified with the tangent space at $\mathrm{Id}\in\SU(2)$ and also the Lie algebra of left-invariant vector fields. It is spanned by the real-linear combination of (scaled) Pauli matrices
$$
\frac{1}{2}\left(\begin{matrix}
0 & i \\
i & 0
\end{matrix}\right),
\quad
\frac{1}{2}\left(\begin{matrix}
0 & -1 \\
1 & 0
\end{matrix}\right),
\quad
\frac{1}{2}\left(\begin{matrix}
i & 0 \\
0 & -i
\end{matrix}\right).
$$
They correspond to the 1-dimensional subgroups
$$
\omega_1(t)=\left(\begin{matrix}
\cos\frac{t}{2} & i\sin\frac{t}{2} \\
i\sin\frac{t}{2} & \cos\frac{t}{2}
\end{matrix}\right),
\quad
\omega_2(t)=\left(\begin{matrix}
\cos\frac{t}{2} & -\sin\frac{t}{2} \\
\sin\frac{t}{2} & \cos\frac{t}{2}
\end{matrix}\right),
\quad
\omega_3(t)=\left(\begin{matrix}
e^{it/2} & 0 \\
0 & e^{-it/2}
\end{matrix}\right)
$$
respectively. We denote these subgroups as $\mathbf{T}_{1,2,3}$. 

We employ the Euler angles as parameterization of $\SU(2)$. For 
$$
\varphi\in[0,2\pi),\quad\theta\in[0,\pi],\quad\psi\in[-2\pi,2\pi),
$$
an element of $\SU(2)$ outside a lower dimensional closed subset is factorized as
\begin{equation}\label{EulerAngle}
\left(\begin{matrix}
e^{i(\varphi+\psi)/2}\cos\frac{\theta}{2} & ie^{i(\varphi-\psi)/2}\sin\frac{\theta}{2} \\
ie^{-i(\varphi-\psi)/2}\sin\frac{\theta}{2} & e^{-i(\varphi+\psi)/2}\cos\frac{\theta}{2}
\end{matrix}\right)
=\left(\begin{matrix}
e^{i\varphi/2} & 0 \\
0 & e^{-i\varphi/2}
\end{matrix}\right)
\left(\begin{matrix}
\cos\frac{\theta}{2} & i\sin\frac{\theta}{2} \\
i\sin\frac{\theta}{2} & \cos\frac{\theta}{2}
\end{matrix}\right)
\left(\begin{matrix}
e^{i\psi/2} & 0 \\
0 & e^{-i\psi/2}
\end{matrix}\right).
\end{equation}
The parameters $\varphi\in[0,2\pi),\,\theta\in(0,\pi),\,\psi\in[-2\pi,2\pi)$ are in 1-1 correspondence with elements of $\SU(2)$ excluding a lower dimensional closed subset (which corresponds to $\theta=\pm\pi$). We shall denote the element with Euler angles $\varphi,\theta,\psi$ as $\Omega(\varphi,\theta,\psi)$. The factorization (\ref{EulerAngle}) can be re-written as
\begin{equation}\label{EulerAngle1}
\Omega(\varphi,\theta,\psi)=\omega_3(\varphi)\omega_1(\theta)\omega_3(\psi)\in\mathbf{T}_3\cdot\mathbf{T}_2\cdot\mathbf{T}_3.
\end{equation}

The Killing form of $\mathfrak{su}(2)$ is $4\Tr(XY)$ for $X,Y\in\mathfrak{su}(2)$. If we set the negative of this Killing form as the inner product on $\mathfrak{su}(2)$, then the matrices $\omega_{1,2,3}'(0)$ form an orthonormal basis. Thus, the Riemann metric on $\SU(2)$ under the Euler angles is
$$
G_0=d\theta^2+d\varphi^2+2\cos\theta d\varphi d\psi+d\psi^2,
$$
or in matrix form
$$
\left(\begin{matrix}
(G_0)_{\theta\theta} & (G_0)_{\theta\varphi} & (G_0)_{\theta\psi} \\
* & (G_0)_{\varphi\varphi} & (G_0)_{\varphi\psi} \\
* & * & (G_0)_{\psi\psi}
\end{matrix}\right)
=\left(\begin{matrix}
1 & 0 & 0 \\
0 & 1 & \cos\theta \\
0 & \cos\theta & 1
\end{matrix}\right).
$$
If we identify $\SU(2)$ with the unit 3-sphere in $\mathbb{C}^2$, then $G_0$ is 4 times the inherited metric on unit 3-sphere. The Laplacian of $G_0$ is
$$
\Delta_{G_0}
=\frac{\partial^2}{\partial\theta^2}+\frac{\cos\theta}{\sin\theta}\frac{\partial}{\partial\theta}
+\frac{1}{\sin^2\theta}\frac{\partial^2}{\partial\varphi^2}
-\frac{2\cos\theta}{\sin^2\theta}\frac{\partial^2}{\partial\varphi\partial\psi}+\frac{1}{\sin^2\theta}\frac{\partial^2}{\partial\psi^2}.
$$

In Chapter 3 of \cite{Vilenkin1978}, Vilenkin explicitly computed the left-invariant vector fields on $\SU(2)$ corresponding to the subgroups $\omega_{1,2,3}(t)$. We denote them by $X_{1,2,3}$ respectively. The action of $X_j$ is 
$$
(X_jf)(x):=\frac{d}{dt}f(x\omega_j(t))\Big|_{t=0}.
$$
They are given by, respectively,
\begin{equation}\label{LISU(2)}
\begin{aligned}
X_1&=\cos\psi\frac{\partial}{\partial\theta}+\frac{\sin\psi}{\sin\theta}\frac{\partial}{\partial\varphi}-\frac{\cos\theta}{\sin\theta}\sin\psi\frac{\partial}{\partial\psi},\\
X_2&=-\sin\psi\frac{\partial}{\partial\theta}+\frac{\cos\psi}{\sin\theta}\frac{\partial}{\partial\varphi}-\frac{\cos\theta}{\sin\theta}\cos\psi\frac{\partial}{\partial\psi},\\
X_3&=\frac{\partial}{\partial\psi}.
\end{aligned}
\end{equation}
It is directly verified that these vector fields form an orthonormal basis in the tangent space of any element in $\SU(2)$, and $\Delta_{G_0}=X_1^2+X_2^2+X_3^2$. However, for symbolic calculus, it is more convenient to consider certain linear combination of the $X_j$'s since it is more convenient for the symbols to contain more 0's. We thus introduce the \emph{creation, annihilation and neutral} operators
\begin{equation}\label{CANSU(2)}
\begin{aligned}
\Pa[+]&=iX_1-X_2=e^{-i\psi}\left(i\frac{\partial}{\partial\theta}-\frac{1}{\sin\theta}\frac{\partial}{\partial\varphi}+\frac{\cos\theta}{\sin\theta}\frac{\partial}{\partial\psi}\right),\\
\Pa[-]&=iX_1+X_2=e^{i\psi}\left(i\frac{\partial}{\partial\theta}+\frac{1}{\sin\theta}\frac{\partial}{\partial\varphi}-\frac{\cos\theta}{\sin\theta}\frac{\partial}{\partial\psi}\right),\\
\Pa[0]&=iX_3=i\frac{\partial}{\partial\psi},
\end{aligned}
\end{equation}
just as in \cite{Vilenkin1978} or \cite{RT2013}. Thus 
$$
\Delta_{G_0}=-\frac{1}{2}(\Pa[-]\Pa[+]+\Pa[+]\Pa[-])-\Pa[0]^2,
$$
and the commutator relations for $\mathfrak{su}(2)$ hold:
$$
[\Pa[0],\Pa[+]]=\Pa[+],
\quad
[\Pa[-],\Pa[0]]=\Pa[-],
\quad
[\Pa[+],\Pa[-]]=2\Pa[0].
$$

The table of irreducible unitary representations of $\SU(2)$ can be found in any standard textbook on representation theory. We sketch the result as follows. Since the group $\SU(2)$ has rank 1, one can label irreducible unitary representations of $\SU(2)$ by a parameter $l\in\hN$, the set of half integers. The irreducible unitary representation with label $l$ has dimension $2l+1$. The $l$'th representation is realized as follows: setting $\Hh[l]$ to be the space of complex homogeneous polynomials in a two-dimensional variable $z\in\mathbb{C}^2$ with degree $2l$, the action of $\SU(2)$ on $\Hh[l]$ is realized as
$$
\big(T^l(x)f\big)(z):=f(xz).
$$
The Hermite structure on $\Hh[l]$ is defined by fixing an orthonormal basis
$$
\frac{z_1^{l-k}z_2^{l+k}}{\sqrt{(l-k)!(l+k)!}},
\quad
k=-l,\,-l+1,\,\cdots,\,l-1,\,l.
$$
Following \cite{Vilenkin1978}, we use $T^l=(T^l_{nm})$, $m,n=-l,\,-l+1,\,\cdots,\,l-1,\,l$, to denote the matrix elements of the representation labeled by $l$ under the basis fixed above. We fix the convention that $n$ labels the rows of $T^l$ and $m$ labels the columns of $T^l$. In Euler angles,
$$
T^l_{nm}\big(\Omega(\varphi,\theta,\psi)\big)=e^{-i(n\varphi+m\psi)}P^l_{nm}(\cos\theta),
$$
where 
$$
P^l_{nm}(z)=\frac{2^{-l}(-1)^{l-m}i^{m-n}}{\sqrt{(l-m)!(l+m)!}}\sqrt{\frac{(l+n)!}{(l-n)!}}
\cdot\frac{(1-z)^{(m-n)/2}}{(1+z)^{(m+n)/2}}\left(\frac{d}{dz}\right)^{l-n}\left((1-z)^{l-m}(1+z)^{l+m}\right).
$$
For example, $l=0$ gives the trivial representation $\SU(2)\to\{1\}$; for $l=1/2$, we obtain the fundamental representation, which is also a faithful: $T^{1/2}(\Omega(\varphi,\theta,\psi))=\Omega(\varphi,\theta,\psi)$, and the entries are
$$
\bordermatrix{%
 & m=-\frac{1}{2} & m=\frac{1}{2} \cr
n=-\frac{1}{2} & e^{i(\varphi+\psi)/2}\cos\frac{\theta}{2} & ie^{i(\varphi-\psi)/2}\sin\frac{\theta}{2} \cr
n=\frac{1}{2} & ie^{-i(\varphi-\psi)/2}\sin\frac{\theta}{2} & e^{-i(\varphi+\psi)/2}\cos\frac{\theta}{2}
};
$$
for $l=1$, the entries of $T^{1}(\Omega(\varphi,\theta,\psi))$ are
$$
\bordermatrix{%
 & m=-1 & m=0 & m=1 \cr
n=-1 & e^{i(\varphi+\psi)/2}\cos^2\frac{\theta}{2} & ie^{i\varphi}\frac{\sin\theta}{\sqrt{2}} & -e^{i(\varphi-\psi)/2}\sin^2\frac{\theta}{2} \cr
n=0 & ie^{i\psi}\frac{\sin\theta}{\sqrt{2}} & \cos\theta & ie^{-i\psi}\frac{\sin\theta}{\sqrt{2}}  \cr
n=1 & -e^{-i(\varphi-\psi)/2}\sin^2\frac{\theta}{2} & ie^{-i\varphi}\frac{\sin\theta}{\sqrt{2}} & e^{-i(\varphi+\psi)/2}\cos^2\frac{\theta}{2}
};
$$

From now on, we shall write $\Ft{f}(l)$ for the Fourier transform of $f\in\mathcal{D}'(\SU(2))$:
\begin{equation}\label{FtSU(2)}
\begin{aligned}
\Ft{f}(l)_{mn}
&=\int_{\SU(2)}f(x)\overline{T^l_{nm}}(x)dx\\
&=\frac{1}{16\pi^2}\int_0^{4\pi}\int_{0}^\pi\int_0^{2\pi} f\big(\Omega(\varphi,\theta,\psi)\big)
\overline{T^l_{nm}}\big(\Omega(\varphi,\theta,\psi)\big)\sin\theta \cdot d\varphi d\theta d\psi.
\end{aligned}
\end{equation}
The Peter-Weyl theorem \ref{PeterWeyl} now takes the form
$$
f(x)=\sum_{l\in\hN}(2l+1)\sum_{m,n=-l}^l \Ft{f}(l)_{mn}T^l_{nm}(x).
$$
Here the convention of summation for $m,n$ is that they exhaust all numbers $-l,\,-l+1,\,\cdots,\,l-1,\,l$. The functions $\{T^l_{nm}\}_{m,n=-l}^l$ span the eigenspace $\Hh[l]$ of $\Delta_{G_0}$ with eigenvalue $-l(l+1)$. 

Since $\SU(2)$ can be identified as the unit 3-sphere in $\mathbb{C}^2$, the functions $\{T^l_{nm}\}_{m,n=-l}^l$ are exactly the spherical harmonics on 3-sphere of degree $2l+1$. The spectral localization property, i.e. Corollary \ref{SpecPrd}, now becomes the well-known fact that the product of spherical harmonics of degree $p$ and $q$ is a linear combination of spherical harmonics of degree between $|p-q|$ and $p+q$. However, it should be emphasized that the Euler angle parameterization \emph{does not} coincide with the spherical coordinate parameterization of spherical harmonics.

\subsection{Symbolic Calculus on \texorpdfstring{$\SU(2)$}{d}}

We now construct pseudo-differential calculus on $\SU(2)$. For 
$$
x=\left(\begin{matrix}
x_1 & x_2 \\
-\bar{x}_2 & \bar{x}_1
\end{matrix}\right)\in \SU(2),
$$
the functions
$$
x_1-1,\quad \bar{x}_1-1,\quad x_2,\quad -\bar{x}_2 
$$
form a strongly RT-admissible tuple of $\SU(2)$ defined in \ref{RTAdm}, which is also the fundamental tuple of $\SU(2)$ defined in (\ref{QFund}). The only disadvantage is that this tuple consists of 4 elements and can bring unnecessary tedium for calculation. Following \cite{RT2013}, we introduce instead
\begin{equation}\label{SU(2)q}
\mathscr{Q}_+(x)=-\bar{x}_2,\quad \mathscr{Q}_-(x)=x_2,\quad \mathscr{Q}_0(x)=x_1-\bar{x}_1.
\end{equation}
The tuple $\mathscr{Q}=(\mathscr{Q}_+,\mathscr{Q}_-,\mathscr{Q}_0)$ is RT-admissible but not strongly RT-admissible, since $-\mathrm{Id}$ is the other common zero of them besides $\mathrm{Id}$. A direct computation gives 
$$
(\Pa[+]\mathscr{Q}_+)(\mathrm{Id})=(\Pa[-]\mathscr{Q}_-)(\mathrm{Id})=(\Pa[0]\mathscr{Q}_0)(\mathrm{Id})=-1.
$$
Taylor's formula, i.e. Proposition \ref{TaylorGp} is then valid for $q$ and the differential operators $\Pa[+],\Pa[-],\Pa[0]$. Note that the higher order left-invariant differential operators in Taylor's formula, denoted by $\partial_x^{(\alpha)}$ from now on, have to be defined inductively, since $\Pa[+],\Pa[-],\Pa[0]$ do not commute. 

Since the dual of $\SU(2)$ is labelled by $l\in\hN$, we write $a(x,l)\in\mathrm{End}(\Hh[l])$ for a symbol $a$ on $\SU(2)$. The symbolic calculus on $\SU(2)$ then starts with the symbol of composition
\begin{equation}\label{SU(2)Compo}
(a\#_{r;\mathscr{Q}} b)(x,l)=\sum_{|\alpha|\leq r}\Df^\alpha_{\mathscr{Q},l}a(x,l)\cdot\partial_x^{(\alpha)}b(x,l),
\end{equation}
and the symbol of adjoint
\begin{equation}\label{SU(2)Adj}
a^{\bullet;r,\mathscr{Q}}(x,l)=\sum_{|\alpha|\leq r}\Df^\alpha_{\mathscr{Q},l}\partial_x^{(\alpha)}a^*(x,l).
\end{equation}
We directly cite the following results from \cite{RT2009}:
\begin{theorem}\label{SymPa}
The symbols $\sigma_{+}$, $\sigma_{-}$ and $\sigma_0$ of $\Pa[+]$, $\Pa[-]$ and $\Pa[0]$ respectively are all Fourier multipliers:
$$
\begin{aligned}
\sigma_+(l)_{mn}&=-\sqrt{(l - n)(l + n + 1)}\delta_{m,n+1}\\
\sigma_-(l)_{mn}&=-\sqrt{(l + n)(l - n + 1)}\delta_{m,n-1}\\
\sigma_0(l)_{mn}&=n\delta_{mn}.
\end{aligned}
$$
\end{theorem}
\begin{theorem}\label{Df_q}
Suppose $a$ is a Fourier multiplier on $\SU(2)$, i.e. for each $l\in\hN$, the value $a(l)\in\mathrm{End}(\Hh[l])$. Then the RT difference operators $\Df_+$, $\Df_-$, $\Df_0$ corresponding to $\mathscr{Q}_+$, $\mathscr{Q}_-$, $\mathscr{Q}_0$ respectively are 
$$
(\Df_+a)(l)_{nm}
=\frac{\sqrt{(l + m)(l - n)}}{2l + 1}a(l^-)_{n^+m^-}-\frac{\sqrt{(l - m+1)(l + n+1)}}{2l + 1}a(l^+)_{n^+m^-},
$$
$$
(\Df_-a)(l)_{nm}
=\frac{\sqrt{(l - m)(l + n)}}{2l + 1}a(l^-)_{n^-m^+}-\frac{\sqrt{(l + m+1)(l - n+1)}}{2l + 1}a(l^+)_{n^-m^+},
$$
$$
\begin{aligned}
(\Df_0a)(l)_{nm}
&=\frac{\sqrt{(l - m)(l - n)}}{2l + 1}a(l^-)_{n^+m^+}+\frac{\sqrt{(l + m+1)(l + n+1)}}{2l + 1}a(l^+)_{n^+m^+}\\
&\quad -\frac{\sqrt{(l + m)(l + n)}}{2l + 1}a(l^-)_{n^-m^-}-\frac{\sqrt{(l - m+1)(l - n+1)}}{2l + 1}a(l^+)_{n^-m^-}
\end{aligned}
$$
Here $k^{\pm}=k\pm1/2$. In particular,
$$
\Df_\mu \sigma_\nu=\delta_{\mu\nu}\cdot\sigma_{\mathrm{Id}},\quad \mu,\nu=+,\,-,\,0.
$$
\end{theorem}
\begin{remark}\label{Para-Classical}
We are actually able to manipulate para-differential operators arising from classical differential operators on $\SU(2)$ fairly easily. For example, let $Y=a\Pa[+]$ and $Z=b\Pa[-]$ be vector fields with coefficients $a,b\in C^r$ with $r>1$. Then the commutator 
$$
[Y,Z]
=a\Pa[+]b\Pa[-]-b\Pa[-]a\Pa[+]
+2ab\Pa[0]
$$
On the other hand, the symbols of $Y,Z$ are respectively
$$
\sigma[Y]=a\sigma_+,
\quad
\sigma[Z]=b\sigma_-,
$$
so a direct computation using Theorem \ref{Df_q} implies a precise equality
$$
\begin{aligned}
\sigma[Y]\#_{r;\mathscr{Q}}\sigma[Z]-\sigma[Z]\#_{r;\mathscr{Q}}\sigma[Y]
&=ab[\sigma_+,\sigma_-]+a\Df_+\sigma_+\Pa[+]b\sigma_--b\Df_-\sigma_-\Pa[-]a\sigma_+\\
&=2ab\sigma_0+a\Pa[+]b\sigma_--b\Pa[-]a\sigma_+.
\end{aligned}
$$
Consequently, Theorem \ref{Compo2} ensures that $[T_Y,T_Z]$ is exactly the para-differential operator corresponding to $[Y,Z]$, modulo a smoothing operator in $\Op\Sigma_{<1/2}^{1-r}$. Consequently, with $\{Y,Z\}$ being the symbol of $[Y,Z]$,
$$
[T_Y,T_Z]=T_{\{Y,Z\}}+\Op\Sigma_{<1/2}^{1-r}.
$$
\end{remark}

\subsection{Passage to the 2-Sphere via Hopf Fiberation}
We finally describe how pseudo-differential calculus can be constructed on $\mathbb{S}^2$. Although $\mathbb{S}^2$ does not carry any Lie group structure, it is still a homogeneous space corresponding to $\SU(2)$, and the projection $\SU(2)\to \mathbb{S}^2$ is \emph{Hopf fiberation}, well-known in geometric topology. To be consistent with \emph{left-invariant differential operators} we have been employing so far, following \cite{RT2013}, we consider $\mathbb{S}^2$ as a \emph{right} $\SU(2)$-homogeneous space, that is, the set of \emph{right cosets} 
$$
\mathbf{T}_3\setminus\SU(2):=\Big\{\mathbf{T}_3\cdot x:\,x\in\SU(2)\Big\}
$$
endowed with the quotient topology, on which $\SU(2)$ \emph{acts from the right}. A point that one should keep in notice is that, in the Euler angle representation as in (\ref{EulerAngle1}), the angle $\varphi\in[0,2\pi)$ only exhausts half of the subgroup $\mathbf{T}_3$, so the orbit of a given element $\Omega(\varphi_0,\theta_0,\psi_0)\in\SU(2)$ is
$$
\Big\{\pm\Omega(\varphi,\theta_0,\psi_0):\,
\varphi\in[0,2\pi)\Big\}.
$$
We thus realize the Hopf map $\SU(2)\to \mathbb{S}^2$ as
$$
\left(\begin{matrix}
x_1 & x_2 \\
-\bar{x}_2 & \bar{x}_1
\end{matrix}\right)
\to
\left(\begin{matrix}
-2\mathrm{Im}(x_1\bar x_2) \\
2\mathrm{Re}(x_1\bar x_2) \\
|x_1|^2-| x_2|^2
\end{matrix}\right)\in\mathbb{R}^3,
\quad
\text{where}
\quad
|x_1|^2+|x_2|^2=1,
$$
or in Euler angles,
$$
\Omega(\varphi,\theta,\psi)=\left(\begin{matrix}
e^{i(\varphi+\psi)/2}\cos\frac{\theta}{2} & ie^{i(\varphi-\psi)/2}\sin\frac{\theta}{2} \\
ie^{-i(\varphi-\psi)/2}\sin\frac{\theta}{2} & e^{-i(\varphi+\psi)/2}\cos\frac{\theta}{2}
\end{matrix}\right)
\to
\left(\begin{matrix}
\sin\theta\cos\psi \\
\sin\theta\sin\psi \\
\cos\theta
\end{matrix}\right).
$$
We note that the Hopf fiberation is a Riemannian submersion if $\SU(2)$ is endowed with the metric $G_0$ and $\mathbb{S}^2$ is endowed with the standard metric $g_0$.

A function on $\mathbb{S}^2$ is lifted to $\SU(2)$ as a function independent to the Euler angle $\varphi$ and symmetric with respect to the reflection map $x\to -x$. Given a function $f$ (distribution) on $\mathbb{S}^2$, we denote by 
\begin{equation}\label{Fsharp}
f^\sharp\big(\Omega(\varphi,\theta,\psi)\big)
:=f\big(\sin\theta\cos\psi,\sin\theta\sin\psi,\cos\theta\big)
\end{equation}
its lift to the $\mathbf{T}_3$-invariant function (distribution) on $\SU(2)$. We refer smooth functions (distributions respectively) on $\SU(2)$ with such symmetry as being \emph{$\mathbf{T}_3$-invariant}, and denote the space of them by $C^\infty(\mathbf{T}_3\setminus\SU(2))$ ($\mathcal{D}'(\mathbf{T}_3\setminus\SU(2))$ respectively). We emphasize that sometimes it is necessary to distinguish between $C^\infty(\mathbf{T}_3\setminus\SU(2))$ and $C^\infty(\mathbb{S}^2)$. Further definitions can be formulated for operators and symbols:
\begin{definition}
A linear operator $A:C^\infty(\SU(2))\to \mathcal{D}'(\SU(2))$ is said to be $\mathbf{T}_3$-invariant, if it maps $C^\infty(\mathbf{T}_3\setminus\SU(2))$ to $\mathcal{D}'(\mathbf{T}_3\setminus\SU(2))$.
\end{definition}
\begin{definition}
A symbol $a(x,l)$ on $\SU(2)$ is said to be $\mathbf{T}_3$-invariant, if for any element $\omega_3(t)\in\mathbf{T}_3$, there holds $a(\omega_3(t)x,l)=a(x,l)$, for all $x\in\SU(2)$ and $l\in\hN$. In Euler angles, $\mathbf{T}_3$-invariant symbols are exactly those that do not depend on the Euler angle $\varphi$ and remain invariant under reflection $x\to-x$. 
\end{definition}
Obviously an operator is $\mathbf{T}_3$-invariant if and only if its symbol is $\mathbf{T}_3$-invariant. We can thus lift pseudo-differential operators on $\mathbb{S}^2$ to $\SU(2)$ via the realization of Hopf fiberation that we fixed above. There are necessarily infinitely many ways of lifting and there is no ``natural" one among them. However, for classical differential operators and pseudo-differential operators constructed out of the Laplacian $\Delta_{g_0}$, lifting to $\SU(2)$ as $\Delta_{G_0}$ is direct. The advantage of this approach is that operators on $\SU(2)$ admit an explicitly manipulable symbolic calculus. In the formulas for symbolic calculus, if $a,b$ are both $\mathbf{T}_3$-invariant, then the symbols
\begin{equation}\label{ParaCompoSU2}
(a\#_{r;\mathscr{Q}} b)(x,l)=\sum_{|\alpha|\leq r}\Df^\alpha_{\mathscr{Q},l}a(x,l)\cdot\partial_x^{(\alpha)}b(x,l),
\end{equation}
and
\begin{equation}\label{ParaAdjSU2}
a^{\bullet;r,\mathscr{Q}}(x,l)=\sum_{|\alpha|\leq r}\Df^\alpha_{\mathscr{Q},l}\partial_x^{(\alpha)}a^*(x,l).
\end{equation}
are still $\mathbf{T}_3$-invariant, since the differential operators $\partial_x^{(\alpha)}$ are left-invariant, hence specifically invariant under left translation via $\omega_3(\varphi)$. Briefly speaking, \emph{symbolic calculus for $\mathbf{T}_3$-invariant operators preserves $\mathbf{T}_3$-invariance}. Consequently, symbolic calculus on $\mathbb{S}^2$ can be constructed by lifting operators on $\mathbb{S}^2$ to $\mathbf{T}_3$-invariant operators on $\SU(2)$, and manipulated without concern for $\mathbf{T}_3$-invariance. For a $\mathbf{T}_3$-invariant rough symbol $a\in\mathcal{A}^m_r$, the corresponding para-differential operator $T_a:=\Op(a^\chi)$ is still $\mathbf{T}_3$-invariant, since spectral cut-off preserves $\mathbf{T}_3$-invariance.

It is more illustrative to write down the Fourier series development. For example, given a function $f\in C^\infty(\mathbb{S}^2)$, we denote its lift to $\SU(2)$ as $f^\sharp\in C^\infty(\mathbf{T}_3\setminus\SU(2))$. By (\ref{FtSU(2)}), The Fourier transform $\Ft{f^\sharp}(l)$ does not vanish only when $l\in\mathbb{N}_0$, and for $l\in\mathbb{N}_0$, the only possible non-vanishing entries are those with $m=0$, i.e. the column with index 0. In that case, we have
$$
\Ft{f^\sharp}(l)_{0n}=\int_{\mathbb{S}^2}f(x)Y^{-n}_l(x)d\mu_0,
$$
with $Y^n_l(x)$ being the standard complex spherical harmonic on $\mathbb{S}^2$. Furthermore, for a $\mathbf{T}_3$-invariant symbol $a(x,l)$, the action $\Op(a)f^\sharp$ reads
$$
\sum_{l\in\mathbb{N}_0}(2l+1)\Tr\big[a(x,l)\Ft{f^\sharp}(l)T^l(x)\big].
$$
Under Euler angle representation, the matrix $\Ft{f^\sharp}(l)T^l(x)$ consists of entries depending only on $\psi$ and $\theta$, so each summand, hence $\Op(a)f^\sharp$, is a function of Euler angles $\theta$ and $\psi$ only. The reflection invariance is checked similarly. As a result, $\Op(a)f^\sharp\in C^\infty(\mathbf{T}_3\setminus \SU(2))$, hence projects to a function defined on $\mathbb{S}^2$. 

\section{Para-linearization of Dirichlet-Neumann Operator for a Distorted Sphere}\label{4}
In this section, we prove the para-linearization formula for the Dirichlet-Neumann operator of a distorted sphere in $\mathbb{R}^3$. We are going to lift our discussion to $\SU(2)$, so that the toolbox of para-differential calculus is available. The idea is to ``factorize" the Laplacian and then reduce to the boundary.

\subsection{The Dirichlet-Neumann Operator}
Recall that we set $\iota_0:\mathbb{S}^2\hookrightarrow\mathbb{R}^3$ to be the standard embedding. We write $\rho$ for the Euclidean distance function to $0$ in $\mathbb{R}^3$, and define
$$
\mathcal{U}=\left\{\frac{1}{2}<\rho<\frac{3}{2}\right\},
$$
so that $\mathcal{U}$ is a tubular neighbourhood of the standard unit sphere, with smooth boundary. Obviously ${\mathcal{U}}$ is diffeomorphic to the product manifold $\mathbb{S}^2\times(-1/2,1/2)$, with smooth boundary $\mathbb{S}^2\times\{-1/2,1/2\}$. 

We shall use $x$ to mark points on $\mathbb{S}^2$ or $\SU(2)$ (there is usually no risk of confusion). We already used $g_0$ to denote the induced metric and $N_0$ to denote the outward pointing normal vector field of $\mathbb{S}^2$, also known as the Gauss map. Note that $N_0$ coincides with $\iota_0$.

Just as shown in Figure \ref{Height}, the distorted sphere $M_\zeta$ we shall consider will be given by the graph of a ``height function" $\zeta\in C^r(\mathbb{S}^2)$ with $|\zeta|_{L^\infty}<1/2$, i.e.
$$
M_\zeta:=\{(1+\zeta(x))N_0(x):x\in \mathbb{S}^2\}.
$$
We may thus abbreviate $\iota_\zeta=(1+\zeta) N_0$, and $M_0=\iota_0(\mathbb{S}^2)$. If $r\geq1$, $M_\zeta$ is $C^r$-diffeomorphic to $\mathbb{S}^2$ and is itself a $C^r$ differentiable hypersurface, but this is the highest regularity to expect for generic $\zeta\in C^r(\mathbb{S}^2)$. From now on, we shall fix $r>2$. 

The most convenient coordinate system that we shall use is the distorted normal coordinate with respect to $M_0$, that is a $C^r$ diffeomorphism mapping a point $(x,y)\in \mathbb{S}^2\times[-1/2,1/2]$ to 
$$
(1+\zeta(x)+y)N_0(x)\in\bar{\mathcal{U}}.
$$
Thus, under our choice of coordinate, we have 
$$
\rho(x,y)=1+\zeta(x)+y.
$$

We then pull everything on $\bar{\mathcal{U}}$ back to $\mathbb{S}^2\times[-1/2,1/2]$ through this specific diffeomorphism. As a starting point, the pulled-back Euclidean metric reads
\begin{equation}\label{g_E}
g_{\mathrm{E}}=\rho^2g_0(x)+(d\zeta+dy)\otimes (d\zeta+dy),
\quad 
\rho=1+\zeta+y.
\end{equation}
In deriving the expression for $g_{\mathrm{E}}$, we used the second fundamental form $\langle d\iota_0(x), dN_0(x)\rangle$, where the bracket denotes the pairing $\langle a_idx^i, c_jdx^j\rangle:=(a_i\cdot c_j)dx^i\otimes dx^j$ between $\mathbb{R}^{3}$-valued differential 1-forms on $\mathbb{S}^2$. Note that for the sphere, the second fundamental form coincides with $g_0$. The following $C^r$-Riemannian manifold (with boundary) is thus isometric to $\bar{\mathcal{U}}$:
$$
\bar{\mathcal{W}}:=\Big(\mathbb{S}^2\times\left[-\frac{1}{2},\frac{1}{2}\right],g_{\mathrm{E}}\Big).
$$
The induced metric on $M_\zeta$, being the hypersurface $y=0$ in $\bar{\mathcal{W}}$, also takes a simple form
$$
g_\zeta=(1+\zeta)^2g_0+d\zeta\otimes d\zeta.
$$
We use the following formulas for block matrices to compute the dual metric of $g_{\mathrm{E}}$: 
$$
\left(
\begin{matrix}
A+c c^\top & c \\
c^{\top} & 1
\end{matrix}
\right)^{-1}
=\left(
\begin{matrix}
A^{-1} & -A^{-1}c \\
-(A^{-1}c)^{\top} & 1+(A^{-1}c)\cdot c
\end{matrix}
\right),
\quad
\det\left(
\begin{matrix}
A+c c^\top & c \\
c^{\top} & 1
\end{matrix}
\right)=\det A
$$
for an $2\times 2$ positive definite symmetric matrix $A$  and $c\in\mathbb{R}^2$. We find the dual metric of $g_{\mathrm{E}}$ reads
\begin{equation}\label{g_E^{-1}}
\begin{aligned}
g_{\mathrm{E}}^{-1}
&=\frac{g_0^{-1}(x)}{\rho^{2}}
+\left(1+\frac{|\nabla_0\zeta|_{g_0}^2}{\rho^{2}}\right)\partial_y\otimes\partial_y
-\frac{\nabla_0\zeta\otimes\partial_y+\partial_y\otimes\nabla_0\zeta}{\rho^{2}}.
\end{aligned}
\end{equation}
Here $g_0^{-1}$ stands for the dual metric of $g_0$.

For a function $\Phi$ defined on $\bar{\mathcal{U}}$, we pull it back to $\bar{\mathcal{W}}$ as
\begin{equation}\label{Psi}
\Psi(x,y):=\Phi\big((1+\zeta(x)+y)N_0(x)\big).
\end{equation}
Regarding $\Psi$ as a function defined on $\mathbb{S}^2$ depending on parameter $y\in[-1/2,0]$, the gradient of $\Psi$ is computed as
$$
\nabla_\mathrm{E}\Psi
=\frac{\nabla_0\Psi-\partial_y\Psi\nabla_0\zeta}{\rho^{2}}
+\left[\left(1+\frac{|\nabla_0\zeta|_{g_0}^2}{\rho^{2}}\right)\partial_y\Psi-\frac{\nabla_0\zeta\cdot\nabla_0\Psi}{\rho^{2}}\right]\partial_y.
$$
The hypersurface $y=0$ corresponds to $M_\zeta$ in $\mathcal{U}$, so the vector field
$$
\nabla_{\mathrm{E}}y=-\frac{\nabla_0\zeta}{\rho^{2}}
+\left(1+\frac{|\nabla_0\zeta|_{g_0}^2}{\rho^{2}}\right)\partial_y\Bigg|_{y=0}
$$
along $\{y=0\}$ is perpendicular to the hypersurface. Similarly, the hypersurface $y+\zeta(x)=0$ corresponds to the undisturbed sphere in $\mathcal{U}$, so 
$$
N_0(x)=\partial_y|_{y=-\zeta(x)}.
$$
Consequently, using the expression (\ref{g_E}) of $g_{\mathrm{E}}$, the Dirichlet-Neumann operator ${D[\zeta]\phi}$ in (\ref{EQ}) is computed as
\begin{equation}\label{Dzetaphi}
\begin{aligned}
D[\zeta]\phi
&=\frac{\nabla_{\mathrm{E}}\Psi\cdot\nabla_{\mathrm{E}}y}{N_0\cdot\nabla_{\mathrm{E}}y}\Bigg|_{y=0}\\
&=\left(1+\frac{|\nabla_0\zeta|_{g_0}^2}{\rho^{2}}\right)\partial_y\Psi\Bigg|_{y=0}-\frac{\nabla_0\zeta\cdot\nabla_0\Psi}{\rho^{2}}\Bigg|_{y=0}.
\end{aligned}
\end{equation}
The $g_{\mathrm{E}}$-Laplacian of $\Psi$, which is just the Euclidean Laplacian of $\Psi$ pulled back to $\bar{\mathcal{W}}$, is computed as
$$
\begin{aligned}
\Delta_{\mathrm{E}}\Psi
&=\frac{\Delta_0\Psi}{\rho^{2}}
+\left(1+\frac{|\nabla_0\zeta|_{g_0}^2}{\rho^{2}}\right)\partial_y^2\Psi
-\frac{2\nabla_0\zeta\cdot(\nabla_0\partial_y\Psi)}{\rho^{2}}-\frac{\Delta_0\zeta-2\rho}{\rho^{2}}\partial_y\Psi.
\end{aligned}
$$
Note that we used the conformality of $\rho^{2}g_0$ with $g_0$ and the fact that $\mathbb{S}^2$ has dimension 2.

The regularity of $D[\zeta]\phi$ can be deduced by a standard elliptic regularity argument using the expressions of $\Delta_{\mathrm{E}}$ under the coordinate system $(x,y)$. The following proposition is proved similarly as Proposition 2.7. of \cite{ABZ2011}:
\begin{proposition}\label{RegDN}
Suppose $s>3$, $s'\leq s$, and $\zeta\in H^{s+0.5}$, $\phi\in H^{s'}$, such that $|\zeta|_{L^\infty}<1/2$. Then there is an increasing function $K$ such that
$$
\|D[\zeta]\phi\|_{H^{s'-1}}
\leq K\big(\|\zeta\|_{H^{s+0.5}}\big)\|\phi\|_{H^{s'}}.
$$
\end{proposition}

Furthermore, the linearization of $D[\zeta]\phi$ can be computed similarly as in the Euclidean case. We directly state the following proposition, which resembles Proposition 2.11. of \cite{ABZ2011}; the proof is quite standard as in \cite{Lannes2005}, and is omitted.
\begin{proposition}\label{LinDN}
Suppose $s>3$, $s'\leq s$, and $\zeta\in H^{s+0.5}$, $\phi\in H^{s'}$, such that $|\zeta|_{L^\infty}<1/2$. There is a neighbourhood $\mathfrak{U}$ of $\zeta$ in $H^{s+0.5}$ in which the mapping 
$$
\mathfrak{U}\ni\zeta\to D[\zeta]\phi\in H^{s'-1}
$$
is differentiable, and the differential along direction $\zeta_1$ is 
$$
-D[\zeta]\big(\mathfrak{b}\zeta_1\big)-\mathrm{div}_{g_0}(\mathfrak{v}\zeta_1),
$$
where 
$$
\mathfrak{b}=\left(1+\frac{|\nabla_{0}\zeta|^2}{(1+\zeta)^{2}}\right)^{-1}\left(
D[\zeta]\phi+\frac{\nabla_{0}\zeta\cdot\nabla_{0}\phi}{(1+\zeta)^{2}}\right),
\quad
\mathfrak{v}=\frac{\nabla_0\phi-\mathfrak{b}\nabla_0\zeta}{(1+\zeta)^2},
$$
as in Theorem \ref{Thm1}.
\end{proposition}

In order to para-linearize the Dirichlet-Neumann operator, we will lift all the quantities of interest from $\mathbb{S}^2\times[-1/2,0]$ to $\SU(2)\times[-1/2,0]$. Suppose $\Phi$ is the velocity potential inside the region enclosed by $M_\zeta$, and $\phi(x):=\Phi\big((1+\zeta(x))N_0(x)\big)$ is the Dirichlet boundary value of $\Phi$ on $M_\zeta$. Then the corresponding pulled-back $\Psi$ satisfies $\Delta_{\mathrm{E}}\Psi=0$ on $\mathbb{S}^2\times[-1/2,0)$, and $\Psi(x,0)=\phi(x)$. Writing $\rho^\sharp=1+\zeta^\sharp+y$, defining the elliptic differential operator 
\begin{equation}\label{L_zeta}
\begin{aligned}
\mathcal{L}_\zeta
&:=\Delta_{G_0}
+\left((\rho^\sharp)^{2}+|\nabla_{G_0}\zeta^\sharp|_{G_0}^2\right)\partial_y^2
-2\nabla_{G_0}\zeta^\sharp\cdot(\nabla_{G_0}\partial_y)
-\left(\Delta_{G_0}\zeta^\sharp-2\rho^\sharp\right)\partial_y
\end{aligned}
\end{equation}
on $\SU(2)\times[-1/2,0]$, the lift $\Psi^\sharp$ satisfies $\mathcal{L}_\zeta\Psi^\sharp=0$ on $\SU(2)\times[-1/2,0)$, and $\Psi^\sharp(x,0)=\phi^\sharp(x)$. In fact, $\mathcal{L}_\zeta$ is just the lift of $\rho^2\Delta_{\mathrm{E}}$. The quantity $D[\zeta]\phi$ is lifted as the boundary value of
\begin{equation}\label{N_zeta}
\mathcal{N}_\zeta\Psi^\sharp
:=\left(1+\frac{|\nabla_{G_0}\zeta^\sharp|_{G_0}^2}{(\rho^\sharp)^{2}}\right)\partial_y\Psi^\sharp
-\frac{\nabla_{G_0}\zeta^\sharp\cdot\nabla_{G_0}\Psi^\sharp}{(\rho^\sharp)^{2}}
\end{equation}
at $y=0$. Our task now becomes the following:

\begin{task}\label{TaskDN}
Let $\mathcal{L}_\zeta$ and $\mathcal{N}_\zeta$ be as in (\ref{L_zeta}) and (\ref{N_zeta}). Find a suitable $\mathbf{T}_3$-invariant symbol $a$ on $\SU(2)$ such that if $\mathcal{L}_\zeta\Psi^\sharp=0$ on $\SU(2)\times[-1/2,0)$, and $\Psi^\sharp(x,0)=\phi^\sharp(x)$, then 
$$
\mathcal{N}_\zeta\Psi^\sharp\big|_{y=0}=T_a\phi^\sharp+\mathbf{T}_3\text{-invariant smoothing terms}.
$$
\end{task}
The rest of this section is devoted to finding out this paralinearization formula. The idea of accomplishing Task \ref{TaskDN} is simple: considering the ``height" coordinate $y$ as a parameter of evolution, the quantity $\mathcal{N}_\zeta\Psi^\sharp$ is the final state of an elliptic evolutionary problem $\mathcal{L}_\zeta\Psi^\sharp=0$; by suitable factorization of the operator $\mathcal{L}_\zeta$, the boundary value $\mathcal{N}_\zeta\Psi^\sharp$ then admits an explicit expression. 

The idea is exactly the one employed by, for example, Appendix C of Chapter 12 of Taylor's book \cite{Taylor2013}. Its para-differential version is exactly the argument employed by \cite{AM2009} and \cite{ABZ2011}. We will basically follow \cite{AM2009} and \cite{ABZ2011}, to whose idea our argument is very close.

From now on, we will fix a real number $s>3$. We assume that the height function $\zeta\in H^{s+0.5}(\mathbb{S}^2)$, so that the surface $M_\zeta$ is of class $C^{s-0.5}_*\subset C^{2.5+}$. We also assume that the boundary value $\phi\in H^{s}(\mathbb{S}^2)\subset C^{s-1}_*(\mathbb{S}^2)\subset C^{2+}(\mathbb{S}^2)$. Although the lifting operator $\sharp$ from functions on $\mathbb{S}^2$ to functions on $\SU(2)$ does not improve Sobolev regularity, it obviously does not undermine Zygmund regularity either: 
\begin{proposition}\label{SobolevForT3}
For $\mathbf{T}_3$-invariant functions on $\SU(2)$, the Sobolev embedding
$$
H^s\big(\mathbf{T}_3\setminus\SU(2)\big)
\subset C^{s-1}_*\big(\mathbf{T}_3\setminus\SU(2)\big)
$$
is valid for $s>1$.
\end{proposition}
For example, we still have 
$$
\begin{aligned}
&\zeta^\sharp\in C^{s-0.5}_*\big(\mathbf{T}_3\setminus\SU(2)\big)
\subset C^{2.5+}\big(\mathbf{T}_3\setminus\SU(2)\big),\\
&\phi^\sharp\in C^{s-1}_*\big(\mathbf{T}_3\setminus\SU(2)\big)
\subset C^{2+}\big(\mathbf{T}_3\setminus\SU(2)\big),
\end{aligned}
$$
although $H^s(\SU(2))\not\subset C^{s-1}_*(\SU(2))$. From now on we we will always be dealing with symbols or functions on $\SU(2)$ with $\mathbf{T}_3$-invariance described above, so this improved Sobolev embedding always holds.

\subsection{Alinhac's Good Unknown}
We turn to the para-linearization of equation $\mathcal{L}_\zeta\Psi^\sharp=0$. We shall look at the Laplace equation $\Delta_{\mathrm{E}}\Psi=0$ inside $M_\zeta$, and after lifting to $\SU(2)\times[-1/2,0]$, we shall find the corresponding \emph{good unknown} in the sense of Alinhac \cite{Alinhac1989}; that is, we need to find the optimal unknown function which eliminates all loss of regularity due to the non-smooth coordinate change. The idea was originally introduced by Alinhac in \cite{Alinhac1989}, Subsection 3.1.. Lannes \cite{Lannes2005} noticed its importance to the study of the Dirichlet-Neumann operator. The para-differential version of the good unknown was employed by \cite{AM2009} and \cite{ABZ2011}. We will basically follow the approach of \cite{AM2009} and \cite{ABZ2011}. 

We start by finding out a harmonic function inside $M_\zeta$ that envelopes all information of the normal direction velocity. We observe a nice commutation property between $\Delta_{\mathrm{E}}$ with $\rho\Pa[\rho]$, the infinitesimal generator of dilation:
$$
\big[\rho^2\Delta_\mathrm{E},\rho\Pa[\rho]\big]
=(2-2\rho^2)\Delta_{\mathrm{E}},
$$
we find that $\rho\partial_\rho\Phi$ is a harmonic function inside $M_\zeta$. Since $\rho=1+\zeta(x)+y$ for our choice of coordinates $(x,y)\in\mathbb{S}^2\times[-1/2,0]$, by pulling back to $\bar{\mathcal{W}}$, it follows that 
$$
\Delta_\mathrm{E}\mathfrak{B}=0,
\quad\text{where}\quad
\mathfrak{B}=\rho\partial_y\Psi.
$$
Lifting to $\SU(2)\times[-1/2,0]$, setting $\rho^\sharp=1+\zeta^\sharp+y$, we find
$$
\mathcal{L}_\zeta\mathfrak{B}^\sharp=0,
\quad\text{where}\quad
\mathfrak{B}^\sharp=\rho^\sharp\partial_y\Psi^\sharp.
$$
We define scalar functions
\begin{equation}\label{Beta13}
\begin{aligned}
\beta_1(y;x)&=\big(\rho^\sharp(x,y)\big)^{2}+|\nabla_{G_0}\zeta^\sharp(x)|_{G_0}^2,
\\
\beta_3(y;x)&=-\Delta_{G_0}\zeta^\sharp(x)+2\rho^\sharp(x,y),
\end{aligned}
\end{equation}
and then define a classical differential symbol
\begin{equation}\label{Beta2}
\begin{aligned}
\beta_2(x,l)
&=\sum_{j=1}^3X_j\zeta^\sharp(x)\cdot\sigma[X_j](l)\\
&=\text{the symbol of }\nabla_{G_0}\zeta^\sharp\cdot\nabla_{G_0},
\end{aligned}
\end{equation}
where the left-invariant vector fields $X_j$ are as in (\ref{LISU(2)}). The operator $\mathcal{L}_\zeta$ is now re-written as
$$
\mathcal{L}_\zeta
=\Delta_{G_0}+\beta_1\partial_y^2-2\nabla_{G_0}\zeta^\sharp\cdot\nabla_{G_0}\partial_y+\beta_3\partial_y.
$$
We are now at the place to introduce the good unknown in the sense of Alinhac:
\begin{proposition}\label{GoodUnknown}
Write $\mathfrak{B}=\rho\partial_y\Psi$. Defining the good unknown $W^\sharp:=\Psi^\sharp-T_{1/\rho^\sharp}T_{\mathfrak{B}^\sharp}\zeta^\sharp$ and the para-differential operator on $\SU(2)\times[-1/2,0]$
$$
\mathcal{P}_\zeta
:=\Delta_{G_0}+T_{\beta_1}\partial_y^2-2T_{\beta_2}\partial_y+T_{\beta_3}\partial_y
$$
corresponding to $\mathcal{L}_\zeta$, we have
$$
\big\|\mathcal{P}_\zeta W^\sharp\big\|_{C^0_yH_x^{s+0.5}}
\lesssim K\big(\|\zeta\|_{H^{s+0.5}}\big)\|\phi\|_{H^s}
$$
on $\SU(2)\times[-1/2,0]$, where $K$ is an increasing function, approximately linear when the argument is small.
\end{proposition}
\begin{remark}
Another reason that $W^\sharp$ has better regularity than $\Psi^\sharp$ is that, by Theorem \ref{Compo2}, we have $T_{1/\rho^\sharp}T_{\mathfrak{B}^\sharp}=T_{\partial_y\Psi^\sharp}\mod\Op\Sigma_{<1/2}^{-(s-2)}$, so
$$
W^\sharp=\Psi^\sharp-T_{\partial_y\Psi^\sharp}\zeta^\sharp
\mod C^0_yH_x^{2s-1.5},
$$
and the difference resembles the remainder term in Bony's para-linearization theorem, i.e. Theorem \ref{Bony}, for the composition $\big[\Phi\big((1+\zeta+y)N_0\big)\big]^\sharp$.
\end{remark}

\begin{notation}\label{Mod}
For any index $s'>0$, we use
$$
\mod C^0_yH_x^{s'}
\quad\text{or simply}\quad
\mod H_x^{s'}
$$
to refer to equality between quantities whose difference has $C^0_yH_x^{s'}$ (or $H^{s'}_x$) norm on $\mathbf{T}_3\setminus\SU(2)$ controlled by $K\big(\|\zeta\|_{H^{s+0.5}}\big)\|\phi\|_{H^s}$.
\end{notation}

\begin{proof}[Proof of Proposition \ref{GoodUnknown}]
Throughout the proof, we employ the ``mod" notation defined in \ref{Mod}.

The proof is divided into three steps.

\textbf{Step 1: intermediate equation for $\mathcal{P}_\zeta w^\sharp$.} Since $\Delta_{\mathrm{E}}\Psi=0$, by exactly the same argument as in \cite{ABZ2011}, Section 3, we obtain
$$
\nabla_0\Psi,\,\partial_y\Psi\in C^0_yH^{s-1}_x(\mathbb{S}^2),
\quad
\partial_y^2\Psi\in C^0_yH^{s-2}_x(\mathbb{S}^2).
$$
Upon lifting to $\SU(2)$, it follows that
$$
\nabla_{G_0}\Psi^\sharp,\,\partial_y\Psi^\sharp\in C^0_yH^{s-1}_x\big(\mathbf{T}_3\setminus\SU(2)\big),
\quad
\partial_y^2\Psi^\sharp\in C^0_yH^{s-2}_x\big(\mathbf{T}_3\setminus\SU(2)\big).
$$
In fact
$$
\big\|\nabla_{G_0}\Psi^\sharp\big\|_{C^0_yH^{s-1}_x}+
\big\|\partial_y\Psi^\sharp\big\|_{C^0_yH^{s-1}_x}
\lesssim \|\phi\|_{H^s},
\quad
\big\|\partial_y^2\Psi^\sharp\big\|_{C^0_yH^{s-2}_x}
\lesssim \|\phi\|_{H^s}.
$$
Using the para-product estimate, i.e. Theorem \ref{ParaPrd} and Remark \ref{R(a)Symbol}, since $\zeta^\sharp\in H^{s+0.5}\big(\mathbf{T}_3\setminus\SU(2)\big)$, we obtain
$$
\begin{aligned}
\mathcal{P}_\zeta\Psi^\sharp
&=\big(\mathcal{P}_\zeta-\mathcal{L}_\zeta\big)\Psi^\sharp
\quad\text{(precise equality)}\\
&=-\left(T_{\partial_y^2\Psi^\sharp}\beta_1-2\sum_{j=1}^3T_{X_j\partial_y\Psi^\sharp}X_j\zeta^\sharp+T_{\partial_y\Psi^\sharp}\beta_3\right)
\mod C_y^0H_x^{s+0.5}.
\end{aligned}
$$

Thus we compute
\begin{equation}\label{P_zetaU}
\begin{aligned}
\mathcal{P}_\zeta w^\sharp
&=\mathcal{P}_\zeta\left(\Psi^\sharp-T_{1/\rho^\sharp}T_{\mathfrak{B}^\sharp}\zeta^\sharp\right)
\quad\text{(precise equality)}\\
&=-\left(T_{1/\rho^\sharp}\mathcal{P}_\zeta T_{\mathfrak{B}^\sharp}\zeta^\sharp +T_{\partial_y^2\Psi^\sharp}\beta_1-2\sum_{j=1}^3T_{X_j\partial_y\Psi^\sharp}X_j\zeta^\sharp
+T_{\partial_y\Psi^\sharp}\beta_3\right)\\
&\quad-\left[\mathcal{P}_\zeta,T_{1/\rho^\sharp}\right]T_{\mathfrak{B}^\sharp}\zeta^\sharp
\mod C_y^0H_x^{s+0.5}.
\end{aligned}
\end{equation}

\textbf{Step 2: elimination of commutator.} To eliminate the commutator $\left[\mathcal{P}_\zeta,T_{1/\rho^\sharp}\right]T_{\mathfrak{B}^\sharp}\zeta^\sharp$ in (\ref{P_zetaU}), we need the composition formula for para-differential operators, i.e. Theorem \ref{Compo2}. The ``classical version" of $[\mathcal{P}_\zeta,T_{1/\rho^{\sharp}}]$ is just $[\mathcal{L}_\zeta,1/\rho^\sharp]$, whose projection to $\mathbb{S}^2\times[-1/2,0]$ is simply 
$$
\left[\rho^2\Delta_{\mathrm{E}},\frac{1}{\rho}\right]
=-2\partial_y\cdot\nabla_{\mathrm{E}}
=-2\partial_y.
$$
Here we used $\Delta_{\mathrm{E}}(1/\rho)=0$. Lifting to $\SU(2)\times[-1/2,0]$, it follows that $[\mathcal{L}_\zeta,1/\rho^\sharp]=-2\partial_y$. 

To pass to the para-differential version $[\mathcal{P}_\zeta,T_{1/\rho^{\sharp}}]$, we shall keep exploiting the fact that $\mathcal{L}_\zeta$ is a classical differential operator. Here with a little abuse of notation, given symbols $a,b$ on $\SU(2)$ leading to classical differential operators, we write $\{a,b\}$ for \emph{the symbol of the commutator} $[\Op(a),\Op(b)]$ (in usual notation the Poisson bracket $\{\cdot,\cdot\}$ only involves first order differentiation of symbols). Now writing down
\begin{equation}\label{[P_zeta,T]}
\begin{aligned}
\big[\mathcal{P}_\zeta,T_{1/\rho^{\sharp}}\big]
&=\big[\Delta_{G_0},T_{1/\rho^{\sharp}}\big]+\big[T_{\beta_1}\partial_y^2,T_{1/\rho^{\sharp}}\big]
-2\big[T_{\beta_2}\partial_y,T_{1/\rho^{\sharp}}\big]
+\big[T_{\beta_3}\partial_y,T_{1/\rho^{\sharp}}\big],
\end{aligned}
\end{equation}
it is not hard to see, with the aid of Theorem \ref{Compo2}, Formula (\ref{SU(2)Compo})  and also Remark \ref{Para-Classical}, that each term in the right-hand-side corresponds to exactly its classical counterpart in $[\mathcal{L}_\zeta,1/\rho^\sharp]$ with an affordable regularizing error. 

For example, by assumption, the number $r:=s-1.5>1.5$, so with $\#_{r;\mathscr{Q}}$ as in Formula (\ref{SU(2)Compo}), the third term in (\ref{[P_zeta,T]}) becomes
$$
\begin{aligned}
\big[T_{\beta_2}\partial_y,T_{1/\rho^{\sharp}}\big]
&=T_{\beta_2}T_{\partial_y(1/\rho^{\sharp})}
+T_{\beta_2}T_{1/\rho^{\sharp}}\partial_y
-T_{1/\rho^\sharp}T_{\beta_2}\partial_y
\quad\text{(precise equality)}\\
&=T_{\beta_2\#_{r;\mathscr{Q}}\partial_y(1/\rho^{\sharp})}
+T_{\{\beta_2,1/\rho^\sharp\}}\partial_y
\mod\Op\Sigma_{<1/2}^{-r}+\Op\Sigma_{<1/2}^{-r}\partial_y,
\end{aligned}
$$
while on the other hand 
$$
\begin{aligned}
\left[\nabla_{G_0}\zeta^\sharp\cdot\nabla_{G_0}\partial_y,\frac{1}{\rho^\sharp}\right]
&=\nabla_{G_0}\zeta^\sharp\cdot\nabla_{G_0}\left(\partial_y\left(\frac{1}{\rho^\sharp}\right)\cdot\right)
+\left[\nabla_{G_0}\zeta^\sharp\cdot\nabla_{G_0},\frac{1}{\rho^\sharp}\right]\partial_y\\
&=\Op\left(\beta_2\#_{r;\mathscr{Q}}\partial_y(1/\rho^{\sharp})\right)+\Op\left(\{\beta_2,1/\rho^\sharp\}\right)\partial_y
\end{aligned}
$$
in $[\mathcal{L}_\zeta,1/\rho^\sharp]$. Here we used the fact that $\beta_2$ is the symbol of a classical first order differential operator. In other words, Remark \ref{Para-Classical} ensures that $[\mathcal{P}_\zeta,T_{1/\rho^\sharp}]$ must coincide with the para-differential operator corresponding to $[\mathcal{L}_\zeta,1/\rho^\sharp]$, which is merely $-2\partial_y$, up to an error term:
$$
\begin{aligned}
\big[\mathcal{P}_\zeta,T_{1/\rho^{\sharp}}\big]
=-2\partial_y
\mod\Op\Sigma_{<1/2}^{-(s-1.5)}\partial_y^2+\Op\Sigma_{<1/2}^{-(s-2.5)}\partial_y+\Op\Sigma_{<1/2}^{-(s-2.5)},
\end{aligned}
$$
The reason that there is an error term of order $-(s-2.5)$ is that the $\mathbf{T}_3$-invariant symbols are $\beta_1,\beta_2$ are of $H^{s-0.5}$ regularity in $x$, while $\beta_3\in H^{s-1.5}\subset C_*^{s-2.5}$.

Since $\partial_y^k\Psi^\sharp\in C_y^0H^{s-k}_x\big(\mathbf{T}_3\setminus\SU(2)\big)$, we find that $\partial_yT_{\mathfrak{B}^\sharp}\zeta^\sharp=T_{\partial_y\mathfrak{B}^\sharp}\zeta^\sharp$ is still of class $C_y^0H^{s+0.5}_x$. On the other hand, the function $\partial_y^2T_{\mathfrak{B}^\sharp}\zeta^\sharp=T_{\partial_y^2\mathfrak{B}^\sharp}\zeta^\sharp$ is of class $C_y^0H^{s-0.5}_x$, since $\partial^2_y\mathfrak{B}^\sharp\in H_x^{s-3}\subset C^{-1}_*$, so the para-product operator $T_{\partial^2_y\mathfrak{B}^\sharp}$ has order 1, according to Proposition \ref{T_aNegIndex}. Taking into account the regularizing operators, we successfully proved that the commutator $\left[\mathcal{P}_\zeta,T_{1/\rho^\sharp}\right]T_{\mathfrak{B}^\sharp}\zeta^\sharp\in C_y^0H^{s+0.5}_x$. Summarizing, we find that (\ref{P_zetaU}) in fact implies
\begin{equation}\label{GoodUnknownTemp1}
\mathcal{P}_\zeta w^\sharp
=-\left(T_{1/\rho^\sharp}\mathcal{P}_\zeta T_{\mathfrak{B}^\sharp}\zeta^\sharp +T_{\partial_y^2\Psi^\sharp}\beta_1-2\sum_{j=1}^3T_{X_j\partial_y\Psi^\sharp}X_j\zeta^\sharp+T_{\partial_y\Psi^\sharp}\beta_3\right)
\mod C_y^0H_x^{2s-2.5}.
\end{equation}
Note that $s>3$, hence $2s-2.5>s+0.5$.

\textbf{Step 3: final elimination.} Now similarly as in the proof of Lemma 3.17. in \cite{ABZ2011}, we compute
\begin{equation}\label{GoodUnknownTemp2}
\begin{aligned}
\mathcal{P}_\zeta T_{\mathfrak{B}^\sharp}\zeta^\sharp
&=\left(\Delta_{G_0}+T_{\beta_1}\partial_y^2-2T_{\beta_2}\partial_y+T_{\beta_3}\partial_y\right)T_{\mathfrak{B}^\sharp}\zeta^\sharp\\
&=2\sum_{j=1}^3T_{X_j\mathfrak{B}^\sharp}X_j\zeta^\sharp
+\sum_{j=1}^3T_{X_j\zeta^\sharp\partial_y\mathfrak{B}^\sharp}X_j\zeta^\sharp
+T_{\mathfrak{B}^\sharp}\Delta_{G_0}\zeta^\sharp
\mod C_y^0H_x^{s+0.5}.
\end{aligned}
\end{equation}
Here we used the $XT_au=T_{Xa}u+T_aXu$ for a left-invariant vector field $X\in\mathfrak{su}(2)$, and the fact that $\mathcal{L}_\zeta\mathfrak{B}^\sharp=0$. Thus
\begin{equation}\label{GoodUnknownTemp3}
\begin{aligned}
\mathcal{P}_\zeta T_{\mathfrak{B}^\sharp}\zeta^\sharp
&=2\sum_{j=1}^3T_{X_j\mathfrak{B}^\sharp}X_j\zeta^\sharp
-T_{\partial_y\mathfrak{B}^\sharp}\beta_1
-T_{\mathfrak{B}^\sharp}\beta_3
&\mod C_y^0H_x^{s+0.5}\\
&=2\sum_{j=1}^3T_{\rho^\sharp X_j\partial_y\Psi^\sharp}X_j\zeta^\sharp
-T_{\rho^\sharp\partial_y^2\Psi^\sharp}\beta_1
-T_{\rho^\sharp\partial_y\Psi^\sharp}\beta_3
&\mod C_y^0H_x^{s+0.5}
\end{aligned}
\end{equation}
Here in the first equality of (\ref{GoodUnknownTemp3}), when computing 
$$
\begin{aligned}
\beta_1
&=|\nabla_{G_0}\zeta^\sharp|_{G_0}^2
&\mod H^{s+0.5}_x\\
&=\sum_{j=1}^3|X_j\zeta^\sharp|^2
&\mod H^{s+0.5}_x\\
&=2\sum_{j=1}^3T_{X_j\zeta^\sharp}X_j\zeta^\sharp
&\mod H^{s+0.5}_x,
\end{aligned}
$$
we used the para-linearization theorem of Bony, i.e. Theorem \ref{Bony}; in the third equality, we used 
$$
\begin{aligned}
2\sum_{j=1}^3T_{\partial_y\Psi^\sharp X_j\rho^\sharp}X_j\zeta^\sharp-T_{\partial_y\Psi^\sharp}\beta_1
&=2\sum_{j=1}^3T_{\partial_y\Psi^\sharp}T_{X_j\zeta^\sharp}X_j\zeta-T_{\partial_y\Psi^\sharp}\sum_{j=1}^3|X_j\zeta^\sharp|^2
&\mod C_y^0H^{2s-2.5}_x\\
&=0 
&\mod C_y^0H^{2s-2.5}_x.
\end{aligned}
$$
The proof is completed by noting that the right-hand-side of (\ref{GoodUnknownTemp2}) cancels with (\ref{GoodUnknownTemp1}) modulo $C_y^0H_x^{s+0.5}$, since $T_{1/\rho^\sharp}T_{\rho^\sharp}=\mathrm{Id}+\Op\Sigma_{<1/2}^{-(s-0.5)}$ by Theorem \ref{Compo2}.
\end{proof}

\subsection{Factorization}
With the elliptic para-differential equation satisfied by the good unknown $w^\sharp$ at hand, we are now at the place to factorize the equation to obtain the behaviour of $\partial_y\Psi^\sharp$ near the boundary $y=0$, to accomplish Task \ref{TaskDN}. This is the essential step for which an explicit, global symbolic calculus on $\SU(2)$ is necessary. The framework of our argument is parallel to Lemma 3.18. of \cite{ABZ2011}, but the details are quite different, since we are obviously working with a different geometry.

So we still assume $s>3$, $\zeta\in H^{s+0.5}(\mathbb{S}^2)$, and set $\beta_1$, $\beta_2$ and $\beta_3$ as in (\ref{Beta13})-(\ref{Beta2}). Imitating Lemma 3.18. of \cite{ABZ2011}, we seek for first-order rough symbols $a,A$ on $\SU(2)$, depending on the parameter $y\in[-1/2,0]$, such that
\begin{equation}\label{Factorization}
\begin{aligned}
\mathcal{P}_\zeta
&=\Delta_{G_0}+T_{\beta_1}\partial_y^2-2T_{\beta_2}\partial_y+T_{\beta_3}\partial_y\\
&=T_{\beta_1}(\partial_y-T_a)(\partial_y-T_A)+R_0+R_1\partial_y.
\end{aligned}
\end{equation}
Here we write $\delta=\min(0.5,s-3)>0$, and suppose $R_0\in\Op\Sigma_{<1/2}^{0.5-\delta}$, $R_1\in\Op\Sigma_{<1/2}^{-0.5-\delta}$.

We need to apply Theorem \ref{Compo2} and the composition formula (\ref{SU(2)Compo}) on $\SU(2)$, with the left-invariant vector fields $\Pa[+,-,0]$ given in (\ref{CANSU(2)}) and RT-admissible tuple $\mathscr{Q}_{+,-,0}$ given in (\ref{SU(2)q}). Suppose that $a,A$ are both of class $C_y^1\mathcal{A}^1_{0.5+\delta}(\SU(2))$. Write $a=a_1+a_0$, $A=A_1+A_0$, with subscript $1,0$ denoting symbols in $\mathcal{A}^1_{1.5+\delta}(\SU(2))$ and $\mathcal{A}^0_{0.5+\delta}(\SU(2))$ respectively. Then we propose the system to be solved as
\begin{equation}\label{AaSystemFull}
\begin{aligned}
a_1A_1+\sum_{\mu\in\Ind}\Df_{\mu}a_1\partial_{\mu}A_1
+a_1A_0+a_0A_1-\partial_yA_1
&=\frac{\sigma[\Delta_{G_0}]}{\beta_1}\\
a_1+A_1+a_0+A_0
&=\frac{2\beta_2-\beta_3}{\beta_1}.
\end{aligned}
\end{equation}
By Theorem \ref{Compo2} and Formula (\ref{SU(2)Compo}), we find that $T_{1/\beta_1}T_{\beta_1}=\mathrm{Id}\mod\Sigma_{<1/2}^{-(s-1.5)}$, so comparing symbols of same order, we find that the resulting $R_0$ and $R_1$ match with our requirement in (\ref{Factorization}).

Now we are at the place to solve (\ref{AaSystemFull}). Preserving the highest order symbols, we obtain
$$
\begin{aligned}
a_1A_1
=\frac{\sigma[\Delta_{G_0}]}{\beta_1},
\quad
a_1+A_1
=\frac{2\beta_2}{\beta_1}.
\end{aligned}
$$
Note that $\sigma[\Delta_{G_0}]=-l(l+1)\I[l]$, so for each label of representation $l\in\hN$ we solve, with the convention that the Hermitian part of $A_1(y;x,l)$ is positive definite for $l$ large,
\begin{equation}\label{AsSystem1}
\begin{aligned}
a_1(y;x,l)&=\frac{\beta_2(x,l)}{\beta_1(y;x)}-\tilde{\lambda}_1(y;x,l)\\
A_1(y;x,l)&=\frac{\beta_2(x,l)}{\beta_1(y;x)}+\tilde{\lambda}_1(y;x,l).
\end{aligned}
\end{equation}
Here we set
$$
\tilde{\lambda}_1(y;x,l)
:=\frac{\sqrt{\beta_2(x,l)^2+\beta_1(y;x)l(l+1)}}{\beta_1(y;x)}
:=\sqrt{\frac{l(l+1)}{\beta_1(y;x)}}\sum_{k=0}^\infty\binom{k}{1/2}\left(\frac{\beta_2(x,l)^2}{\beta_1(y;x)l(l+1)}\right)^k.
$$
At this level there is no issue of commutativity, as the matrix $\beta_2(x,l)$ commutes with $\beta_1(y;x)l(l+1)$, which is simply a scaling transformation on each $\Hh[l]$. Corollary \ref{CoroPSSymbol} immediately implies that the power series in the definition of square root sums to a symbol of class $\mathcal{A}^0_{s-1.5}$. Thus the symbols $\tilde{\lambda}_1(y;x,l)$, $a_1(y;x,l)$ and $A_1(y;x,l)$ are indeed of class $\mathcal{A}^1_{s-1.5}$. This justifies our search for symbols $a,A$ up to highest order. 

Now the equations for $a_0,A_0$ read
$$
\begin{aligned}
a_1A_0+a_0A_1
&=\partial_yA_1-\sum_{\mu\in\Ind}\Df_{\mu}a_1\partial_{\mu}A_1\\
a_0+A_0
&=-\frac{\beta_3}{\beta_1}.
\end{aligned}
$$
The issue here is that the symbol $\partial_\mu A_1$ does not commute with the symbol $\beta_2$ anymore, though all other terms do. Fortunately, we may content ourselves with solutions up to error of order $-1$. Since $A_1$ and $a_0$ are of order 1 and 0 in the sense of Definition \ref{2Order}, we find that $[a_0,A_1]$ should be a symbol of order $0$, thus negligible. Replacing $a_0A_1$ by $A_1a_0+[a_0,A_1]$, we can then re-write the system as
$$
\begin{aligned}
a_1A_0+A_1a_0
&=\partial_yA_1-\sum_{\mu\in\Ind}\Df_{\mu}a_1\partial_{\mu}A_1
\mod \mathcal{A}^{0}_{s-2.5}\\
a_0+A_0
&=-\frac{\beta_3}{\beta_1}.
\end{aligned}
$$
Multiplying the second equation by $A_1$ from the left, subtracting it from the first, we can explicitly solve $A_0$, and then $a_0$, up to an error in $\mathcal{A}^{-1}_{s-2.5}$.

It is thus safe to choose
\begin{equation}\label{AsSystem0}
\begin{aligned}
a_0&=\frac{\tilde{\lambda}_1^{-1}}{2}\partial_yA_1
+\frac{\tilde{\lambda}_1^{-1}\beta_3\beta_2-2\beta_1\beta_3+1}{2\beta_1^2}
-\frac{\tilde{\lambda}_1^{-1}}{2}\sum_{\mu\in\Ind}\Df_{\mu}a_1\partial_{\mu}A_1\\
A_0&=-\frac{\tilde{\lambda}_1^{-1}}{2}\partial_yA_1
-\frac{\tilde{\lambda}_1^{-1}\beta_3\beta_2+1}{2\beta_1^2}
+
\frac{\tilde{\lambda}_1^{-1}}{2}\sum_{\mu\in\Ind}\Df_{\mu}a_1\partial_{\mu}A_1.
\end{aligned}
\end{equation}
Applying Proposition \ref{PSSymbol} and the Corollary, we justify $a_0,A_0\in\mathcal{A}^0_{s-2.5}$, and this final result justifies the legitimacy of neglecting $[a_0,A_1]$.

To sum up, we have successfully obtained the desired factorization (\ref{Factorization}), with $a,A$ given by (\ref{AsSystem1})-(\ref{AsSystem0}), and 
$$
\|R_0\|_{H^{s+0.5-\delta}\to H^s},\,
\|R_1\|_{H^{s-0.5-\delta}\to H^s}
\lesssim K\big(\|\zeta\|_{H^{s+0.5}}\big)\|\phi\|_{H^s}.
$$

The next step is to find out what the boundary value $\partial_yw^\sharp\big|_{y=0}$ looks like. We state the following lemma, whose proof is identical to that of Proposition 3.19. in \cite{ABZ2011}:
\begin{lemma}[Regularity for Evolutionary Elliptic Para-Differential Equation]\label{MaxReg}
Suppose $r>1$, the $y$-dependent symbol $a\in C^1_y\mathcal{A}^1_{r}(\SU(2))$ for $y\in[-1/2,0]$, such that the Hermitian part of $a$ is $\geq c|l|$ for some $c>0$ when $l$ is sufficiently large. Suppose $N<0$, $W\in C^0_yH^{-N}_x$ solves the evolutionary elliptic para-differential equation
$$
\partial_yW+T_aW=f\in C^0_yH^{s}_x,\quad s\in\mathbb{R},
$$
then in fact $W(0)\in H^{s+1-\varepsilon}$:
$$
\|W(0)\|_{H^{s+1-\varepsilon}}
\lesssim_{\varepsilon}\|f\|_{C^0_yH^{s}_x}+\|w\|_{C^0_yH^{-N}_x}.
$$
\end{lemma}
With the factorization (\ref{Factorization}), we set $\tilde W=(\partial_y-T_A)W^\sharp$ and find
$$
\partial_y\tilde W+T_a\tilde W=R_0\tilde W+R_0\partial_y\tilde W
\mod C^0_yH^{2s-1.5}.
$$
From the expression $a=a_1+a_0$, where $a_1,a_0$ are as in (\ref{AsSystem1})-(\ref{AsSystem0}), we find that $a$ in fact depends smoothly on $y$. On the other hand, $\tilde W\in C^0_yH^{s-1}_x$, with norm controlled in terms of $\|\zeta\|_{H^{s+0.5}}$ and $\|\phi\|_{H^s}$, so we can apply Lemma \ref{MaxReg} to conclude
\begin{equation}\label{AUy=0}
\left\|\partial_yW^\sharp\big|_{y=0}-T_AW^\sharp\big|_{y=0}\right\|_{H^{s+0.5}}
\lesssim K\big(\|\zeta\|_{H^{s+0.5}}\big)\|\phi\|_{H^s},
\end{equation}
where 
$$
A=A_1+A_0,
$$
with $A_1$ given in (\ref{AsSystem1}) and $A_0$ given in (\ref{AsSystem0}). All the symbols and functions of concern are $\mathbf{T}_3$-invariant.

\subsection{Concluding the Proof of Theorem \ref{Thm1}}
We are at the stage to complete Task \ref{TaskDN}. With a little abuse of notation, if $b=\sum_{j}b^jX_j$ is a vector field, then we write
$$
T_b\cdot\nabla_{G_0}=\sum_{j}T_{b^j}X_j.
$$
Thus we compute, with Bony's para-linearization theorem (and still the ``mod" notation in \ref{Mod}),
\begin{equation}\label{GoodUnknownTemp4}
\begin{aligned}
(\rho^\sharp)^2\mathcal{N}_\zeta\Psi^\sharp
&=\beta_1\partial_y\Psi^\sharp-\nabla_{G_0}\zeta^\sharp\cdot\nabla_{G_0}\Psi^\sharp
\quad\text{(precise equality)}\\
&=T_{\beta_1}\partial_y\Psi^\sharp
+2T_{\partial_y\Psi^\sharp\nabla_{G_0}\zeta^\sharp}\cdot\nabla_{G_0}\zeta^\sharp
-T_{\nabla_{G_0}\zeta^\sharp}\cdot\nabla_{G_0}\Psi^\sharp-T_{\nabla_{G_0}\Psi^\sharp}\cdot\nabla_{G_0}\zeta^\sharp
\mod C_y^0H^{s+0.5}_x
\end{aligned}
\end{equation}
Substituting in $\Psi^\sharp=W^\sharp+T_{1/\rho^\sharp}T_{\mathfrak{B}^\sharp}\zeta^\sharp$, we find that the second and third term in (\ref{GoodUnknownTemp4}) cancel out the $\nabla_{G_0}\zeta^\sharp$ within them:
$$
T_{\partial_y\Psi^\sharp\nabla_{G_0}\zeta^\sharp}\cdot\nabla_{G_0}\zeta^\sharp
-T_{\nabla_{G_0}\zeta^\sharp}\cdot\nabla_{G_0}\Psi^\sharp
=-T_{\nabla_{G_0}\zeta^\sharp}\cdot\nabla_{G_0}W^\sharp
\mod C_y^0H^{s+0.5}_x.
$$
Thus using the fact $\mathfrak{B}^\sharp,\partial_y\mathfrak{B}^\sharp\in C^0_yH^{s+0.5}_{x}$, (\ref{GoodUnknownTemp4}) reduces to
$$
(\rho^\sharp)^2\mathcal{N}_\zeta\Psi^\sharp
=T_{\beta_1}\partial_yW^\sharp-T_{\nabla_{G_0}\zeta^\sharp}\cdot\nabla_{G_0}W^\sharp
-T_{\nabla_{G_0}\Psi^\sharp-\partial_y\Psi^\sharp\nabla_{G_0}\zeta^\sharp}\cdot\nabla_{G_0}\zeta^\sharp
\mod C_y^0H^{s+0.5}_x.
$$
Using (\ref{AUy=0}), substituting in (\ref{AsSystem1})-(\ref{AsSystem0}) at $y=0$, we obtain the desired result: the Dirichlet-Neumann operator ${D[\zeta]\phi}$ is lifted to $\SU(2)$ as
\begin{equation}\label{DNParaLin}
\big(D[\zeta]\phi\big)^\sharp
=T_{\lambda}\big(\phi^\sharp-T_{\mathfrak{b}^\sharp}\zeta^\sharp\big)
-T_{\mathfrak{v}^\sharp}\cdot\nabla_{G_0}\zeta^\sharp
\mod H^{s+0.5}_x.
\end{equation}
Here with $\beta_1,\beta_2,\beta_3$ given in (\ref{Beta13})-(\ref{Beta2}), $\tilde{\lambda}_1$ and $A_0$ given in (\ref{AsSystem1})-(\ref{AsSystem0}) (evaluated at $y=0$), we have
\begin{equation}\label{DNlambda}
\begin{aligned}
\lambda(x,l)
&=\frac{\beta_1(x)\tilde{\lambda}_1(x,l)}{(1+\zeta^\sharp)^2}
+\frac{\beta_1(x)A_0(x,l)}{(1+\zeta^\sharp)^2}\\
&=\frac{\sqrt{\beta_2(x,l)^2+\beta_1(x)l(l+1)}}{(1+\zeta^\sharp)^2}
+\frac{\beta_1(x)A_0(x,l)}{(1+\zeta^\sharp)^2}\\
&=:\lambda^{(1)}(x,l)+\lambda^{(0)}(x,l).
\end{aligned}
\end{equation}
The quantities $\mathfrak{b},\mathfrak{v}$ on $\mathbb{S}^2$ are
\begin{equation}\label{bvExpression}
\begin{aligned}
\mathfrak{b}
&=\partial_y\Psi\big|_{y=0}
=\left(1+\frac{|\nabla_{0}\zeta|^2}{(1+\zeta)^{2}}\right)^{-1}\left(
{D[\zeta]\phi}+\frac{\nabla_{0}\zeta\cdot\nabla_{0}\phi}{(1+\zeta)^{2}}\right),\\
\\
\mathfrak{v}
&=\frac{\nabla_0\Psi-\partial_y\Psi\nabla_0\zeta}{\rho^2}\Bigg|_{y=0}
=\frac{\nabla_0\phi-\mathfrak{b}\nabla_0\zeta}{(1+\zeta)^2}.
\end{aligned}
\end{equation}
Note that we used $T_{\beta_2}=T_{\nabla_{G_0}\zeta^\sharp}\cdot\nabla_{G_0}$ and $\beta_1(y;x)A_1(y;x,l)-\beta_2(x,l)=\tilde{\lambda}_1(y;x,l)$. When $\zeta$ is close to 0, the error term is obviously quadratic in $(\zeta,\phi)\in H^{s+0.5}\times H^s$.

This finishes the proof of Theorem \ref{Thm1}.

\section{Para-differential Form of the System}\label{5}
In this section, we conclude the proof of Theorem \ref{Thm2} using the results we have obtained so far. Although the method is quite standard, the symmetrization process is more complicated than that in \cite{ABZ2011} due to non-commutativity of symbolic calculus on $\SU(2)$. We thus mainly focus on what is different from \cite{ABZ2011} due to this non-commutativity, and sketch everything else that is parallel.

\subsection{Para-linearizing the Full System}
We start by deriving the para-linearized form of the full system (\ref{EQ}). We still write $\rho=1+\zeta$, $\beta_1=\rho^2+|\nabla_0\zeta|_{g_0}^2$ at the level set $\{y=0\}$. As a consequence of the para-linearization formula for the Dirichlet-Neumann operator, we already know that 
$$
\begin{aligned}
\partial_t\zeta^\sharp
&=\big(D[\zeta]\phi\big)^\sharp\\
&=T_{\lambda}\big(\phi^\sharp-T_{\mathfrak{b}^\sharp}\zeta^\sharp\big)
-T_{\mathfrak{v}^\sharp}\cdot\nabla_{G_0}\zeta^\sharp
+f_1(\zeta,\phi),
\end{aligned}
$$
where $\lambda,\mathfrak{b},\mathfrak{v}$ are as in (\ref{DNlambda})-(\ref{bvExpression}),  $f_1(\zeta,\phi)$ is $\mathbf{T}_3$-invariant and 
$$
\|f_1(\zeta,\phi)\|_{H_x^{s+0.5}}
\lesssim \left(\|\zeta\|_{H_x^{s+0.5}}+\|\phi\|_{H_x^s}\right)^2.
$$

To proceed further, we introduce the actual \emph{good unknown} of interest, being the boundary value of $W^\sharp$ in Proposition \ref{GoodUnknown} modulo $H_x^{s+0.5}$:
\begin{equation}\label{ActualGUknown}
w^\sharp:=\phi^\sharp-T_{\mathfrak{b}^\sharp}\zeta^\sharp.
\end{equation}
Then with $\lambda,\mathfrak{b},\mathfrak{v}$ as in (\ref{DNlambda})-(\ref{bvExpression}), we find that the mapping
$$
\left(\begin{matrix}
\zeta^\sharp \\
\phi^\sharp
\end{matrix}\right)
\to
\left(\begin{matrix}
\zeta^\sharp \\
w^\sharp
\end{matrix}\right)
$$
is a diffeomorphism in a neighbourhood of $(0,0)\in H^{s+0.5}_x\times H^s_x$. In fact, the mapping $\phi^\sharp\to\phi^\sharp-T_{\mathfrak{b}^\sharp}\zeta^\sharp$ is linear in $\phi^\sharp$, and estimates in Section \ref{2} on the para-product operator imply
$$
\|T_{\mathfrak{b}^\sharp}\zeta^\sharp\|_{H^s_x}
\lesssim \|\zeta\|_{H_x^{s+0.5}}\|\phi\|_{H_x^s}.
$$
Thus when $\zeta$ is close to 0 in $H^{s+0.5}_x$, the linear operator $\phi^\sharp\to\phi^\sharp-T_{\mathfrak{b}^\sharp}\zeta^\sharp$ is invertible, and the inverse is smooth in $\zeta^\sharp$. Consequently, the passage of unknown from $(\zeta^\sharp,\phi^\sharp)$ to $(\zeta^\sharp,w^\sharp)$ does not lose any information. It then suffices to derive an equation governing $(\zeta^\sharp,w^\sharp)$.

To derive the evolution equation for $w^\sharp$, we need to study the terms in $\partial_t\phi$ separately. We start with the mean curvature.

\begin{proposition}\label{H(zeta)}
Suppose $s>3$, $\zeta\in H^{s+0.5}$. Still write $\rho=1+\zeta$, and $\beta_1$ for the lift of $\rho^2+|\nabla_0\zeta|_{g_0}^2$. The mean curvature $H(\zeta)$ admits a para-linearization formula
$$
\big(H(\zeta)\big)^\sharp=-T_{h}\zeta^\sharp
\mod H_x^{2s-3}.
$$
Here $h=h^{(2)}+h^{(1)}$ is a classical differential symbol, with $h^{(2)}=(\rho^\sharp)^3\beta_1^{-3/2}(\lambda^{(1)})^2$ being a multiplication of $(\lambda^{(1)})^2$, and $h^{(1)}$ takes the form
$$
h^{(1)}
=\sum_{j=1}^3 h_{1j}\big(\zeta^\sharp,\nabla_{G_0}\zeta^\sharp,\nabla_{G_0}^2\zeta^\sharp\big)\sigma[X_j]
+h_{10}\big(\zeta^\sharp,\nabla_{G_0}\zeta^\sharp,\nabla_{G_0}^2\zeta^\sharp\big),
$$
where the coefficients $h_{1j}$ are all $C^\infty$ in its arguments.
\end{proposition}
\begin{proof}
We are going to use the first variation formula to compute the mean curvature. The area of $M_\zeta$ is 
$$
A(\zeta)=\int_{\mathbb{S}^2}\rho\sqrt{\rho^2+|\nabla_0\zeta|_{g_0}^2}d\mu_0.
$$
Given any $\zeta_1\in C^\infty(\mathbb{S}^2)$, we consider the slightly varied surfaces $M_{\zeta+\theta\zeta_1}$, i.e. the surfaces given by the graphs $x\to (1+\zeta(x)+\theta\zeta_1(x))N_0(x)$. Recall that under the coordinate system $(x,y)$ introduced in the previous section, the hypersurface $y=0$ corresponds to the distorted sphere $M_\zeta$, and the vector field
$$
\nabla_{\mathrm{E}}y=-\frac{\nabla_0\zeta}{\rho^{2}}
+\left(1+\frac{|\nabla_0\zeta|_{g_0}^2}{\rho^{2}}\right)\partial_y
$$
along $\{y=0\}$ is perpendicular to the hypersurface. Thus $$
N(\iota)=\frac{\nabla_{\mathrm{E}}y}{|\nabla_{\mathrm{E}}y|}
$$
is the unit outer normal vector field along $M_\zeta$. Note that
$$
|\nabla_{\mathrm{E}}y|=\sqrt{1+\frac{|\nabla_0\zeta|_{g_0}^2}{\rho^2}}.
$$
Under this choice of coordinate, the variational vector field of this variation is $\zeta_1\partial_y\big|_{y=0}$, so 
$$
N(\iota)\cdot \big(\zeta_1\partial_y)
=\frac{\zeta_1\rho}{\sqrt{\rho^2+|\nabla_0\zeta|_{g_0}^2}}. 
$$

The first variation formula of area then gives
$$
\begin{aligned}
\frac{d}{d\theta}A(\zeta+\theta\zeta_1)\Bigg|_{\theta=0}
=\int_{\mathbb{S}^2} H(\zeta)\zeta_1\rho^2d\mu_0.
\end{aligned}
$$
On the other hand, a direct computation with integration by parts gives that the left-hand-side in the first variation formula reads
$$
\begin{aligned}
\frac{d}{d\theta}A(\zeta+\theta\zeta_1)\Bigg|_{\theta=0}
&=\int_{\mathbb{S}^2}
\left(-\rho\Delta_{g_0}\zeta
+\frac{\rho\mathrm{Hess}_{g_0}(\zeta)(\nabla_0\zeta,\nabla_0\zeta)}{\rho^2+|\nabla_0\zeta|_{g_0}^2}
+2\rho^2\right)
\frac{\zeta_1}{\sqrt{\rho^2+|\nabla_0\zeta|_{g_0}^2}}d\mu_0,
\end{aligned}
$$
Thus the mean curvature operator reads
\begin{equation}\label{H(zeta)Expr}
H(\zeta)
=\frac{1}{\rho\sqrt{\rho^2+|\nabla_0\zeta|_{g_0}^2}}\left(\Delta_{g_0}\zeta
-\frac{\mathrm{Hess}_{g_0}(\zeta)(\nabla_0\zeta,\nabla_0\zeta)}{\rho^2+|\nabla_0\zeta|_{g_0}^2}\right)
+\frac{2}{\rho\sqrt{\rho^2+|\nabla_0\zeta|_{g_0}^2}}.
\end{equation}

We lift (\ref{H(zeta)Expr}) to $\SU(2)$. Recalling that $\sum_{j=1}^3X_j\zeta^\sharp\sigma[X_j]$, where the left-invariant vector fields $X_i$ are as in (\ref{LISU(2)}), we compute the lift of the Hessian form as
$$
\begin{aligned}
\mathrm{Hess}_{G_0}(\zeta^\sharp)(\nabla_{G_0}\zeta^\sharp,\nabla_{G_0}\zeta^\sharp)
&=\sum_{i,j=1}^3\left(X_iX_j\zeta^\sharp-(\nabla_{X_i}X_j)\zeta^\sharp\right)\cdot X_i\zeta^\sharp\cdot X_j\zeta^\sharp\\
&=\Op(\beta_2(x,l)^2)\zeta^\sharp.
\end{aligned}
$$
Here we used that $\nabla_{X_i}X_j=[X_i,X_j]/2$ (with respect to the bi-invariant metric) on the Lie group $\SU(2)$. Recalling the symbol $\lambda^{(1)}=\sqrt{\beta_2^2+\beta_1l(l+1)}/(\rho^\sharp)^2$ in (\ref{DNlambda}), we find that the first term in (\ref{H(zeta)Expr}) are lifted as
\begin{equation}\label{H(zeta)2nd}
\begin{aligned}
\frac{\Delta_{G_0}\zeta^\sharp}{\rho^\sharp(x)\beta_1(x)^{1/2}}
-\frac{\mathrm{Hess}_{G_0}(\zeta^\sharp)(\nabla_{G_0}\zeta^\sharp,\nabla_{G_0}\zeta^\sharp)}{\rho^\sharp(x)\beta_1(x)^{3/2}}
&=\Op\left(\frac{-l(l+1)}{\rho^\sharp(x)\beta_1(x)^{1/2}}
-\frac{\beta_2^2(x,l)}{\rho^\sharp(x)\beta_1(x)^{3/2}}\right)\zeta^\sharp\\
&=-\Op\big(h^{(2)}(x,l)\big)\zeta^\sharp.
\end{aligned}
\end{equation}
By Bony's para-linearization theorem, (\ref{H(zeta)2nd}) is then para-linearized as (with the ``mod" notation in \ref{Mod} used)
$$
-T_{h^{(2)}}\zeta^\sharp
+T_{\tilde h^{(1)}}\cdot\nabla_{G_0}\zeta^\sharp
+T_{\tilde h_0}\zeta^\sharp
\mod H^{2s-3}_x,
$$
where $\tilde h^{(1)},\tilde h_0$ are $C^\infty$ function of $\big(\zeta^\sharp,\nabla_{G_0}\zeta^\sharp,\nabla_{G_0}^2\zeta^\sharp\big)$. Finally, the last term in (\ref{H(zeta)Expr}) is also lifted and para-linearized similarly. This completes the proof.
\end{proof}

We are now at the place to para-linearize the second order terms in the expression of $\partial_t\phi$, and derive the evolution equation of $w^\sharp$.
\begin{proposition}\label{Quad}
Suppose $s>3$ and $\zeta\in H^{s+0.5}$, $\phi\in H^s$. The evolution of $w^\sharp$ is given by
$$
\partial_tw^\sharp
=-T_{\mathfrak{v}^\sharp}\cdot\nabla_{G_0}w^\sharp
-T_{h}\zeta^\sharp
+f_2(\zeta,\phi).
$$
The error term satisfies 
$$
\|f_2(\zeta,\phi)\|_{H^{s}_x}\leq K\big(\|\zeta\|_{H^{s+0.5}_x}\big)\|\phi\|_{H^s_x},
$$
where $K$ is an increasing function vanishing linearly at zero.
\end{proposition}
\begin{proof}
We start by looking at the quadratic terms in the equation for $\partial_t\phi$. By a direct computation using our coordinate system $(x,y)$ introduced in last section, we argue just as in Lemma 3.26 of \cite{ABZ2011} as follows. Still write $\rho=1+\zeta$. Then with
$$
F(\mu,a,b,c)=\frac{\left(\mu^2a+b\cdot c\right)^2}{2\mu^2\big(\mu^2+|b|_{g_0}^2\big)},
$$
the quadratic part in the right-hand-side of the equation for $\phi$ in (\ref{EQ}) may be written as
\begin{equation}\label{RHSphi}
\begin{aligned}
-\frac{|\nabla_0\phi|_{g_0}^2}{2\rho^2}
+\frac{\left(\rho^2D[\zeta]\phi+\nabla_0\zeta\cdot\nabla_0\phi\right)^2}{2\rho^2\big(\rho^2+|\nabla_0\zeta|_{g_0}^2\big)}
=-\frac{|\nabla_0\phi|_{g_0}^2}{2\rho^2}
+F(\rho,D[\zeta]\phi,\nabla_0\zeta,\nabla_0\phi).
\end{aligned}
\end{equation}
Bony's para-linearization theorem implies that
\begin{equation}\label{|Dphi|^2}
\begin{aligned}
\left(-\frac{|\nabla_0\phi|_{g_0}^2}{2\rho^2}\right)^\sharp
&=-T_{(\rho^\sharp)^{-2}\nabla_{G_0}\phi^\sharp}\cdot\nabla_{G_0}\phi^\sharp\\
&=-T_{(\rho^\sharp)^{-2}\mathfrak{v}^\sharp}\cdot\nabla_{G_0}\phi^\sharp
-T_{(\rho^\sharp)^{-2}\mathfrak{b}^\sharp\nabla_{G_0}\zeta^\sharp}\cdot\nabla_{G_0}\phi^\sharp
\mod H^{2s-3}_x,
\end{aligned}
\end{equation}
and (here subscript for $F$ stands for partial differentiation)
$$
\big(F(\rho,D[\zeta]\phi,\nabla_0\zeta,\nabla_0\phi)\big)^\sharp
=T_{F_\mu^\sharp}\rho^\sharp
+T_{F_a^\sharp}\big(D[\zeta]\phi\big)^\sharp
+T_{F_b^\sharp}\cdot\nabla_{G_0}\zeta^\sharp
+T_{F_c^\sharp}\cdot\nabla_{G_0}\phi^\sharp
\mod H^{2s-3}_x.
$$
Since $\rho\in H^{s+0.5}_x$, we find $T_{F_\mu^\sharp}\rho^\sharp\in H^{s+0.5}_x$ and thus can be disregarded. At $(\mu,a,b,c)=(\rho,D[\zeta]\phi,\nabla_0\zeta,\nabla_0\phi)$, we compute, recalling (\ref{bvExpression}),
$$
\begin{aligned}
F_a
&=\frac{\mu^2a+b\cdot c}{\mu^2+|b|_{g_0}^2}
=\mathfrak{b},\\
F_b
&=\frac{\mu^2a+b\cdot c}{\mu^2\big(\mu^2+|b|_{g_0}^2\big)}
\left(c-\frac{(\mu^2a+b\cdot c)b}{\mu^2+|b|_{g_0}^2}\right)
=\mathfrak{bv},\\
F_c
&=\frac{\mu^2a+b\cdot c}{\mu^2\big(\mu^2+|b|_{g_0}^2\big)}b
=\frac{\mathfrak{b}\nabla_0\zeta}{\rho^2}.
\end{aligned}
$$
Thus
\begin{equation}\label{F^sharp}
\big(F(\rho,D[\zeta]\phi,\nabla_0\zeta,\nabla_0\phi)\big)^\sharp
=T_{\mathfrak{b}^\sharp}\big(D[\zeta]\phi\big)^\sharp
+T_{\mathfrak{b}^\sharp}T_{\mathfrak{v}^\sharp}\cdot\nabla_{G_0}\zeta^\sharp
+T_{(\rho^\sharp)^{-2}\mathfrak{b}^\sharp\nabla_{G_0}\phi^\sharp}\cdot\nabla_{G_0}\phi^\sharp
\mod H^{s}_x.
\end{equation}
Combining (\ref{|Dphi|^2}) and (\ref{F^sharp}), after suitable cancellation and commutation (causing acceptable error), we find that (\ref{RHSphi}) lifts to $\SU(2)$ as 
$$
-T_{\mathfrak{v}^\sharp}\cdot\nabla_{G_0}w^\sharp
+T_{\mathfrak{b}^\sharp}\big(D[\zeta]\phi\big)^\sharp
\mod H_x^{s}.
$$
Consequently, noticing that $\partial_tw^\sharp=\partial_t\phi^\sharp-T_{\partial_t\mathfrak{b}^\sharp}\phi^\sharp-T_{\mathfrak{b}^\sharp}\partial_t\zeta^\sharp$, we find
$$
\partial_t w^\sharp
=-T_{\mathfrak{v}^\sharp}\cdot\nabla_{G_0}w^\sharp
-T_{h}\zeta^\sharp
+T_{\partial_t\mathfrak{b}^\sharp}\phi^\sharp
\mod H_x^{s}.
$$
The error caused by para-linearization is quadratic in $\zeta$ and $\phi$, since the para-linearization already includes the linear part of the nonlinear expression. Finally, the regularity of $T_{\partial_t\mathfrak{b}^\sharp}\phi^\sharp$ is obtained by applying Proposition \ref{RegDN}-\ref{LinDN}, identically as in Lemma 3.27. of \cite{ABZ2011}. This finishes the proof.
\end{proof}

Just as in \cite{ABZ2011}, we obtain the para-linearization of system (\ref{EQ}), lifted to $\SU(2)$:
\begin{equation}\label{EQPara}
\begin{aligned}
\partial_t\zeta^\sharp
&=T_\lambda w^\sharp-T_{\mathfrak{v}^\sharp}\cdot\nabla_{G_0}\zeta^\sharp+f_1(\zeta,\phi)\\
\partial_tw^\sharp&=-T_{\mathfrak{v}^\sharp}\cdot\nabla_{G_0}w^\sharp
-T_{h}\zeta^\sharp
+f_2(\zeta,\phi),
\end{aligned}
\end{equation}
where the errors $f_1,f_2$ takes the explicit form
\begin{equation}\label{f1f2}
\begin{aligned}
f_1(\zeta,\phi)
&=\big(D[\zeta]\phi\big)^\sharp-T_\lambda\big(w^\sharp-T_{\mathfrak{v}^\sharp}\cdot\nabla_{G_0}\zeta^\sharp\big)\\
f_2(\zeta,\phi)
&=\left(-\frac{|\nabla_0\phi|_{g_0}^2}{2\rho^2}
+\frac{\left(\rho^2D[\zeta]\phi+\nabla_0\zeta\cdot\nabla_0\phi\right)^2}{2\rho^2\big(\rho^2+|\nabla_0\zeta|_{g_0}^2\big)}
+H(\zeta)\right)^\sharp
+T_{\mathfrak{v}^\sharp}\cdot\nabla_{G_0}w^\sharp
+T_{h}\zeta^\sharp
-T_{\partial_t\mathfrak{b}^\sharp}w^\sharp,
\end{aligned}
\end{equation}
and are affordable and in fact quadratic:
$$
\|f_1(\zeta,\phi)\|_{H_x^{s+0.5}}+\|f_2(\zeta,\phi)\|_{H_x^{s}}
\lesssim\left(\|\zeta\|_{H_x^{s+0.5}}+\|\phi\|_{H_x^s}\right)^2.
$$
We already know that (\ref{EQPara}) is equivalent to (\ref{EQ}) at least when $\|\zeta(t)\|_{H^{s+0.5}_x}$ remains sufficiently small. Thus we have successfully derived the para-linear form of (\ref{EQ}).

\subsection{Adjoint of Para-differential Operators}
In order to symmetrize the system (\ref{EQPara}), we need to understand the adjoint properties of the Dirichlet-Neumann operator and the mean curvature operator. We start from the former. We still employ the notation $\rho=1+\zeta$, and define the symbol $\lambda$ as in (\ref{DNlambda}). The following proposition reflects the self-adjointness of the Dirichlet-Neumann operator.

\begin{proposition}\label{AdjDN}
Still fix $s>3$ and $\zeta\in H^{s+0.5}$, $\phi\in H^s$. Then the symbol $\lambda$ of the Dirichlet-Neumann operator satisfies
\begin{equation}\label{AdjT_Lambda}
(\rho^\sharp)^{-2}\lambda^{\bullet;1,\mathscr{Q}}=
\lambda\#_{1;\mathscr{Q}}(\rho^\sharp)^{-2}
\mod\mathcal{A}^{-1}_{s-2.5}
\end{equation}
Here the operations $\#_{r;\mathscr{Q}}$ and ${\bullet;r,\mathscr{Q}}$ are defined in (\ref{ParaCompoSU2})-(\ref{ParaAdjSU2}). Thus the para-differential operator $T_\lambda$ satisfies
$$
T_{(\rho^\sharp)^{-2}}T_\lambda^*=T_\lambda T_{(\rho^\sharp)^{-2}}\mod\Op\Sigma^{-0.5}_{<1/2}.
$$
\end{proposition}

\begin{proof}
If we directly manipulate with (\ref{AsSystem0}) and (\ref{DNlambda}), the algebra in the proof of (\ref{AdjT_Lambda}) would be extremely involved. Thus we proceed differently by recalling the geometric interpretation of the Dirichlet-Neumann operator. 

Assume for the moment that $\zeta$ is \emph{smooth}. Recall that the linear operator taking the Dirichlet boundary value $\phi$ to the Neumann boundary value $\bar\nabla\Phi\cdot N(\iota)=\big(N_0\cdot N(\iota)\big)D[\zeta]\phi$ on $M_\zeta$ is self-adjoint with respect to the surface measure of $M_\zeta$; this is the starting point of our argument. Recall that $N_0\cdot N(\iota)
=\rho/\sqrt{\rho^2+|\nabla_0\zeta|_{g_0}^2}$, while the surface measure of $M_\zeta$ is $d\mu(\iota)=\rho\sqrt{\rho^2+|\nabla_0\zeta|_{g_0}^2}d\mu_0$. Consequently, we find, for smooth functions $\phi$, $\phi_1$ defined on $\mathbb{S}^2$,
$$
\int_{\mathbb{S}^2}D[\zeta]\phi\cdot\phi_1 \rho^2d\mu_0
=\int_{\mathbb{S}^2}\phi\cdot D[\zeta]\phi_1\cdot \rho^2d\mu_0.
$$
Lifting to $\SU(2)$, this is exactly
$$
\int_{\SU(2)}\mathfrak{D}\phi^\sharp\cdot\phi_1^\sharp \cdot(\rho^\sharp)^2d\mu(G_0)
=\int_{\SU(2)}\phi^\sharp\cdot \mathfrak{D}\phi_1^\sharp\cdot (\rho^\sharp)^2d\mu(G_0),
$$
where $\mathfrak{D}\phi^\sharp:=\big(D[\zeta]\phi\big)^\sharp$. As a result, we have
\begin{equation}\label{DNAdjTemp1}
(\rho^\sharp)^{-2}\mathfrak{D}^*
=\mathfrak{D}\circ( \rho^\sharp)^{-2}.
\end{equation}
On the other hand, if $\zeta$ is smooth, the para-linearization formula of the Dirichlet-Neumann operator shows that in fact
$$
\mathfrak{D}=\Op(\lambda)+\Op\mathscr{S}^{-1}_{1,0}.
$$
Converting (\ref{DNAdjTemp1}) to the symbolic side modulo symbols of order less than $-1$, this gives
\begin{equation}\label{DNAdjTemp2}
(\rho^\sharp)^{-2}\lambda^{\bullet;1,\mathscr{Q}}-
\lambda\#_{1;\mathscr{Q}}(\rho^\sharp)^{-2}
\in\mathscr{S}^{-1}_{1,0}.
\end{equation}
If we denote the left-hand-side of (\ref{DNAdjTemp2}) as $Q(\zeta;x,l)$, then (\ref{DNAdjTemp2}) of course implies
\begin{equation}\label{DNAdjTemp3}
\begin{aligned}
\big\|\Df_\mu^\alpha Q(\zeta;x,l)\big\|_{C^{s-2.5}_*}
&=O(l^{-1-|\alpha|}),
\quad
\text{for all multi-index }\alpha.
\end{aligned}
\end{equation}
This matrix norm estimate has been deduced for smooth $\zeta$. However, as we can see from the definition of $\lambda$ and symbolic operation, the concrete expression of entries of $Q(\zeta;x,l)$ only contains up to second order derivative of $\zeta$. Thus we can approximate a given $\zeta\in H^{s+0.5}$ by smooth functions and conclude that (\ref{DNAdjTemp3}) continues to hold for $\zeta\in H^{s+0.5}$. Thus for $\zeta\in H^{s+0.5}$, 
$$
(\rho^\sharp)^{-2}\lambda^{\bullet;1,\mathscr{Q}}-
\lambda\#_{1;\mathscr{Q}}(\rho^\sharp)^{-2}
$$
is a symbol of order $-1$. This proves (\ref{AdjT_Lambda}).
\end{proof}

We then proceed to the mean curvature.
\begin{proposition}\label{T_hAdjoint}
Suppose $s>3$, $\zeta\in H^s$. Then with the symbol $h=h^{(2)}+h^{(1)}$ as in Proposition \ref{H(zeta)}, there holds a precise equality
\begin{equation}\label{Adjh}
(\rho^\sharp)^{-2}h^{\bullet;2,\mathscr{Q}}
=h\#_{2,\mathscr{Q}}(\rho^\sharp)^{-2}.
\end{equation}
Thus
$$
T_{(\rho^\sharp)^{-2}}T_h^*=T_hT_{(\rho^\sharp)^{-2}} \mod\Op\Sigma^{0.5}_{<1/2}.
$$
\end{proposition}
\begin{proof}
Of course we may directly deal with the expression of $h(x,l)$. But there is a method that could tactfully avoid such lengthy computation. Consider a variation of the surface $M_\zeta$:
$$
x\to\big(1+\zeta(x)+\theta_1\zeta_1(x)+\theta_2\zeta_2(x)\big)N_0.
$$
By second variation formula for area functional, the bilinear form
$$
\frac{\partial^2 A}{\partial\theta_1\partial\theta_2}(\zeta+\theta_1\zeta_1+\theta_2\zeta_2)\Bigg|_{\theta=0}
$$
is symmetric in $\zeta_1,\zeta_2$. On the other hand, the second variation obviously equals
$$
\int_{\mathbb{S}^2}H'(\zeta)\zeta_2\cdot\zeta_1\rho^2d\mu_0
+2\int_{\mathbb{S}^2}H(\zeta)\zeta_1\zeta_2\rho d\mu_0.
$$
Thus the second order differential operator $\rho^2H'(\zeta)$ is self-adjoint with respect to $\mu_0$. Lifting to $\SU(2)$, it follows that the second order differential operator $(\rho^\sharp)^2\big(H'(\zeta)\big)^\sharp$ is self-adjoint with respect to $\mu(G_0)$. But we also notice that the symbol $h$ is nothing but the symbol of the second order differential operator $\big(H'(\zeta)\big)^\sharp$. By the symbolic calculus theorem of para-differential operators, this proves (\ref{Adjh}).
\end{proof}

\subsection{Symmetrization}\label{Symmetrization}
We are at the place to symmetrize the para-differential system (\ref{EQPara}). The general procedure is identical to that in \cite{ABZ2011} but with different computational details due to the non-flat geometry in our scenario. 

The symmetrizer that we are looking for should take the form
$$
\left(\begin{matrix}
T_p & \\
 & T_q
\end{matrix}\right),
$$
where $p\in\mathcal{A}^{1/2}_{s-2.5}$ should be close to the symbol of $|\nabla_{G_0}|^{0.5}$, and $q=q(x)$ is a scalar function. We require that for some elliptic symbol $\gamma$ of order 1.5, there holds
\begin{equation}\label{Symmetrizer}
\left(\begin{matrix}
T_p & \\
 & T_q
\end{matrix}\right)\left(\begin{matrix}
 & T_\lambda \\
-T_h & 
\end{matrix}\right)
=\left(\begin{matrix}
 & T_\gamma \\
-T_\gamma & 
\end{matrix}\right)\left(\begin{matrix}
T_p & \\
 & T_q
\end{matrix}\right)
\mod
\left(\begin{matrix}
 & \Op\Sigma^{0}_{<1/2} \\
\Op\Sigma^{0.5}_{<1/2} & 
\end{matrix}\right)
\end{equation}
The operator norms of the error term will be bounded linearly by $\|\zeta\|_{H_x^{s+0.5}}$. 

The elliptic symbol $\gamma$ will be the one that satisfy
\begin{equation}\label{T_gamma}
T_\gamma^2=\frac{T_hT_\lambda+T_\lambda^*T_h^*}{2}
\mod\Op\Sigma^{0.5}_{<1/2}.
\end{equation}
This ensures that $T_\gamma$ is approximately self-adjoint. Our symmetrization is finished by several steps.

\textbf{Step 1: solve $\gamma$.} We look for a solution $\gamma=\gamma^{(1.5)}+\gamma^{(0.5)}\in\mathcal{A}^{1.5}_{s-1.5}+\mathcal{A}^{0.5}_{s-2.5}$ of (\ref{T_gamma}). The highest order symbol $\gamma^{(1.5)}$ is just fixed as
$$
\gamma^{(1.5)}=\sqrt{\lambda^{(1)}h^{(2)}},
$$
which is Hermitian on each representation space $\Hh[l]$. Note that we used the commutativity of $h^{(2)}$ and $\lambda^{(1)}$, as shown in Proposition \ref{H(zeta)}. For the next order symbol $\gamma^{(0.5)}$, we just have to fix it by looking at the next order symbols in the composition formula of para-differential operators. That is, we require 
$$
\gamma\#_{1;\mathscr{Q}}\gamma
=\frac{h\#_{1,\mathscr{Q}}\lambda+\lambda^{\bullet;1,\mathscr{Q}}\#_{1,\mathscr{Q}}h^{\bullet;1,\mathscr{Q}}}{2},
$$
or equivalently,
$$
\begin{aligned}
\sqrt{\lambda^{(1)}h^{(2)}}\gamma^{(0.5)}+\gamma^{(0.5)}\sqrt{\lambda^{(1)}h^{(2)}}
&+\sum_{\mu\in\Ind}\Df_\mu\sqrt{\lambda^{(1)}h^{(2)}}\partial_\mu\sqrt{\lambda^{(1)}h^{(2)}}\\
&=
\frac{h\#_{1,\mathscr{Q}}\lambda
+\lambda^{\bullet;1,\mathscr{Q}}\#_{1,\mathscr{Q}}h^{\bullet;1,\mathscr{Q}}}{2}
-\lambda^{(1)}h^{(2)}.
\end{aligned}
$$
Since $[\sqrt{\lambda^{(1)}h^{(2)}},\gamma^{(0.5)}]$ is a symbol of order 1 by Proposition \ref{2OrderComm}, we can just replace $\gamma^{(0.5)}\gamma^{(1.5)}$ by $\gamma^{(1.5)}\gamma^{(0.5)}$, and solve $\gamma^{(0.5)}$ as
$$
\gamma^{(0.5)}
=\frac{(\lambda^{(1)}h^{(2)})^{-1/2}}{2}\left(
\frac{h\#_{1,\mathscr{Q}}\lambda
+\lambda^{\bullet;1,\mathscr{Q}}\#_{1,\mathscr{Q}}h^{\bullet;1,\mathscr{Q}}}{2}
-\lambda^{(1)}h^{(2)}
-\sum_{\mu\in\Ind}\Df_\mu\sqrt{\lambda^{(1)}h^{(2)}}\partial_\mu\sqrt{\lambda^{(1)}h^{(2)}}
\right).
$$

\textbf{Step 2: solve $q$}. We seek for a scalar function $q=q(x)$ defined on $\SU(2)$ such that 
\begin{equation}\label{Tq}
T_qT_hT_\lambda
=\frac{T_hT_\lambda+T_\lambda^*T_h^*}{2}\cdot T_q
\mod\Op\Sigma^{1.5}_{<1/2}.
\end{equation}
This is the step that we should employ the almost self-adjoint properties of $T_\lambda$ and $T_h$. Note that we have already proved, in Proposition \ref{AdjDN}-\ref{T_hAdjoint}, that
$$
\begin{aligned}
T_{(\rho^\sharp)^{-2}}T_\lambda^*&=T_\lambda T_{(\rho^\sharp)^{-2}}
&\mod\Op\Sigma^{-0.5}_{<1/2}\\
T_{(\rho^\sharp)^{-2}}T_h^*&=T_hT_{(\rho^\sharp)^{-2}}
&\mod\Op\Sigma^{0.5}_{<1/2}.
\end{aligned}
$$
So (\ref{Tq}) becomes 
$$
2T_qT_hT_\lambda=\left(T_hT_\lambda
+T_{(\rho^\sharp)^2\lambda}T_hT_{(\rho^\sharp)^{-2}}\right)T_q
\mod\Op\Sigma^{1.5}_{<1/2}.
$$
Here we used that $\rho\in H^{s+0.5}$, so $T_{(\rho^\sharp)^{2}}T_h=T_{(\rho^\sharp)^{2}h}\mod\Op\Sigma^{3.5-s}_{<1/2}$. Since only symbols of order $\geq2$ is of our concern in the above equality, we can simply use the composition formula of para-differential operators and then neglect symbols of order $\leq1$. Thus in order that (\ref{Tq}) holds, it suffices to require 
\begin{equation}\label{SymmTemp}
\begin{aligned}
2q\sum_{\mu\in\Ind}\Df_\mu h^{(2)}\partial_\mu \lambda^{(1)}
&=q\sum_{\mu\in\Ind}\Df_\mu h^{(2)}\partial_\mu \lambda^{(1)}\\
&\quad+2\sum_{\mu\in\Ind}\Df_\mu(\lambda^{(1)}h^{(2)})\partial_\mu q\\
&\quad+q(\rho^\sharp)^{2}\sum_{\mu\in\Ind}\Df_\mu\lambda^{(1)}\partial_\mu\left(\frac{h^{(2)}}{(\rho^\sharp)^{2}}\right)\\
&\quad+q(\rho^\sharp)^{2}\sum_{\mu\in\Ind}\lambda^{(1)}\Df_\mu h^{(2)}\partial_\mu(\rho^\sharp)^{-2}
\mod\mathcal{A}^{1}_{s-3}.
\end{aligned}
\end{equation}

Being tedious at a first glance, there are several key features of the above equality that can be exploited. Recall that $h^{(2)}=(\rho^\sharp)^3\beta_1^{-3/2}(\lambda^{(1)})^2$, and from the expression of $\lambda^{(1)}$, i.e. formula (\ref{DNlambda}), we known that $\lambda^{(1)}$ is a quasi-homogeneous symbol of order 1 in the sense of Definition \ref{QuasiHomoSym}, where the vector field is just $\nabla_{G_0}\zeta^\sharp/\beta_1(x)$. Proposition \ref{PSSymbol} thus ensures that the commutators $[\lambda^{(1)},\partial_\mu\lambda^{(1)}]$ and $[\lambda^{(1)},\Df_\mu\lambda^{(1)}]$ are rough symbols of order 1 and 0 respectively. Consequently, we find 
$$
\begin{aligned}
\Df_\mu h^{(2)}
&=\frac{2(\rho^\sharp)^3}{\beta_1^{3/2}}\Df_\mu\lambda^{(1)}\cdot\lambda^{(1)}
&\mod\mathcal{A}^{1}_{s-3},\\
\Df_\mu(\lambda^{(1)}h^{(2)})
&=\frac{3(\rho^\sharp)^3}{\beta_1^{3/2}}\Df_\mu\lambda^{(1)}\cdot(\lambda^{(1)})^2
&\mod\mathcal{A}^{1}_{s-3},\\
\partial_\mu\left(\frac{h^{(2)}}{(\rho^\sharp)^{2}}\right)
&=(\lambda^{(1)})^2\partial_\mu\big(\rho^\sharp\beta_1^{-3/2}\big)+\frac{2\rho^\sharp\lambda^{(1)}\partial_\mu\lambda^{(1)}}{\beta_1^{3/2}}
&\mod\mathcal{A}^{1}_{s-3},
\end{aligned}
$$
leading to cancellation between the left-hand-side and the first/thrid terms in the right-hand-side of (\ref{SymmTemp}) modulo symbol of order $1$. Thus (\ref{SymmTemp}) is in fact equivalent to
$$
\begin{aligned}
\frac{6(\rho^\sharp)^3}{\beta_1^{3/2}}\sum_{\mu\in\Ind}
\Df_\mu\lambda^{(1)}\cdot(\lambda^{(1)})^2\partial_\mu q
&+2q\big(\rho^\sharp\big)^2\sum_{\mu\in\Ind}\Df_\mu\lambda^{(1)}\cdot(\lambda^{(1)})^2\partial_\mu\big(\rho^\sharp\beta_1^{-3/2}\big)\\
&+\frac{2q(\rho^\sharp)^5}{\beta_1^{3/2}}\sum_{\mu\in\Ind}\Df_\mu\lambda^{(1)}\cdot(\lambda^{(1)})^2\partial_\mu(\rho^\sharp)^{-2}
=0\mod\mathcal{A}^{1}_{s-3}.    
\end{aligned}
$$
In order that this equality holds, it suffices to require 
$$
\frac{6(\rho^\sharp)^3}{\beta_1^{3/2}}\frac{\nabla_{G_0}q}{q}
+2\big(\rho^\sharp\big)^2\nabla_{G_0}\big(\rho^\sharp\beta_1^{-3/2}\big)
+\frac{2(\rho^\sharp)^5}{\beta_1^{3/2}}\nabla_{G_0}(\rho^\sharp)^{-2}=0,
$$
which can be solved as
\begin{equation}\label{Symmq}
q
={\beta_1^{3/2}}{(\rho^\sharp)^{1/3}}
=(1+\zeta^\sharp)^{1/3}\sqrt{(1+\zeta^\sharp)^2+|\nabla_{G_0}\zeta^\sharp|_{G_0}^2}.
\end{equation}

\textbf{Step 3: solve $p$.} Finally, the symbol $p=p^{(0.5)}+p^{(-0.5)}$ is solved directly from the equation 
$$
T_pT_\lambda=T_\gamma T_q
\mod\Op\Sigma^0_{<1/2},
$$
or equivalently
$$
p\#_{1;\mathscr{Q}}\lambda=\gamma\#_{1;\mathscr{Q}}q
\mod\Op\Sigma^0_{<1/2}.
$$
The principal symbol of $p$ is just 
\begin{equation}\label{Symmp}
p^{(0.5)}=\frac{(\rho^\sharp)^{1.5}}{\beta_1^{3/4}}\sqrt{\lambda^{(1)}}.
\end{equation}
The sub-principal symbol $p^{(-0.5)}$ is then fixed as
$$
p^{(-0.5)}
=(\lambda^{(1)})^{-1}\gamma^{(0.5)}q+(\lambda^{(1)})^{-1}\sum_{\mu\in\Ind}\Df_\mu\gamma^{(1.5)}\partial_\mu q.
$$
Thus we have had all symbols $\gamma,p,q$ defined and matching (\ref{Symmetrizer}). The operator norm of the error term in (\ref{Symmetrizer}) is bounded by $\|\zeta\|_{H_x^{s+0.5}}$ linearly when $\zeta\simeq0$, as a consequence of Theorem \ref{Compo2}-\ref{ParaAdj}.

\textbf{Step 4: symmetrization.} We now apply the symmetrizer $\left(\begin{matrix}
T_p & \\ & T_q\end{matrix}\right)$ to the para-differential sysmtem (\ref{EQPara}). We first notice that the expressions for $\partial_tp$ and $\partial_tq$ actually do not explicitly contain time derivative of $\zeta$ and $\phi$, since we assumed \emph{in a priori} that $(\zeta,\phi)$ solves the autonomous partial differential equation (\ref{EQ}). In order not to cause unnecessary confusion, we denote the symbols
\begin{equation}\label{DtpDtq}
a+b
:=\partial_tp^{(0.5)}+\partial_tp^{(-0.5)}
\in\mathcal{A}^{0.5}_{s-3}+\mathcal{A}^{-0.5}_{s-4},
\quad 
c:=\partial_tq\in\mathcal{A}^{0}_{s-2}.
\end{equation}
Here the regularity of $\partial_tp^{(-0.5)}$ is justified by noting that it is of the type indicated in Proposition \ref{SymbolVeryRough}.

Hence we compute
\begin{equation}\label{SymmEQTemp}
\begin{aligned}
\partial_t\left(\begin{matrix}
T_p\zeta^\sharp \\
T_qw^\sharp
\end{matrix}\right)
&=\left(\begin{matrix}
 & T_\gamma \\
-T_\gamma & 
\end{matrix}\right)\left(\begin{matrix}
T_p\zeta^\sharp \\
T_qw^\sharp
\end{matrix}\right)
-T_{\mathfrak{v}^\sharp}\cdot\nabla_{G_0}
\left(\begin{matrix}
T_p\zeta^\sharp \\
T_qw^\sharp
\end{matrix}\right)\\
&\quad
+\left(\begin{matrix}
(T_pT_\lambda-T_\gamma T_q)w^\sharp \\
(-T_qT_h+T_\gamma T_p)\zeta^\sharp
\end{matrix}\right)
+\left(\begin{matrix}
\left(T_{a+b}-\big[T_{\mathfrak{v}^\sharp}\cdot\nabla_{G_0},T_p\big]\right)\zeta^\sharp \\
\left(T_c-\big[T_{\mathfrak{v}^\sharp}\cdot\nabla_{G_0},T_q\big]\right)w^\sharp
\end{matrix}\right)
+\left(\begin{matrix}
T_pf_1(\zeta,\phi) \\
T_qf_2(\zeta,\phi)
\end{matrix}\right),
\end{aligned}
\end{equation}
where the mappings $f_1,f_2$ are as in (\ref{f1f2}). We just define the sum of last three terms as the mapping $f_3(\zeta,\phi)$.

Due to the para-differential reduction, we know that $f_1,f_2$ are controlled by $\left(\|\zeta\|_{H_x^{s+0.5}}+\|\phi\|_{H_x^s}\right)^2$ as (\ref{f1f2}) indicates. Due to (\ref{Symmetrizer}), the operator norms of $T_pT_\lambda-T_\gamma T_q$ and $-T_qT_h+T_\gamma T_p$ are controlled linearly by $\|\zeta\|_{H^{s+0.5}_x}$ when $\zeta$ is small.

It remains to show that given $\alpha\in\mathbb{R}$,
\begin{equation}\label{HsBdd1}
\big[T_{\mathfrak{v}^\sharp}\cdot\nabla_{G_0},T_p\big]
\quad\text{and}\quad
T_{a+b}
\end{equation}
both have operator norms in $\mathcal{L}(H^{\alpha+0.5},H^\alpha)$ controlled linearly by $\|\zeta\|_{H^{s+0.5}_x}+\|\phi\|_{H_x^s}$ when $(\zeta,\phi)\simeq0$, while
\begin{equation}\label{HsBdd2}
\big[T_{\mathfrak{v}^\sharp}\cdot\nabla_{G_0},T_q\big]
\quad\text{and}\quad
T_c\zeta^\sharp
\end{equation}
both have operator norms in $\mathcal{L}(H^{\alpha},H^\alpha)$ controlled linearly by $\|\zeta\|_{H^{s+0.5}_x}+\|\phi\|_{H_x^s}$ when $(\zeta,\phi)\simeq0$.

The claim for first operators in (\ref{HsBdd1})-(\ref{HsBdd2}) follows from the fact that
$\big[T_{\mathfrak{v}^\sharp}\cdot\nabla_{G_0},T_p\big]
$ and $\big[T_{\mathfrak{v}^\sharp}\cdot\nabla_{G_0},T_q\big]$
are para-differential operators of order 1/2 and 0 respectively, since commutator with a vector field does not increase the order of the operator. The control of the norms is a consequence of Corollary \ref{ParaComm}. AS for the claim for second operators in (\ref{HsBdd1})-(\ref{HsBdd2}), this is exactly the content of Lemma 4.10. of \cite{ABZ2011}:
\begin{lemma}[Equivalent form of Lemma 4.10., \cite{ABZ2011}]\label{ABZ4-10}
The para-differential operators $T_{a}$ and $T_{b}$ are both of order $0.5$: for all $\alpha\in\mathbb{R}$, there is an increasing function $K$ vanishing linearly at 0, such that
$$
\big\|T_{a}\big\|_{\mathcal{L}(H^\alpha,H^{\alpha-0.5})}
+\big\|T_{b}\big\|_{\mathcal{L}(H^\alpha,H^{\alpha-0.5})}
\leq K\big(\|\zeta\|_{H_x^{s+0.5}}\big).
$$
The para-differential operator $T_c$ has order 0:
for all $\alpha\in\mathbb{R}$, there is an increasing function $K$ anishing linearly at 0, such that
$$
\big\|T_c\big\|_{\mathcal{L}(H^\alpha,H^{\alpha})}
\leq K\big(\|\zeta\|_{H_x^{s+0.5}}\big).
$$
\end{lemma}
\begin{remark}
The only technical difficulty with the proof of this Lemma is that the symbol $b(x,l)$ is too rough in $x$. In fact, from the expression of $p^{(-0.5)}$, one concludes that $p^{(-0.5)}$ has $H^{s-1.5}$ regularity in $x$. Differentiating in $t$ reduces its regularity to $H^{s-3}\subset C_*^{s-4}$. Thus $b(x,l)$ is a symbol of order $-0.5$ in $l$ and $C_*^{s-4}$ regularity in $x$. By Proposition \ref{T_aNegIndex}, $T_{b}$ is a para-differential operator of order at most $0.5$.
\end{remark}

To summarize, introducing the diagonal unknown $\left(\begin{matrix}
T_p\zeta^\sharp \\ T_q w^\sharp \end{matrix}\right)$, the original system (\ref{EQ}) is equivalent to
\begin{equation}\label{EQSymm'}
\partial_t\left(\begin{matrix}
T_p\zeta^\sharp \\ T_q w^\sharp \end{matrix}\right)=\left(\begin{matrix}
 & T_\gamma \\
-T_\gamma & 
\end{matrix}\right)\left(\begin{matrix}
T_p\zeta^\sharp \\ T_q w^\sharp \end{matrix}\right)
-T_{\mathfrak{v}^\sharp}\cdot\nabla_{G_0} u
+f_3(\zeta,\phi),
\end{equation}
where the error $f_3(\zeta,\phi)$ collects the last three terms in (\ref{SymmEQTemp}), and $\|f_3(\zeta,\phi)\|_{H^s_x}\lesssim\left(\|\zeta\|_{H_x^{s+0.5}}+\|\phi\|_{H_x^s}\right)^2$. This is exactly (\ref{EQSymm}), We hence complete the proof of Theorem \ref{Thm2}.

\section{Energy Estimate and Local Well-posedness}\label{6}
We finally complete the energy estimate for (\ref{EQSymm}), and prove local well-posedness of the original system (\ref{EQ}). For simplicity, we sketch the proof of local well-posedness for small amplitude solutions. Since the argument is rather standard, and all the key estimates are already done in \cite{ABZ2011}, we will only present in detail the parts that are different from \cite{ABZ2011}; for those that are identical to \cite{ABZ2011}, we will directly cite the corresponding Lemmas from \cite{ABZ2011} and show how they aid our proof.

\subsection{Energy Estiamte}
In \cite{ABZ2011}, the authors used the $2s/3$'th order power of the principal symbol $\gamma$ to construct a suitable energy functional. The reason is that the Poisson bracket $\{(\gamma^{(1.5)})^{2\alpha/3},\gamma\}$ vanishes up to order $\alpha-0.5$, hence $[T_{(\gamma^{(1.5)})^{2\alpha/3}},T_\gamma]$ is in fact of order $\leq \alpha$. A similar result applies to the spherical system (\ref{EQSymm}) as well, although not as trivial as in the Euclidean case. 

\begin{proposition}\label{OrderS}
Given $\zeta\in H^{s+0.5}$, let $\gamma^{(1.5)}$ be the rough symbol as given in Subsection \ref{Symmetrization}. Then $\big[T_{(\gamma^{(1.5)})^{2\alpha/3}},T_\gamma\big]$ is a para-differential operator of order $\leq \alpha$, instead of just $\alpha+0.5$. Quantitatively, its operator norm for $H^s\to H^s$ is controlled linearly by $\|\zeta\|_{H_x^{s+0.5}}$ when $\zeta\simeq0$.
\end{proposition}
\begin{proof}
Recall from Subsection \ref{Symmetrization} that 
$$
\gamma^{(1.5)}=\sqrt{\lambda^{(1)}h^{(2)}}
=(\rho^\sharp)^{1.5}\beta_1^{-3/4}(\lambda^{(1)})^{1.5},
$$
which may be abbreviated as
$$
\gamma^{(1.5)}(x,\xi)=g(x)\kappa(\xi)^{1.5}f\left(\frac{b(x,\xi)}{\kappa(\xi)}\right),
$$
where $g(x)$ is some scalar function, $b$ is the symbol of a vector field, all depending on first order derivatives of $\zeta$; $\kappa$ is the symbol of $|\nabla_{G_0}|$, and $f(z)=(1+z^2)^{3/4}$. By our symbolic calculus formula for para-differential operators, namely Theorem \ref{Compo2}, the commutator $\big[T_{(\gamma^{(1.5)})^{2\alpha/3}},T_\gamma\big]$ in fact equals the para-differential operator corresponding to
\begin{equation}\label{GammaCommu}
\sum_{\mu\in\Ind}\Df_{\mu}(\gamma^{(1.5)})^{2\alpha/3}\partial_\mu\gamma^{(1.5)}
-\sum_{\mu\in\Ind}\Df_{\mu}\gamma^{(1.5)}\partial_\mu(\gamma^{(1.5)})^{2\alpha/3},
\end{equation}
modulo an operator of class $\Op\Sigma^{\alpha}_{<1/2}$. Obviously $\kappa(\xi)^{1.5}f\left(\frac{b(x,\xi)}{\kappa(\xi)^2}\right)$ is a quasi-homogeneous symbol of order 1.5 in the sense of Definition \ref{QuasiHomoSym}. We can now apply Proposition \ref{Dkappa}-\ref{PSSymbol} as follows.

We first compute, by Proposition \ref{Dkappa}, that
$$
\Df_\mu\kappa^{1.5}=\frac{3}{2}\kappa^{0.5}\Df_\mu\kappa
\mod\mathscr{S}^{-0.5}_{1,0},
\quad
\Df_\mu\kappa^{2\alpha/3}=\frac{2\alpha}{3}\kappa^{2\alpha/3-1}\Df_\mu\kappa
\mod\mathscr{S}^{\alpha-2}_{1,0}.
$$
Thus by Proposition \ref{PSSymbol}, together with the Leibniz property of $\Df_\mu$,
$$
\begin{aligned}
\Df_{\mu}(\gamma^{(1.5)})^{2\alpha/3}
&=\alpha g^{2\alpha/3}\kappa^{\alpha-1}f\left(\frac{b}{\kappa}\right)^{2\alpha/3}\Df_\mu\kappa
+\frac{2\alpha}{3}g^{2\alpha/3}\kappa^{\alpha-1}
f\left(\frac{b}{\kappa}\right)^{2\alpha/3-1}f'\left(\frac{b}{\kappa}\right)\Df_\mu b\\
&\quad+\frac{2\alpha}{3} g^{2\alpha/3}\kappa^{\alpha-2}f\left(\frac{b}{\kappa}\right)^{2\alpha/3-1}f'\left(\frac{b}{\kappa}\right)\Df_\mu \kappa b
&\mod \mathcal{A}^{\alpha-2}_{s-1.5}\\
&=\frac{2\alpha}{3}(\gamma^{(1.5)})^{2\alpha/3-1}\Df_\mu\gamma^{(1.5)}
&\mod \mathcal{A}^{\alpha-2}_{s-1.5}.
\end{aligned}
$$
On the other hand, by Proposition \ref{PSSymbol},
$$
\begin{aligned}
\partial_\mu(\gamma^{(1.5)})^{2\alpha/3}
&=\frac{2\alpha}{3}g^{2\alpha/3-1}\kappa^\alpha\partial_\mu gf\left(\frac{b}{\kappa}\right)^{2\alpha/3}
+\frac{2\alpha}{3}g^{2\alpha/3}\kappa^\alpha f\left(\frac{b}{\kappa}\right)^{2\alpha/3-1}f'\left(\frac{b}{\kappa}\right)\partial_\mu b
&\mod \mathcal{A}^{\alpha-2}_{s-2.5}\\
&=\frac{2\alpha}{3}(\gamma^{(1.5)})^{2\alpha/3-1}\partial_\mu\gamma^{(1.5)}
&\mod \mathcal{A}^{\alpha-2}_{s-2.5}.
\end{aligned}
$$
Thus (\ref{GammaCommu}) in fact equals $0\mod\mathcal{A}^{\alpha-1}_{s-2.5}$. Consequently  $\big[T_{(\gamma^{(1.5)})^{2\alpha/3}},T_\gamma\big]$ is a para-differential operator of order $\leq\alpha$, not just $\alpha+0.5$.
\end{proof}

We can now state the energy estimate:
\begin{proposition}\label{Energy}
Let
$$
(\zeta,\phi)\in \big(C_TH^{s+0.5}_x\times C_TH^s_x\big)\cap\big(C_T^1H^{s-1}_x\times C_T^1H^{s-1.5}_x\big)
$$
be given, with norm in this space less than some small $R$. We define the good unknown $w^\sharp$ as in (\ref{ActualGUknown}), the symbols $\lambda,h,\mathfrak{b}^\sharp,\mathfrak{v}^\sharp$ as in Section \ref{5}. Suppose $f\in C_TH^{s+0.5}_x\times C_TH^s_x$. 
Then the \emph{linear} Cauchy problem 
\begin{equation}\label{CauchyLin}
\partial_t\left(\begin{matrix}
\eta^\sharp \\
v^\sharp
\end{matrix}\right)
=
\left(\begin{matrix}
-T_{\mathfrak{v}^\sharp}\cdot\nabla_{G_0} & T_\lambda v^\sharp \\
-T_{h} & -T_{\mathfrak{v}^\sharp}\cdot\nabla_{G_0}
\end{matrix}\right)\left(\begin{matrix}
\eta^\sharp \\
v^\sharp
\end{matrix}\right)
+f
\end{equation}
in the unknown $(\eta,v)$ admits a unique solution 
$$
(\eta,v)\in \big(C_TH^{s+0.5}_x\times C_TH^s_x\big)\cap\big(C_T^1H^{s-1}_x\times C_T^1H^{s-1.5}_x\big),
$$
satisfying the following energy estimate for any real number $\alpha\leq s$:
$$
\begin{aligned}
\|\eta(t)\|_{H^{\alpha+0.5}_x}
+\|v(t)\|_{H^{\alpha}_x}
\leq C_\alpha e^{C_\alpha Rt}\left(\|\eta(0)\|_{H^{\alpha+0.5}_x}
+\|v(0)\|_{H^{\alpha}_x}
+\int_0^t\big\|f(\tau);{H^{\alpha+0.5}_x\times  H^\alpha_x}\big\|d\tau\right).
\end{aligned}
$$
\end{proposition}
\begin{proof}
If we define symbols $p,q,\gamma,a,b,c$ corresponding to $(\zeta,\phi)$ as in Subsection \ref{Symmetrization} (especially (\ref{DtpDtq})), and set $u=(T_p\zeta^\sharp,T_qv^\sharp)$, then the procedure in Subsection \ref{Symmetrization} transforms (\ref{CauchyLin}) into the following equivalent form:
\begin{equation}\label{CauchySymm}
\begin{aligned}
\partial_tu
&=\left(\begin{matrix}
 & T_\gamma \\
-T_\gamma & 
\end{matrix}\right)u
-T_{\mathfrak{v}^\sharp}\cdot\nabla_{G_0}
u
+\left(\begin{matrix}
\left(T_{a+b}-\big[T_{\mathfrak{v}^\sharp}\cdot\nabla_{G_0},T_p\big]\right)T_p^{-1} \\
\left(T_c-\big[T_{\mathfrak{v}^\sharp}\cdot\nabla_{G_0},T_q\big]\right)T_q^{-1}
\end{matrix}\right)u
\\
&\quad
+\left(\begin{matrix}
 & (T_pT_\lambda-T_\gamma T_q)T_q^{-1} \\
(-T_qT_h+T_\gamma T_p)T_p^{-1} & 
\end{matrix}\right)u
+\left(\begin{matrix}
T_p & \\
 & T_q
\end{matrix}\right)f.
\end{aligned}
\end{equation}

For any $\alpha\leq s$, we simply apply the operator $\Gamma_\alpha:=T_{(\gamma^{(1.5)})^{2\alpha/3}}$ to (\ref{CauchySymm}). We first notice that
$$
\big\|\Gamma_\alpha u\big\|_{L^2_x}
\simeq
\|u\|_{H^\alpha_x},
$$
and the implicit constants in this equivalence of norms depend only on $\|\zeta\|_{H^{s+0.5}_x}$. In fact, $\gamma^{(1.5)}(x,l)^{2\alpha/3}$ is a symbol of order $\alpha$ and involves only up to first order derivative of $\zeta$. Furthermore, by the symbolic calculus formula, $\Gamma_\alpha\cdot\Gamma_{-\alpha}=\mathrm{Id}+\Op\Sigma^{-1}_{<1/2}$, so at least when $\|\zeta\|_{H^{s+0.5}_x}$ stays close to 0 (which is the case by our assumption), the inverse of the operator $\Gamma_\alpha$ is a para-differential operator of order $-\alpha$. This gives the equivalence of norms. The estimate on the implicit constants is a direct consequence of Theorem \ref{Compo2}.

By Proposition \ref{OrderS}, $\big[\Gamma_\alpha,T_{\gamma}\big]$ is a para-differential operator of order $\alpha$. The symbols in this commutator involve only up to second order derivatives of $\zeta$. We next notice that $\big[\Gamma_\alpha,T_{\mathfrak{v}^\sharp}\cdot\nabla_{G_0}\big]$ is also a para-differential operators of order $\leq \alpha$. This is because commuting with a vector field does not increase the order of an operator. The commutator of $\Gamma_\alpha$ with
$$
\left(T_{a+b}-\big[T_{\mathfrak{v}^\sharp}\cdot\nabla_{G_0},T_p\big]\right)T_p^{-1},
\quad
\left(T_c-\big[T_{\mathfrak{v}^\sharp}\cdot\nabla_{G_0},T_q\big]\right)T_q^{-1}
$$
and 
$$
(T_pT_\lambda-T_\gamma T_q)T_q^{-1},
\quad
(-T_qT_h+T_\gamma T_p)T_p^{-1}
$$
are all of order $\leq\alpha$, due to Lemma \ref{ABZ4-10} and (\ref{Symmetrizer}). Furthermore, Proposition \ref{PSSymbol} ensures that
$$
\partial_t\Gamma_\alpha
=\frac{2\alpha}{3} T_{(\gamma^{(1.5)})^{2\alpha/3-1}\partial_t\gamma^{(1.5)}}
+\text{lower order para-differential operator}.
$$
Since $\gamma^{(1.5)}$ involves first order derivative of $\zeta$, while by the assumption we have $\partial_t\zeta\in C_T^0H^{s-1}_x$, it follows that $\partial_t\gamma^{(1.5)}$ is a symbol with $H^{s-2}\subset C^{s-3}_*$ regularity in $x$, so $\partial_t\Gamma_\alpha$ is still a para-differential operator of order $\alpha$. The operator norms of these are all controlled linearly by $R$, as a consequence of Theorem \ref{Compo2}.

To summarize, for any real number $\alpha\leq s$, we have
\begin{equation}\label{CauchySymmalpha}
\partial_t\Gamma_\alpha u=\left(\begin{matrix}
 & T_\gamma \\
-T_\gamma & 
\end{matrix}\right)\Gamma_\alpha u
-T_{\mathfrak{v}^\sharp}\cdot\nabla_{G_0} \Gamma_\alpha u
+\Upsilon_\alpha u
+\Gamma_\alpha\left(\begin{matrix}
T_p & \\
 & T_q
\end{matrix}\right)f,
\end{equation}
where $\Upsilon_\alpha$ is a para-differential operator of order $\alpha$, with operator norm controlled linearly by $R$. We now find, using that $T_\gamma$ and $T_{\mathfrak{v}^\sharp}\cdot\nabla_{G_0}$ are both approximately anti-self adjoint, the following a priori inequality for a solution $u$ of (\ref{CauchySymm}): for any real number $\alpha\leq s$,
$$
\frac{d}{dt}\|\Gamma_\alpha u(t)\|_{L^2_x}^2
\lesssim R\|\Gamma_\alpha u(t)\|_{L^2_x}^2
+\big\|f(\tau);{H^{\alpha+0.5}_x\times  H^\alpha_x}\big\|\cdot\|\Gamma_\alpha u(t)\|_{L^2_x}.
$$
Using the Grönwall inequality and norm equaivalence $\big\|\Gamma_\alpha u\big\|_{L^2_x}
\simeq
\|u\|_{H^\alpha_x}$, this implies the \emph{a priori} energy inequalities for a solution $u$ of (\ref{CauchySymm}), provided that it does exist:
$$
\|u(t)\|_{H^\alpha_x}
\leq C_\alpha e^{C_\alpha Rt}\left(\|u(0)\|_{H^\alpha_x}+\int_0^t\big\|f(\tau);{H^{\alpha+0.5}_x\times  H^\alpha_x}\big\|d\tau\right).
$$
The regularity $\partial_tu\in C^0_TH^{s-1.5}_x$ follows from the equation itself.

We may then apply the standard duality argument for linear hyperbolic systems. These a priori energy estimates do imply well-posedness of the linear Cauchy problem (\ref{CauchySymm}). It is important that $\alpha$ are allowed to be negative. See for example Section 6.3 in H\"{o}rmander's textbook \cite{Hormander1997}. The estimates for $u$ are then converted back to estimates for $(\eta,v)$.
\end{proof}

\subsection{Banach Fixed Point Argument}
We now formulate a Banach fixed point problem. We shall fix a sufficiently small $R>0$ to measure the size of the solution. As we shall see shortly, the major issue in closing a contraction argument is that we need to control some operator norms under \emph{weaker} regularity assumptions than $(\zeta,\phi)\in C_TH_x^{s+0.5}\times C_TH_x^s$. We thus seek aid from Proposition \ref{T_aNegIndex} constantly.

\textbf{The Metric Space and the Map.} 
We will be dealing with the good unknown $(\zeta,w)$ instead of $(\zeta,\phi)$. We set
$$
\mathcal{G}:\left(\begin{matrix}
\zeta \\
\phi
\end{matrix}\right)
\to
\left(\begin{matrix}
\zeta \\
w
\end{matrix}\right)
$$
the assignment of $(\zeta,\phi)$ to the corresponding good unknown. We already know that $\mathcal{G}$ is a diffeomorphism in a neighbourhood of zero in $C_TH^{s+0.5}_x\times C_TH^s_x$. 

Fix $\bar{\mathfrak{X}}_{R}$ to be the set of all 
$$
(\zeta,w)\in \big(C_TH^{s+0.5}_x\times C_TH^s_x\big)\cap\big(C_T^1H^{s-1}_x\times C_T^1H^{s-1.5}_x\big)
$$ 
such that $(\zeta,w)$ has norm $\leq R$ in that space, and also $(\zeta(0),\phi(0))$ has norm $\leq A_sR$ in that space, for some $A_s\ll1$ to be specified. We equip $\bar{\mathfrak{X}}_{R}$ with a \emph{weaker} metric
$$
\begin{aligned}
d\big((\zeta_1,w_1),(\zeta_2,w_2)\big)
:=\|\zeta_1-\zeta_2\|_{C_TH^{s-1}_x}
+\|w_1-w_2\|_{C_TH^{s-1.5}_x}.
\end{aligned}
$$
By weak compactness of bounded closed convex sets in Hilbert spaces, $(\bar{\mathfrak{X}}_R,d)$ is a complete metric space.

Now let $(\zeta,w)\in\bar{\mathfrak{X}}_R$ be given, and $(\zeta,\phi)=\mathcal{G}^{-1}(\zeta,w)$. We define the symbols $\lambda,h,\mathfrak{b}^\sharp,\mathfrak{v}^\sharp$ as in Section \ref{5}, and the mappings $f_1,f_2$ as in (\ref{f1f2}). Let us consider the \emph{linear} Cauchy problem for the unknown $(\eta,v)\in C_TH^{s+0.5}_x\times C_TH^s_x$:
\begin{equation}\label{ParaLin'}
\partial_t\left(\begin{matrix}
\eta^\sharp \\
v^\sharp
\end{matrix}\right)
=
\left(\begin{matrix}
-T_{\mathfrak{v}^\sharp}\cdot\nabla_{G_0} & T_\lambda \\
-T_{h} & -T_{\mathfrak{v}^\sharp}\cdot\nabla_{G_0}
\end{matrix}\right)\left(\begin{matrix}
\eta^\sharp \\
v^\sharp
\end{matrix}\right)
+\left(\begin{matrix}
f_1(\zeta,\phi)\\
f_2(\zeta,\phi)
\end{matrix}\right).
\end{equation}
Using the operator norm estimates in Subsection \ref{Symmetrization}, together with the quadratic estiamte for $f_1,f_2$ in (\ref{f1f2}), Proposition \ref{Energy} makes sure that the Cauchy problem for this linear hyperbolic system can be uniquely solved with given initial data. Thus, (\ref{ParaLin'}) has a unique solution $(\eta,v)\in C_TH_x^{s+0.5}\times C_TH^s_x$ with initial value being $(\zeta(0),w(0))$, satisfying the energy estimates
\begin{equation}\label{Energy'}
\begin{aligned}
\|\eta(t)\|_{H^{s+0.5}_x}+\|v(t)\|_{H^s_x}
&\leq C_se^{Rt}(A_sR+R^2t),\\
\|\partial_t\eta(t)\|_{H^{s-1}_x}+\|\partial_tv(t)\|_{H^{s-1.5}_x}
&\leq C_se^{Rt}(A_sR+R^2t)+C_sR^2.
\end{aligned}
\end{equation}
We set the \emph{solution operator} $\mathcal{S}$ to be the unique solution $(\eta,v)$ of (\ref{ParaLin'}) with initial data being $(\zeta(0),\phi(0))$.

The map $\mathscr{F}$ on $\bar{\mathfrak{X}}_R$ under consideration is thus defined by
$$
\mathscr{F}\left(\begin{matrix}
\zeta \\
w 
\end{matrix}\right)
:=\mathcal{S}\mathcal{G}^{-1}\left(\begin{matrix}
\zeta \\
w 
\end{matrix}\right).
$$
When $R\simeq0$ and $T$ is suitably small, if $A_s$ is also suitably small, then for $(\zeta,w)\in\bar{\mathfrak{X}}_R$, (\ref{Energy'}) implies 
$$
\begin{aligned}
\big\|\mathscr{F}(\zeta,w);{C_TH^{s+0.5}_x\times C_TH^{s}_x}\big\|
&\leq C_se^{C_sT}(A_s+C_sTR)R
\ll R\\
\big\|\partial_t\mathscr{F}(\zeta,w);{C_TH^{s-1}_x\times C_TH^{s-1.5}_x}\big\|
&\leq C_se^{C_sT}(A_s+C_sTR)R+C_sR^2
\ll R
\end{aligned}
$$
Thus a suitable choice of $A_s$ and $T$ makes sure that $\mathscr{F}$ maps $\bar{\mathfrak{X}}_{R}$ to itself.

\textbf{Contraction Argument.}
Let us now consider two different $(\zeta,w),(\zeta_1,w_1)\in\bar{\mathfrak{X}}_R$. We use void subscript or subscript 1 to denote functions or symbols constructed out of $(\zeta,w),(\zeta_1,w_1)\in\bar{\mathfrak{X}}_R$ respectively. Write $(\eta_j,v_j)=\mathscr{F}(\zeta_j,w_j)$. We would like to prove that for small $R>0$, there holds
\begin{equation}\label{Contraction}
\|\eta-\eta_1\|_{C_TH^{s-1}_x}
+\|v-v_1\|_{C_TH^{s-1.5}_x}
\ll d\big((\zeta,w),(\zeta_1,w_1)\big),
\end{equation}
which ensures $\mathscr{F}:\bar{\mathfrak{X}}_R\to\bar{\mathfrak{X}}_R$ is a contraction.

Write for simplicity $U=\eta-\eta_1$, $V=v-v_1$. We compute, by the definition of solution operator, that $U(0)=0$, $V(0)=0$, and
\begin{equation}\label{UV}
\begin{aligned}
\partial_t\left(\begin{matrix}
U \\
V
\end{matrix}\right)
&=\left(\begin{matrix}
-T_{\mathfrak{v}^\sharp}\cdot\nabla_{G_0} & T_{\lambda} \\
-T_{h} & -T_{\mathfrak{v}^\sharp}\cdot\nabla_{G_0}
\end{matrix}\right)\left(\begin{matrix}
U \\
V
\end{matrix}\right)\\
&\quad
+\left(\begin{matrix}
T_{\mathfrak{v}^\sharp-\mathfrak{v}_1^\sharp}\cdot\nabla_{G_0} & -T_{\lambda-\lambda_1} \\
T_{h-h_1} & T_{\mathfrak{v}^\sharp-\mathfrak{v}_1^\sharp}\cdot\nabla_{G_0}
\end{matrix}\right)\left(\begin{matrix}
\eta \\
v
\end{matrix}\right)
+\left(\begin{matrix}
f_1(\zeta,w)-f_1(\zeta_1,w_1)\\
f_2(\zeta,w)-f_2(\zeta_1,w_1)
\end{matrix}\right).
\end{aligned}
\end{equation}
(\ref{UV}) is exactly in the form indicated by Proposition \ref{Energy}. We just need to show that the $H^{s-1}_x\times H^{s-1.5}_x$ norm of the last two terms of the right-hand-side in (\ref{UV}) is controlled linearly by $d\big((\zeta,w),(\zeta_1,w_1)\big)$. But this is exactly the content of Lemma 6.6, Corollary 6.7 and Lemma 6.8 of \cite{ABZ2011}, which concerns with exactly the same quantities. If we want to reproduce the proof, we just constantly use Proposition \ref{SymbolVeryRough} and Proposition \ref{T_aNegIndex}. For example, the first-order symbol $h^{(1)}-h_1^{(1)}$ contains up to second order derivative of $\zeta$ and $\zeta_1$, so 
$$
\big\|h^{(1)}(x,l)-h_1^{(1)}(x,l)\big\|_{C^{s-4}_*}
\leq Cl\|\zeta-\zeta_1\|_{H^{s-1}_x}.
$$
By Proposition \ref{T_aNegIndex}, this implies that the para-differential operator $T_{h^{(1)}-h_1^{(1)}}$ is of order 2, and
$$
\big\|T_{h^{(1)}-h_1^{(1)}}\eta\big\|_{H^{s-1.5}_x}
\lesssim \|\zeta-\zeta_1\|_{H^{s-1}_x}\|\eta\|_{H^{s+0.5}_x}
\lesssim Rd\big((\zeta,w),(\zeta_1,w_1)\big).
$$

The final step then becomes simple. Applying Proposition \ref{Energy} to (\ref{UV}) with $\alpha=s-1.5$, we find
$$
\|U\|_{C_TH^{s-1}_x}+\|V\|_{C_TH^{s-1.5}_x}
\leq C_se^{C_sRT}TRd\big((\zeta,w),(\zeta_1,w_1)\big).
$$
Thus for $T$ suitably small, the mapping $\mathscr{F}$ is a contraction from $\bar{\mathfrak{X}}_R$ to itself, hence has a unique fixed point. The fixed point of $\mathscr{F}$ is exactly the solution of the para-linearized system (\ref{EQPara}), thus the original (\ref{EQ}).

\begin{remark}
In \cite{ABZ2011}, the existence and uniqueness for solution of the symmetrized capillary-gravity water waves system is proved as follows: one first prove energy estimate for a smoothed system with some parameter $\varepsilon$ (e.g. artificial viscosity), and show that a solution sequence converges to a genuine solution of the original one as $\varepsilon\to0$ by a weak compactness argument, then prove uniqueness by considering the energy estimate for the difference of two solutions, which is just (\ref{UV}). The reason that we would rather formulate a Banach fixed point theorem is that the latter is more constructive, and is in fact equivalent to the iterative method used in most literature on quasi-linear hyperbolic systems, for example Chapter 6 of \cite{Hormander1997}.
\end{remark}

\section*{Acknowledgement}
The author would like to thank Professor Carlos Kenig for weekly discussion on this project, and Professor Gigliola Staffilani for constant support. The author benefits a lot from discussion with Professor Thomas Alazard, Jean-Marc Delort, Veronique Fischer, Isabelle Gallagher and David Jerison. Thanks also goes to the author's friend, Kai Xu, for comments on representation theory.

\bibliographystyle{alpha}
\bibliography{References}

\end{spacing}
\end{document}